\PassOptionsToPackage{usenames, dvipsnames}{xcolor}
\documentclass{aic}

\usepackage{ccicons}

\usepackage{hyperref}
\usepackage{latexsym,amsmath, nccmath}
\usepackage{enumerate}
\usepackage{amsfonts}
\usepackage{amssymb}
\usepackage{latexsym}
\usepackage{fixmath}
\usepackage{tikz}
\usetikzlibrary{decorations.markings}
\usepackage{caption}

\usepackage{graphicx}

\usepackage[T1]{fontenc}
\usepackage{fourier}
\usepackage{bbm}
\usepackage{xcolor}
\usepackage{fullpage}

\usepackage{array}
\usepackage{booktabs}
\numberwithin{equation}{section}

\newcommand{\bDelta}{\mathfrak d}
\newcommand{\bGamma}{\mathfrak g}
\newcommand{\vDelta}{\vec\Delta}
\newcommand{\vGamma}{\vec\Gamma}

\newcommand{\WP}{\mathrm{WP}}
\newcommand{\malpha}{\alpha_0}

\newcommand{\colf}{\textcolor{BrickRed}{$\frozen$}}
\newcommand{\colu}{\textcolor{ForestGreen}{$\unfrozen$}}
\newcommand{\cols}{\textcolor{Blue}{$\slush$}}

\newcommand{\Vp}{V_{\perma}}
\newcommand{\Vf}{V_{\frozen}}
\newcommand{\Vu}{V_{\unfrozen}}
\newcommand{\Cp}{C_{\perma}}

\newcommand{\Cu}{C_{\unfrozen}}

\newcommand{\field}{\FF_2}
\newcommand{\perma}{\mathtt{f}}
\newcommand{\frozen}{\mathtt{f}}
\newcommand{\unfrozen}{\mathtt{u}}
\newcommand{\fu}{\star}
\newcommand{\slush}{\mathtt{s}}

\newcommand\vAnp{\vA}
\newcommand\vAs{\vA_\slush}

\renewcommand{\vec}[1]{\boldsymbol{#1}}
\renewcommand{\subset}{\subseteq}
\renewcommand{\ln}{\log}

\newcommand{\avgR}{\bar R}

\newcommand\disteq{\,\sim\,}
\newcommand\cutm{\Delta_{\Box}}

\newcommand{\Vs}{V_{\slush}}
\newcommand{\Cs}{C_{\slush}}

\newcommand\vX{\vec X}
\newcommand\vZ{\vec Z}

\newcommand\vk{\vec k}
\newcommand\vy{\vec y}
\newcommand\vz{\vec z}
\newcommand\vx{\vec x}
\newcommand\vv{\vec v}
\newcommand\vxi{\vec\xi}
\newcommand\vr{\vec r}
\newcommand\valpha{\vec\alpha}
\newcommand\vt{\vec t}
\newcommand\vh{\vec h}
\newcommand\vi{\vec i}
\newcommand\vell{\vec{\ell}}

\newcommand\CPC{Combinatorics, Probability and Computing}
\newcommand{\TT}{\mathbb T}
\newcommand\vm{{\vec m}}
\newcommand\vs{{\vec s}}
\newcommand\DELTA{{\vec\Delta}}
\newcommand\nix{\,\cdot\,}
\newcommand\vA{\vec A}
\newcommand\vW{\vec W}
\newcommand\G{\vec G}
\newcommand\SIGMA{\vec\sigma}
\newcommand\vsigma{\vec\sigma}
\newcommand\vtau{\vec\tau}
\newcommand\fS{\mathfrak{S}}
\newcommand\fL{\mathfrak{L}}
\newcommand\fC{\mathfrak{C}}
\newcommand\fl{\mathfrak{l}}
\newcommand\fA{\mathfrak{A}}
\newcommand\fF{\mathfrak{F}}
\newcommand\cA{\mathcal{A}}
\newcommand\cB{\mathcal{B}}
\newcommand\cC{\mathcal{C}}
\newcommand\cD{\mathcal{D}}
\newcommand\cF{\mathcal{F}}
\newcommand\cG{\mathcal{G}}
\newcommand\cE{\mathcal{E}}
\newcommand\cU{\mathcal{U}}
\newcommand\cS{\mathcal{S}}
\newcommand\cT{\mathcal{T}}
\newcommand\cI{\mathcal{I}}
\newcommand\cK{\mathcal{K}}
\newcommand\cP{\mathcal{P}}
\newcommand\cX{\mathcal{X}}
\newcommand\cV{\mathcal{V}}
\newcommand\cW{\mathcal{W}}
\newcommand\cZ{\mathcal{Z}}
\newcommand\fm{\mathfrak{m}}
\newcommand\fV{\mathfrak{V}}
\newcommand\fZ{\mathfrak{Z}}
\newcommand\fX{\mathfrak{X}}
\newcommand\fE{\mathfrak{E}}
\newcommand\fp{\mathfrak{p}}
\newcommand\fH{\mathfrak{H}}
\def\cR{{\mathcal R}}
\newcommand\eul{\mathrm{e}}
\newcommand\eps{\varepsilon}
\newcommand\ZZ{\mathbb{Z}}
\newcommand\FF{\mathbb{F}}
\newcommand\NN{\mathbb{N}}
\newcommand\Var{\mathrm{Var}}
\newcommand\Erw{\mathbb{E}}
\newcommand{\vecone}{\vec{1}}
\newcommand{\Po}{{\rm Po}}
\newcommand{\Bin}{{\rm Bin}}
\newcommand\dTV{d_{\mathrm{TV}}}
\newcommand\bc[1]{\left({#1}\right)}
\newcommand\cbc[1]{\left\{{#1}\right\}}
\newcommand\sqbc[1]{\left[{#1}\right]}
\newcommand\bcfr[2]{\bc{\frac{#1}{#2}}}
\newcommand{\bck}[1]{\left\langle{#1}\right\rangle}
\newcommand\brk[1]{\left\lbrack{#1}\right\rbrack}
\newcommand\abs[1]{\left|{#1}\right|}
\newcommand\RR{\mathbb{R}}
\newcommand{\Whp}{W.h.p.}
\newcommand{\whp}{w.h.p.}

\newcommand{\Szemeredi}{Szemer\'edi}
\newcommand\pr{\mathbb{P}} 
\renewcommand\Pr{\pr} 
\newcommand\Lem{Lemma}
\newcommand\Prop{Proposition}
\newcommand\Thm{Theorem}
\newcommand\Def{Definition}
\newcommand\Cor{Corollary}
\newcommand\Sec{Section}

\newtheorem{definition}{Definition}[section]
\newtheorem{claim}[definition]{Claim}

\newtheorem{theorem}[definition]{Theorem}
\newtheorem{lemma}[definition]{Lemma}
\newtheorem{proposition}[definition]{Proposition}
\newtheorem{corollary}[definition]{Corollary}
\newtheorem{fact}[definition]{Fact}

\newtheoremstyle{case}{}{}{}{}{}{:}{ }{}
\theoremstyle{case}

\DeclareMathOperator{\nul}{nul}

\newcommand{\ceil}[1]{\left\lceil#1\right\rceil}

\newcommand\A{\vA}

\def\ex{{\mathbb E}}

\newcommand{\graphprop}{\xi}
\newcommand{\branchprob}{\zeta}
\newcommand{\plaingraphprop}{\underline{\xi}}
\newcommand{\plainbranchprob}{\underline{\zeta}}

\newcommand{\rvx}{\tilde \vx}
\newcommand{\opw}{\tilde w}

\usepackage[normalem]{ulem}

\newcommand{\Lu}{L}
\newcommand{\s}{z}
\newcommand{\x}{x}
\newcommand{\y}{y}
\newcommand{\vLu}{\vec{\Lu}}
\newcommand{\indunf}[1]{\epsilon_{#1}}
\newcommand{\newp}{q}
\newcommand{\newx}{\xi}
\newcommand{\stoptime}{T_0}

\newcommand{\Dev}[1]{\cD_{#1}}
\newcommand{\explicit}{\bar\omega}

\aicAUTHORdetails{%
  title = {The Sparse Parity Matrix}, 
  author={Amin~Coja-Oghlan, Oliver Cooley, Mihyun Kang, Joon~Lee and Jean Bernoulli Ravelomanana}
  
    plaintextauthor = {Amin~Coja-Oghlan, Oliver Cooley, Mihyun Kang, Joon~Lee,  Jean Bernoulli Ravelomanana},
  plaintexttitle = {The Sparse Parity Matrix}, 
  runningauthor = {A. Coja-Oghlan, O. Cooley, M. Kang, J. Lee, and J. B. Ravelomanana},
  copyrightauthor = {A. Coja-Oghlan, O. Cooley, M. Kang, J. Lee, and J. B. Ravelomanana},
}   

\aicEDITORdetails{%
   year={2023},
   number={5},
   received={14 July 2021},   
   revised={9 January 2023},    
   published={14 September 2023},  
   doi={10.19086/aic.2023.5},      
}   

\begin{document}
	
	\begin{frontmatter}[classification=text]
		
		\title{The Sparse Parity Matrix \titlefootnote{An extended abstract of this work appeared in the proceedings of the 2022 Annual ACM-SIAM Symposium on Discrete Algorithms (SODA)}} 
		
		\author[ACO]{Amin~Coja-Oghlan\thanks{Supported by DFG CO 646/4}}
		\author[OC]{Oliver Cooley\thanks{Supported by Austrian Science Fund (FWF): I3747}}
		\author[MK]{Mihyun Kang\thanks{Supported by Austrian Science Fund (FWF): I3747 and Friedrich Wilhelm Bessel research award of the Alexander von Humboldt Foundation (AUT 1204138 BES)}}
		\author[JL]{Joon~Lee}
		\author[JBR]{Jean Bernoulli Ravelomanana\thanks{Supported by DFG CO 646/4}}
		
		\begin{abstract}
			Let $\vA$ be an $n\times n$-matrix over $\field$ whose every entry equals $1$ with probability $d/n$ independently for a fixed $d>0$.
			Draw a vector $\vy$ randomly from the column space of $\vA$.
			It is a simple observation that the entries of a random solution $\vx$ to $\vA x=\vy$ are asymptotically pairwise independent, i.e., $\sum_{i<j}\Erw|\pr[\vx_i=s,\,\vx_j=t\mid\vA]-\pr[\vx_i=s\mid\vA]\pr[\vx_j=t\mid\vA]|=o(n^2)$ for $s,t\in\field$.
			But what can we say about the {\em overlap} of two random solutions $\vx,\vx'$, defined as $n^{-1}\sum_{i=1}^n\vecone\{\vx_i=\vx_i'\}$?
			We prove that for $d<\eul$ the overlap concentrates on a single deterministic value $\alpha_*(d)$.
			By contrast, for $d>\eul$ the overlap concentrates on a single value once we condition on the matrix $\vA$, while
			over the probability space of $\vA$
			its conditional expectation vacillates between two different values $\alpha_*(d)<\alpha^*(d)$, either of which occurs with probability $1/2+o(1)$.
			This bifurcated non-concentration result provides an instructive contribution to both the theory of random constraint satisfaction problems and of inference problems on random structures.
			\hfill
		\end{abstract}
	\end{frontmatter}

	\section{Introduction}\label{Sec_intro}
	
	\subsection{Motivation and background}
	Sharp thresholds are the hallmark of probabilistic combinatorics.
	The classic, of course, is the giant component threshold, below which the random graph decomposes into many tiny components but above which a unique giant emerges~\cite{ER}. 
	Its (normalised) size concentrates on a deterministic value.
	Similarly, once the edge probability crosses a certain threshold the random graph contains a Hamilton cycle \whp, which fails to be present below that threshold~\cite{KS}. 
	Monotone properties quite generally exhibit sharp thresholds~\cite{Ehud}.
	Only inside the critical windows of phase transitions are we accustomed to deviations from this zero/one behaviour~\cite{BB}.
	
	In this paper we investigate the simplest conceivable model of a sparse random matrix.
	There is one single parameter, the average number $d>0$ of non-zero entries per row.
	Specifically, we obtain the $n\times n$-matrix $\vA=\vA(n,p)$ over $\field$ by setting every entry to one with probability $p=(d/n)\wedge 1$ independently.
	We will show that,
	remarkably, this innocuous random matrix exhibits a critical behaviour, deviant from the usual zero--one law, for all $d$ outside a small interval.
	This result, which is the main focus of this paper, has ramifications for random constraint satisfaction and statistical inference. 
	
	To begin with constraint satisfaction (we will turn to inference in Section~\ref{Sec_intro_overlap}), consider a random vector $\vy$ from the column space of $\vA$.
	The random linear system $\vA x=\vy$ constitutes a random constraint satisfaction problem par excellence.
	Its space of solutions is a natural object of study.
	In fact, the problem is reminiscent of the intensely studied random $k$-XORSAT problem, where we ask for solutions to a Boolean formula whose clauses are XORs of $k$ random literals~\cite{AchlioptasMolloy,CDMM,DuboisMandler,Dietzfelbinger,Ibrahimi,MRTZ,PittelSorkin}.
	Random $k$-XORSAT is equivalent to a random linear system over $\FF_2$ whose every row contains precisely $k$ ones.
	
	The most prominent feature of random $k$-XORSAT is its sharp satisfiability threshold.
	Specifically, for any $k\geq3$ there exists a critical value of the number of clauses up to which the random $k$-XORSAT formula possesses a solution, while for higher number of clauses no solution exists \whp~\cite{Dietzfelbinger,DuboisMandler,PittelSorkin}.
	For $m \times n$ matrices over $\FF_2$, the first candidate for the satisfiability threshold is the point $m/n=1$ which is the obvious point after which the corresponding $\field$-matrix cannot have full row rank anymore because there are more rows than columns (additional constraints are liable to cause conflicts rather than only redundancies). However, it was shown in~\cite[Theorem 1.1]{PittelSorkin} that  the satisfiability threshold is strictly smaller than this obvious point. 
	Instead, for $k$-XORSAT, the satisfiability threshold coincides with the threshold where due to long-range effects a linear number of variables {\em freeze}, i.e., are forced to take the same value in all solutions.
	Clearly, once an extensive number of variables freeze, additional random constraints are apt to cause conflicts.
	
	The precise freezing threshold can be characterised in terms of the 2-core of the random hypergraph underlying the $k$-XORSAT formula.
	We recall that the 2-core is what remains after recursively deleting variables of degree at most one along with the constraint that binds them (if any).
	If the 2-core is non-empty, then its constraints are more tightly interlocked than those of the original problem, which, depending on the precise numbers, may cause freezing.
	Indeed, the precise number of frozen variables can be calculated by way of a message passing process called Warning Propagation~\cite{Ibrahimi,MM}.
	The number of frozen variables concentrates on a deterministic value that comes out in terms of a fixed point problem.
	Although the $k$-XORSAT problem is conceptually far simpler than, say, the $k$-SAT problem, freezing plays a pivotal role in basically all other random constraint satisfaction problems as well~\cite{Barriers,DSS3,pnas,MM,FrozenMolloy,MolloyRestrepo}.
	
	Surprisingly, our linear system $\vA x=\vy$ behaves totally differently as two competing combinatorial forces of exactly equal strength engage in a tug of war.
	As a result, for densities $d>\eul$ the fraction of frozen variables fails to concentrate on a single value.
	Instead, that number and, in effect, the geometry of the solution space vacillate between two very different scenarios that both materialise with asymptotically equal probability.
	In other words, the model perennially remains in a critical state for all $d>\eul$.
	Let us proceed to formulate the result precisely, and to understand how it comes about.
	
	\subsection{Frozen variables}\label{Sec_intro_frozen}
	One of the two forces resembles the emergence of the 2-core in random $k$-XORSAT.
	Indeed, we could run the process of peeling variables appearing in at most one equation of the linear system $\vA x=\vy$ as well.
	The size of the $2$-core and the total number of coordinates that would freeze if the entire 2-core were to freeze can be calculated.
	Specifically, let
	\begin{align}\label{eqphid}
		\phi_d&:[0,1]\to[0,1],&\alpha&\mapsto 1-\exp\bc{-d\exp\bc{-d(1-\alpha)}}
	\end{align}
	and let $\alpha^*=\alpha^*(d)$ be its {\em largest} fixed point. 
	According to the ``2-core heuristic'', the number of frozen coordinates $x_i$ comes to about $\alpha^*n$. 
	A proof that \whp\ precisely this many variables freeze (or actually a more general statement) has been posed as an exercise~\cite{MM},
	perhaps based on the fact that the dimension of the kernel projected onto only the variables of the $2$-core is very small,
		so intuitively one might imagine that very few of these variables are unfrozen.
	But as we shall see momentarily, this conclusion is erroneous.
	
	For on the other hand we could trace the number of variables that freeze because of unary equations.
	Indeed, because the number of ones in a row of $\vA$ has distribution $\Po(d)$, about $d\eul^{-d}n$ equations contain just one variable.
	Naturally, each such variable freezes.
	Substituting these frozen values into the other equations likely produces more equations of degree one, etc.
	Interestingly enough, the number of frozen variables that this ``unary equations heuristic'' predicts equals $\alpha_*n$, with $\alpha_*$ the {\em least} fixed point of $\phi_d$.
	While for $d<\eul$ there is a unique fixed point and thus $\alpha_*=\alpha^*$, for $d>\eul$ the two fixed points $\alpha_*,\alpha^*$ are distinct.
	Indeed, apart from $\alpha_*,\alpha^*$, which are stable fixed points, there occurs a third unstable fixed point $\alpha_*<\malpha<\alpha^*$; see Figure~\ref{Fig_alpha}.
	
	\begin{figure} \centering
		\includegraphics[height=41mm]{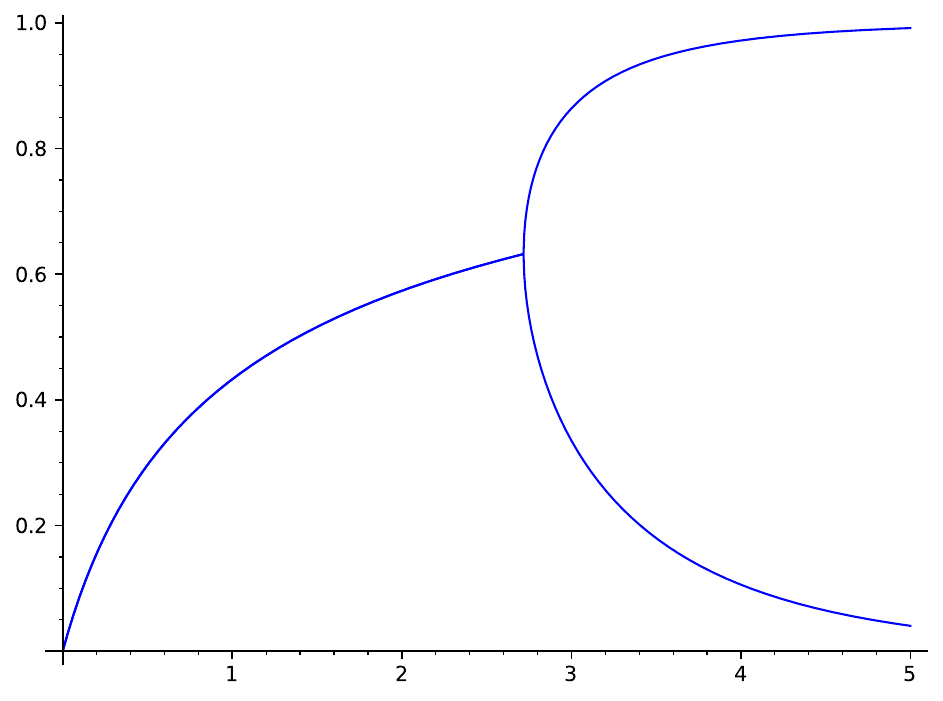}$\quad\quad\quad$
		\includegraphics[height=41mm]{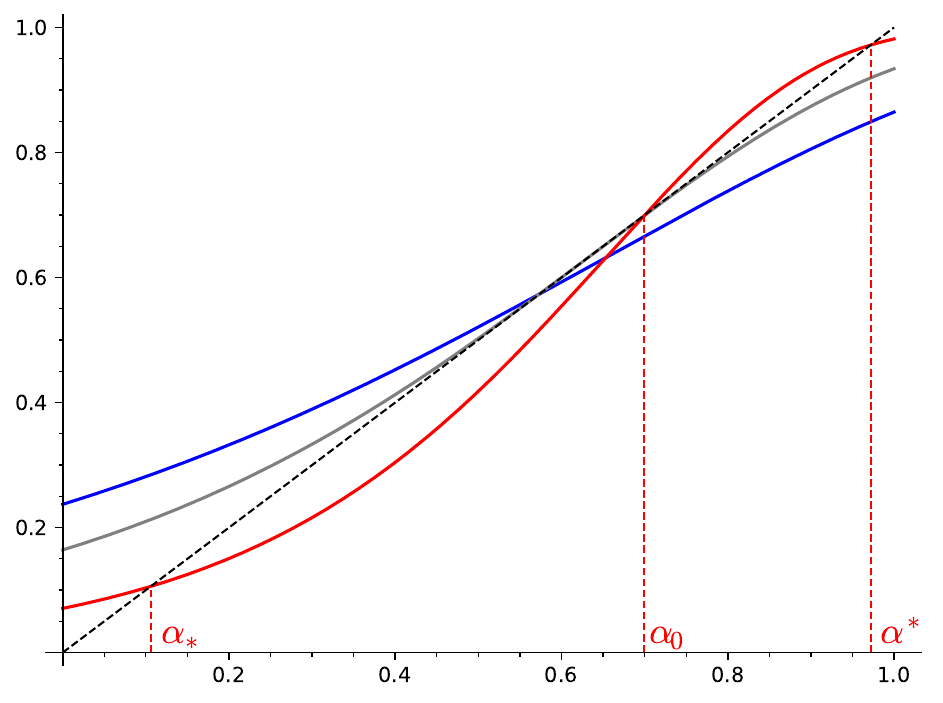}
		\caption{Left: the two fixed points $\alpha_*=\alpha_*(d)$ and $\alpha^*=\alpha^*(d)$ of $\phi_d$.
			Right: the function $\phi_d$ for $d=2.5$ (blue) possesses a unique fixed point, while for $d=3$ (red) there are two stable fixed points and an unstable one in between.}\label{Fig_alpha} 
	\end{figure}
	
	Which one of these heuristics provides the right answer? 
	To find out we could try to assess the total number of solutions that the linear system $\vA x=\vy$ should possess according to either prediction. 
	Indeed, \cite[\Thm~1.1]{Gao} yields an asymptotic formula for the number of solutions to a sparse random linear system in terms of a parameter $\alpha$ that, at least heuristically, should equal the fraction of frozen variables.
	For the random matrix $\vA$ the formula shows that, in probability,
	\begin{align}\label{eqGaoSimple}
		\lim_{n\to\infty}\frac{\nul\vA}{n}&=\max_{\alpha\in[0,1]}\Phi_d(\alpha),
		&&\mbox{where}&
		\Phi_d(\alpha)&=\exp\bc{-d\exp(-d(1-\alpha))}+(1+d(1-\alpha))\exp(-d(1-\alpha))-1
	\end{align}
	and where $\nul\vA$ denotes the nullity, i.e., the dimension of the kernel, of $\vA$.
	Hence, the correct answer should be the value $\alpha\in\{\alpha_*,\alpha^*\}$ that maximises $\Phi_d$.
	But it turns out that $\Phi_d(\alpha_*)=\Phi_d(\alpha^*)$ for all $d>0$.
	Accordingly, the main theorem shows that both predictions $\alpha_*$ and $\alpha^*$ are correct, or more precisely each of them is correct about half of the time. The set of solutions of $\vA x = \vy$ is a translation of the kernel of $\vA$,
		therefore to identify the frozen variables it suffices to identify the frozen variables in $\ker\vA$.
	Formally, let
	$$ f(\vA)=\abs{\cbc{i\in[n]:\forall x\in\ker\vA:x_i=0}}/n $$
	be the fraction of frozen variables.
	
	\begin{theorem}\label{Thm_square}
		\begin{enumerate}[(i)]
			\item For $d\leq\eul$ the function $\phi_d$ has a unique fixed point and
			\begin{align*}
				\lim_{n\to\infty}f(\vA)&=\alpha_*=\alpha^*
				&&\mbox{in probability.}
			\end{align*}
			\item For $d>\eul$ we have $\alpha_*<\alpha^*$ and for all $\eps>0$,
			\begin{align*}
				\lim_{n\to\infty}\pr\brk{|f(\vA)-\alpha_*|<\eps}&=
				\lim_{n\to\infty}\pr\brk{|f(\vA)-\alpha^*|<\eps}=\frac{1}{2}.
			\end{align*}	
		\end{enumerate}
	\end{theorem}
	
	\noindent
	Hence, the fraction of frozen variables fails to exhibit a zero--one behaviour for $d>\eul$.
	Instead, it shows a critical behaviour as one would normally associate only with the critical window of a phase transition.
	
	\subsection{The overlap}\label{Sec_intro_overlap}
	Apart from considering the linear system $\vA x=\vy$ as a random constraint satisfaction problem, the random linear system can also be viewed as an inference problem.
	Indeed, we can think of the vector $\vy$, which is chosen randomly from the column space of $\vA$, as actually resulting from multiplying $\vA$ with a uniformly random vector $\rvx\in\field^n$.
	Then $\vy=\vA\rvx$ turns into a noisy observation of the `ground truth' $\rvx$.
	Thus, it is natural to ask how well we can learn $\rvx$ given $\vA$ and $\vy$.
	
	These two viewpoints are actually equivalent because the posterior of $\rvx$ given $(\vA,\vy)$ is nothing but the uniform distribution on the set of solutions to the linear system $\vA x=\vy$.
	Hence,
	\begin{align}\label{eqposterior}
		\pr\brk{\rvx=x\mid\vA,\vy}&=\frac{\vecone\cbc{\vA x=\vy}}{|\ker\vA|}&&(x\in\field^n).
	\end{align}
	Therefore, the optimal inference algorithm just draws a random solution $\vx$ from among all solutions to the linear system.
	The number of bits that this algorithm recovers correctly reads
	\begin{align*}
		R(\vx,\rvx)=\frac1n\sum_{i=1}^n\vecone\cbc{\vx_i=\rvx_i}.
	\end{align*}
	Adopting mathematical physics jargon, we call $R(\vx,\rvx)$ the {\em overlap} of $\vx,\rvx$.
	Its average given $\vA,\vy$ boils down to
	\begin{align*}
		\avgR(\vA)&=\Erw[R(\vx,\rvx)\mid\vA,\vy]=\frac1{|\ker\vA|^2}\sum_{x,x'\in\ker\vA}R(x,x'),
	\end{align*}
	which is independent of $\vy$.
	
	Conventional wisdom in the statistical physics-inspired study of inference problems holds that the overlap concentrates on a single value given the `disorder', in our case $(\vA,\vy)$ (see~\cite{LF}). 
	This property is called {\em replica symmetry}.
	We will verify that replica symmetry holds for the random linear system \whp\
	Additionally, in all the random inference problems that have been studied over the past 20 years the overlap concentrates on a single value that does not depend on the disorder, except perhaps at a few critical values of the model parameters where phase transitions occur~\cite{JeanDmitry}.
	This enhanced property is called {\em strong replica symmetry}.
	A natural question is whether strong replica symmetry holds universally.
	It does not.
	As the next theorem shows, the random linear system with $d>\eul$ provides a counterexample: it is replica symmetric, but not strongly so.
	
	\begin{theorem}\label{Cor_square}
		\begin{enumerate}[(i)]
			\item If $d<\eul$ then $\lim_{n\to\infty}R(\vx,\rvx)=(1+\alpha_*)/2$ in probability.
			\item For all $d>\eul$ we have $\lim_{n\to\infty}\Erw\abs{R(\vx,\rvx)-\avgR(\vA)}=0$ while
			\begin{align*}
				\lim_{n\to\infty}\pr\brk{\abs{\avgR(\vAnp)-\frac{1+\alpha_*}2}<\eps}&=
				\lim_{n\to\infty}\pr\brk{\abs{\avgR(\vAnp)-\frac{1+\alpha^*}2}<\eps}=\frac{1}{2}&&\mbox{for any $\eps>0$.}
			\end{align*}
		\end{enumerate}
	\end{theorem}
	
	The first part of the theorem posits that for $d<\eul$ the overlap concentrates on the single value$(1+\alpha_*)/2$.
	In light of \Thm~\ref{Thm_square} this means that the optimal inference algorithm, while, unsurprisingly, capable of correctly recovering the frozen coordinates, is at a loss when it comes to the unfrozen ones.
	Indeed, we can get only about half the unfrozen coordinates right, no better than a random guess.
	
	The second part of the theorem is more interesting.
	While the random variable $R(\vx,\rvx)$ concentrates on the conditional expectation $\bar R(\vA)$ given $\vA,\vy$, the conditional expectation $\bar R(\vA)$ itself fails to concentrate on its mean $\Erw[\bar R(\vA)]$.
	Instead it vacillates between two different values $(1+\alpha_*)/2$ and $(1+\alpha^*)/2$, each of which occurs with asymptotically equal probability.
	In fact, this failure to concentrate does not just occur at a few isolated points, but throughout the entire regime $d>\eul$.
	This behaviour mirrors the bifurcation of the number of frozen variables from \Thm~\ref{Thm_square}.
	Moreover, as in the case $d<\eul$ the optimal inference algorithm does, of course, correctly recover the frozen variables, but cannot outperform a random guess on the unfrozen ones.
	
	We proceed to outline the key ideas behind the proofs of \Thm s~\ref{Thm_square} and~\ref{Cor_square}.
	Unsurprisingly, to prove the critical behaviour that these theorems assert we will need to conduct a rather subtle, accurate analysis of the random linear system and its space of solutions, far more so than one would normally have to undertake when aiming at a zero-one result.
	On the positive side the proofs reveal novel combinatorial insights that may have an impact on other random constraint satisfaction or inference problems as well.
	Let us thus survey the proof strategy.
	
	\subsection{Techniques}\label{Sec_intro_outline}
	
	The main result of the paper is that for $d>\eul$ the proportion $f(\vA)$ of frozen variables is asymptotically equal to either of the two stable fixed points $\alpha_*,\alpha^*$ of the function $\phi_d$ with probability $1/2+o(1)$ (see Figure~\ref{Fig_alpha}).
	Proving this statement takes three strikes.
	
	\begin{description}
		\item[FIX] $f(\vA)$ concentrates on the fixed points of $\phi_d$, either one of the two stable ones $\alpha_*,\alpha^*$ or the third unstable fixed point $\malpha$.
		\item[STAB] The unstable fixed point is an unlikely outcome.
		\item[EQ] The two stable fixed points are equally likely.
	\end{description}
	
	\subsubsection{Heuristics}
	Why are these three statements plausibly true?
	Let us begin with {\bf FIX}.
	The random matrix $\vA$ naturally induces a bipartite graph called the {\em Tanner graph} $G(\vA)$.
	Its vertex classes are {\em variable nodes} $v_1,\ldots,v_n$ representing the columns of $\vA$ and {\em check nodes} $a_1,\ldots,a_n$ representing the rows.
	There is an edge between $a_i$ and $v_j$ iff $\vA_{ij}=1$.
	The Tanner graph is distributed as a random bipartite graph with edge probability~$d/n$.
	As a consequence, its local structure is roughly that of a $\Po(d)$ Galton-Watson tree.
	
	Exploring the Tanner graph from a given variable node $v_i$, we may view $v_i$ as the root of such a tree.
	The grandchildren of $v_i$, i.e., the variable nodes at distance two, are essentially uniformly random. 
	Therefore, the grandchildren should each be frozen with probability $f(\vA)+o(1)$ and behave very nearly independently.
	Further, for the obvious algebraic reason the root $v_i$ itself is frozen iff it is parent to some check all of whose children are frozen.
	A few lines of calculations based on the Poisson tree structure then show that $v_i$ ought to be frozen with probability $\phi_d(f(\vA))$.
	But at the same time, since $v_i$ was itself chosen randomly, it is frozen with probability $f(\vA)$.
	Hence, we are led to expect that $f(\vA) = \phi_d(f(\vA))$.
	In other words, {\bf FIX} expresses that the local structure of $G(\vA)$ is given by a Poisson tree, and that freezing manifests itself locally.
	
	Apart from the two stable fixed points $\alpha_*,\alpha^*$, Figure~\ref{Fig_alpha} indicates that $\phi_d$ possesses an unstable fixed point $\malpha$ somewhere in between.
	How can we rule out that $f(\vA)$ will take this value?
	The nullity formula~\eqref{eqGaoSimple} suggests that $f(\vA)$ should be a {\em maximiser} of the function $\Phi_d(\alpha)$.
	But its maximisers are precisely the \emph{stable} fixed points $\alpha_*,\alpha^*$,
	while the unstable fixed point is where the function takes its local minimum.
	That is why {\bf STAB} appears plausible.
	However, we will see that this simplistic line of reasoning cannot be turned into a proof easily.
	
	Finally, coming to {\bf EQ}, we need to argue that for $d>\eul$ both stable fixed points are equally likely.
	To this end we employ the Warning Propagation (WP) message passing scheme, where messages are sent along the edges of the Tanner graph in either direction.
	The message from $v_j$ to $a_i$ is updated at each time step according to the messages that $v_j$ receives from its other neighbours, and similarly for the reverse message.
	WP does faithfully describe the local dynamics that cause freezing, but there remains a loose end: we must initialise messages somehow.
	
	Two obvious initialisations suggest themselves.
	First, if we initialise assuming everything to be unfrozen, then because of {\bf FIX} and the local geometry
	approximating a Galton-Watson branching tree, WP reduces to repeated application of the $\phi_d$ function starting from $0$.
	Since $\lim_{t\to \infty} \phi_d^{\circ t}(0) = \alpha_*$, WP then predicts $f(\vA)=\alpha_*$.
	Second, if we initialise assuming everything to be frozen, WP mimics iterating $\phi_d$ from $1$ and thus predicts $f(\vA) = \lim_{t\to \infty}\phi_d^{\circ t}(1) = \alpha^*$.
	
	So which initialisation is correct?
	Neither, unfortunately.
	We thus need a more nuanced version of WP, in which we describe messages and ultimately variables as ``frozen'', ``unfrozen'' and ``slush'', the last meaning uncertain.
	Initialising WP with either all messages frozen or all messages unfrozen still leads to the same results as before. 
	But initialising with all messages being ``slush'', WP predicts that approximately $\alpha_*n$ variables are frozen and $(1-\alpha^*)n$ variables are unfrozen, these predictions indeed being correct for almost all variables, while $(\alpha^*-\alpha_*)n$ variables remain slush,
	i.e., WP makes no prediction about them.
	Thus, there are actually {\em three} distinct categories.
	
	How does this help?
	Since $f(\vA)$ is concentrated around the stable fixed points $\alpha_*,\alpha^*$, we know that actually the slush portion must be either (almost) entirely frozen or unfrozen; it is impossible that, say, half the slush variables freeze.
	To figure out whether the slush freezes, consider the minor $\vA_\slush$ of $\vA$ induced on the corresponding variables and constraints.
	If this minor has fewer rows than columns, then the corresponding linear system is under-constrained.
	In effect, it is inconceivable that the slush freezes completely.
	On the other hand, if $\vA_\slush$ has more rows than columns, then by analogy to the random   $k$-XORSAT problem we expect that the slush freezes.
	Now, crucially, both the random matrix model $\vA$ and the WP message passing process are invariant under transposition of the matrix.
	Hence, $\vA_\slush$ should be over-constrained just as often as it is under-constrained.
	We are thus led to believe that  the slush freezes with probability about half, which explains the peculiar behaviour stated in the theorems.
	Once again, this simple reasoning, while plausible, cannot easily be converted into an actual proof.
	
	\subsubsection{Formalising the heuristics}\label{Sec_formalheuristic}
	Hence, how can we corroborate these heuristics rigorously?
	Concerning {\bf FIX}, consider the following game of ``thimblerig''.
	The opponent generates two random graphs independently: one is simply the Tanner graph $G_1 \sim G(\vA)$ of $\vA$, the other is an independent copy $G_2\sim G(\vA)$ of the Tanner graph, but with some random alterations.
	Specifically, the trickster generates a $\Po(d)$ branching tree of height two, embeds the root and its children onto isolated variable and check nodes respectively, and embeds the remaining leaves onto variables chosen uniformly at random.
	The opponent then presents you with the two graphs and asks you to determine which is which.
	It turns out that the changes are so well-disguised that you can do no better than a random guess.
	To compound your misery, having told you which is the perturbed graph, your opponent asks you to guess which variable is the root of the added tree.
	Again, the changes are so well-disguised that you can do no better than a random guess.
	Not content with winning twice, your opponent wishes to assert their complete dominance and performs the same trick again, this time adding not just one tree but a slowly growing number (of order $o(\sqrt n)$).
	For the third time, you can only resort to a random guess.
	
	The point of this game is to demonstrate that the root variables of the trees added
	behave identically to randomly chosen variables of the original graph.
	In particular, the proportion of variables which are frozen is distributed as $f(\vA)$.
	But we can also calculate this proportion in a different way: by considering whether the \emph{attachment} variables are frozen and tracking the effects down to the roots.
	This tells us that the proportion of frozen roots is $\phi_d(f(\vA)+o(1))$, {\em provided that} the newly added constraints do not dramatically shift the overall number of frozen variables due to long-range effects.
	To rule this out we use a delicate argument drawing on ideas from the study of random factor graph models and involving replica symmetry
	and the cut metric for discrete probability distributions from~\cite{Victor,KathrinWill,CKPZ,Will,BetheLattices}.
	
	Perhaps surprisingly, it takes quite an effort to verify the claim {\bf STAB} that $f(\vA)$ is not likely to be near the unstable fixed point.
	The proof employs a combinatorial construction that we call {\em covers}.
	A cover is basically a designation of the variable nodes, checks and edges of the Tanner graph that encodes which variables are frozen, and because of which constraints they freeze.
	We will then pursue a novel ``hammer and anvil'' strategy to rule out the unstable fixed point.
	On the one hand, we will show that if $f(\vA)$ is near $\malpha$, then the Tanner graph $G(\vA)$ must contain covers that each induce a cluster of solutions with about $\malpha$ frozen variables.
	On the other hand, we will use a moment computation to show that \whp\ the Tanner graph $G(\vA)$ only contains a sub-exponential number $\exp(o(n))$ of covers.
	Furthermore, another moment computation shows that \whp\ each of them only extends to about $2^{\Phi_d(\malpha)n}$ solutions to the linear system $\vA x=\vy$.
	As a consequence, if $f(\vA)$ is near $\malpha$, then the random linear system $\vA x=\vy$ would have far fewer solutions than provided by \eqref{eqGaoSimple}.
	Since the nullity of the random matrix is tightly concentrated, we conclude that the event $f(\vA)\sim\malpha$ is unlikely.
	The novelty of this argument, and the source of its technical intricacy, is the two-step cover--solution consideration: first we verify that the set of solutions actually decomposes into clusters encoded by ``covers''.
	Then we calculate the number of covers (corresponding to solution clusters), and finally we estimate the number of solutions inside each cluster.
	This two-level approach is necessary as a direct first moment calculation of the expected number of solutions with a given Hamming weight seems doomed to fail, at least for $d$ near the critical value $\eul$.
	
	Coming to {\bf EQ}, as indicated in the previous subsection, the ``slush'' portion of the matrix enjoys a symmetry property, in that it is also the slush portion of the transposed matrix.
	We will prove that, depending on the precise aspect ratio of the slush minor, the slush variables either do or do not freeze.
	But there is one subtlety: we need to to show that the number of rows and the number of columns are not {\em exactly} equal \whp\
	Indeed it is not hard to show that the both numbers have standard deviation $\Theta(\sqrt{n})$.
	Hence, if they were independent they would differ by $\Theta(\sqrt{n})$ \whp.
	But this independence is quite clearly not satisfied.
	Thus, we need to argue that at least they have non-trivial covariance.
	
	To show this, we perform a similar trick to the game of thimblerig: we show that the matrix can be randomly perturbed to decrease the number of slush columns, while preserving the number of slush rows.
	Furthermore, this can be achieved without an opponent being able to identify that a change has been made.
	Performing this trick carefully shows that it is unlikely that the slush portion of the matrix is approximately square.
	Symmetry then tells us that with probability asymptotically $1/2$ it has significantly more rows than columns, and also with probability asymptotically $1/2$ it has significantly more columns than rows.
	
	It remains to prove that these two cases are likely to lead to all slush variables being frozen, or all being unfrozen respectively.
	Unfortunately, a simple symmetry argument does not quite suffice.
	Instead we first prove that it is unlikely that there are significantly, say $\omega\gg1$, more slush variables than slush checks, but that almost all slush variables are frozen.
	The number of slush variables that remain unfrozen must certainly be at least $\omega$ due to elementary consideration of the nullity.
	We are thus left to exclude that the number is between $\omega$ and $\eps n$, which we establish by way of an expansion argument.
	
	We finally need to show that it is unlikely that there are significantly more slush checks (say $m_\slush$) than slush variables ($n_\slush$), but that these slush variables remain mostly unfrozen.
	Crucially, 
	using symmetry arguments
	we can indeed show that a ``typical'' kernel vector will set approximately half of the slush variables to $1$ and half to $0$.
	Of course there are approximately $2^{n_\slush}$ such vectors.
	On the other hand, imagine that a check with $k$ slush variable neighbours chooses these neighbours uniformly at random
	(this can be made formally correct 
	with some technical arguments).
	Then the probability that this check is satisfied by a vector of Hamming weight approximately $n_\slush/2$ is approximately $1/2$ (since, e.g., based on the values of the first $k-1$ neighbours, the last must be chosen from the correct class).
	Therefore the expected number of kernel vectors should be approximately $2^{n_\slush-m_\slush} = o(1)$.
	
	The problem with this basic calculation is that error terms occur which turn out to be too significant to ignore.
	These error terms ultimately come from check nodes of degree two in the slush minor.
	To deal with them, we employ a delicate percolation argument in which we contract check nodes of degree exactly two, since they just equalise their two adjacent variable nodes.
	Importantly, we can show that this process neither affects the number of kernel vectors nor the balance $m_\slush-n_\slush$.
	We can thus complete the moment calculation and show that the slush cannot have an excess of rows and still be entirely unfrozen.
	
	\subsection{Discussion}
	How do the techniques that we develop in this paper compare to previously known ones, and how can our techniques be extended to other problems?
	
	The general Warning Propagation message passing scheme captures the local effects of constraint satisfaction problems; for example, in the context of satisfiability WP boils down to Unit Clause Propagation~\cite{MM}.
	WP also yields the $k$-XORSAT threshold~\cite{Ibrahimi} as well as the freezing threshold in random graph colouring~\cite{FrozenMolloy}.
	In addition, WP can also be used to study structural graph properties such as the $k$-core~\cite{Skubch,Pittel}. 
	In all these examples, the ``correct'' initialisation from which to launch WP is obvious, and the proof that random variable of interest converges to the fixed point is based on a direct and straightforward combinatorial analysis.
	Indeed, the standard strategy is then a two-stage one: first, show that WP quickly converges to something close to the conjectured limit; and second, show that
	after this initial convergence, not much else will change~\cite{CLR_WP}.
	
	However, this usual technique is not enough for our purposes, essentially because of the $2$-point rather than $1$-point concentration of $f(\vA)$. 
	Naively one might imagine that WP will converge to one of the two fixed points, each with probability $1/2$. 
	But intriguingly, the dichotomy of the random variable $f(\vA)$ induces a dichotomy for WP in each \emph{instance} of $\vA$ -- WP hedges its bets, identifying the two possible answers, but is unable to tell which is actually correct.
	As such, we are left with the ``uncertain'' portion of the matrix (or its Tanner graph).
	
	To deal with this complication we enhance the WP message passing scheme to expressly identify the portion of the Tanner graph that may go either way.
	Along the way, we develop a versatile {\em indirect} method for proving convergence to {\em some} fixed point to replace the usual direct combinatorial argument.
	This technique is based on the thimblerig game that more or less justifies the WP heuristic in general.
	While the argument appears to be reasonably universal, it fails to identify precisely which fixed point is the correct one.
	As mentioned above, we follow WP up with a novel type of moment calculation based on covers to rule out the unstable fixed point.
	One could envisage a generalisation of this technique to other planted constraint satisfaction problems or, more generally, spin glass models.
	The place of the nullity formula \eqref{eqGaoSimple} would then have to be filled by a formula for the leading exponential order of the partition function.
	
	The thimblerig argument is enabled by the important observation that unfrozen variables, for the most part, behave more or less independently of each other and that the random variable $f(\vA)$ is fairly ``robust'' with respect to small numbers of local changes (see \Prop~\ref{Prop_sym}).
	We establish this robustness by way of a pinning argument, in which unary checks are added that freeze certain previously unfrozen variables, and we analyse the effect that this has on the kernel.
	The thimblerig argument is an extension of arguments used in the study of random factor graph models~\cite{Will,BetheLattices,Montanari}, where the pinning operation also plays a crucial role~\cite{Max,CKPZ}.
	
	Because the slush minor of the matrix displays a peculiar critical phenomenon, such as one would normally associate only with critical regimes around a phase transition, new techniques are required to study it.
	In particular, while it seems intuitively natural that the uncertain proportion is unfrozen if $n_\slush -m_\slush \ge \omega$ is large and positive, but frozen if it is large and negative, proving this formally requires some significant new ideas.
	In particular, to prove the first statement we introduce \emph{flippers}, induced subgraphs of the uncertain portion which could confound expectations by being frozen. 
	These flippers must satisfy various properties, and the proof consists of showing that large flippers (or more precisely, large unions of flippers) are unlikely due to expansion properties.
	This sort of expansion argument appears by no means restricted to the present problem.
	A related combinatorial structure appeared in the proof of limit theorems for cores of random graphs~\cite{Forging}.
	
	Proving the second statement involves a delicate moment calculation. 
	The modification involved in contracting the checks of degree~2, which are the reason that the naive version of the argument fails, is similar to the operation to construct the kernel of a graph from its $2$-core.
	This moment calculation is the single place where we make critical use of the fact that we are studying a problem whose variables range over a finite domain, viz.\ the field $\field$.
	
	What are potential generalisations?
	The random linear system $\vA x=\vy$ is one case of a class of constraint satisfaction problems known as \emph{uniquely extendable problems}~\cite{Connamacher}.
	Such problems are characterised by the property that if all but one of the variables appearing in a constraint are fixed, there is precisely one choice for the value of the remaining variable such that the constraint is satisfied. 
	Some of these problems are intractable, such as, for example, algebraic constraints with variables ranging over finite groups.
	It would be most interesting to see if and how the methods developed in this paper could be extended to uniquely extendable problems.
	Furthermore, since we study a critical phenomenon, namely the two-point concentration of the proportion of frozen variables, our ideas may help to understand the behaviour at the critical point of phase transitions of random constraint satisfaction problems.
	This type of question remains an essentially blank spot on the map.
	
	\subsection{Further related work}\label{Sec_intro_relatedwork}
	Perhaps surprisingly, apart from the article~\cite{Gao} that establishes a nullity formula for general sparse random matrices and in particular \eqref{eqGaoSimple}, there have been no prior studies of the random matrix $\vA(n,p)$.
	However, random $m\times n$-matrix over finite fields $\FF_q$ where every row contains an equal number $k\geq2$ of non-zero entries have been studied extensively.
	In the case $k=q=2$ this model is directly related to the giant component phase transition~\cite{Kolchin2,Kolchin}, because each row constrains two random entries to be equal.
	Moreover, we already saw that for $k\geq3$ and $q=2$ the model is equivalent to random $k$-XORSAT.
	Dubois and Mandler~\cite{DuboisMandler} computed the critical aspect ratio $m/n$ up to which such a matrix has full row rank for $k=3$.
	The result was subsequently extended to $k>3$~\cite{Dietzfelbinger,PittelSorkin}.
	Indeed, the threshold value of $m$ up to which the random matrix has full rank can be interpreted in terms of the Warning Propagation message passing scheme~\cite{CDMM}.
	Beyond its intrinsic interest as a basic model of a random constraint satisfaction problem~\cite{MM}, the random $k$-XORSAT model has found applications in hashing and data compression~\cite{Dietzfelbinger,Wainwright}.
	
	The asymptotic rank of random matrices with a fixed number $k$ of non-zero entries per row over finite fields has been computed independently via two different arguments by Ayre, Coja-Oghlan, Gao and M\"uller~\cite{Ayre} and Cooper, Frieze and Pegden~\cite{CFP}.
	Additionally, Miller and Cohen~\cite{MillerCohen} studied the rank of random matrices in which both the number of non-zero entries in each row and the number of non-zero entries in each column are fixed.
	However, they left out the critical case in which these two numbers are identical, which was solved recently by Huang~\cite{Huang}.
	Additionally, Bordenave, Lelarge and Salez~\cite{BLS} studied the rank over $\RR$
	of the adjacency matrix of sparse random graphs.
	Of course, a crucial difference between the random matrix model that we study here and the adjacency matrix of a random graph is that the latter is symmetric.
	
	A problem that appears to be inherently related to the binomial random matrix problem studied here is the matching problem on random bipartite graphs~\cite{BLS2}.
	It would be interesting to see if in some form the criticality observed in \Thm s~\ref{Thm_square} and~\ref{Cor_square} extends to the matching problem or, equivalently, the independent set problem on random bipartite graphs.
	The critical value $d=\eul$ appears to be related to the uniqueness of the Gibbs measure of the latter problem~\cite{Bandyopadhyay}.
	In the context of the matching problem, our function $\Phi_d(\alpha)$ appears (as $F(1-\alpha)$) in~\cite{BLS2},
	in particular in the appendix where a figure shows the emergence of the two global maxima above the threshold $d=\eul$.
	(In fact the discussion there is about the one-type graph $G(n,d/n)$ rather than the bipartite $G(n,n,d/n)$,
	which is the distribution of $G(\vA)$, but since the two graphs have the same local weak limit the more general results of~\cite{BLS2} show that
	the matching problem displays similar behaviour.) In some sense it is not surprising that the same function should arise in these two problems:
	the Warning propagation process to determine which variables are certainly frozen in essence mimics a one-sided version of the first stage of the Karp-Sipser
	algorithm in which leaves and their neighbours are removed. This removal results in a remaining ``core'', similar to our ``slush'',
	of minimum degree at least $2$. This is where we encounter our first fixed point of $\phi_d$ (or maximum of $\Phi_d$).
	For the matching problem, this first roadblock is easy to overcome: the core turns out to have an almost perfect matching \whp,
	which implies that it is always the same fixed point which
	gives the correct answer. By contrast, our situation is more delicate because the slush need not freeze.

	\section{Organisation}\label{Sec_proofmainresults}
	
	\noindent
	In this section, we state the intermediate results that lead up to the main theorems.
	We also detail where in the following sections the proofs of these intermediate results can be found.
	
	\subsection{The functions $\phi_d$ and $\Phi_d$}
	The formula~\eqref{eqGaoSimple} yields the approximate number of solutions to the linear system $\vA x=\vy$.
	We already discussed the combinatorial intuition behind the maximiser $\alpha$ in \eqref{eqGaoSimple}: we will prove that the function $\Phi_d$ attains its global maxima at the conceivable values of $f(\vA)$.
	However, the proof of \eqref{eqGaoSimple} in~\cite{Gao} falls short of already implying this fact as that proof strategy relies on a purely variational argument.
	For a start, we verify that the function $\phi_d$ actually has a unique fixed point for $d\leq\eul$ and two distinct stable fixed points for $d>\eul$, and that these fixed points coincide with the local maxima of $\Phi_d$.
	
	\begin{lemma}\label{Lemma_stability}
		For all $d>0, d\neq \eul$ the local maxima of $\Phi_{d}$ and the stable fixed points of $\phi_{d}$ coincide.
		For $d=\eul$ the local maximum of $\Phi_{\eul}$ coincides with the lone fixed point, simultaneously the inflection point of $\phi_{\eul}$. 
	\end{lemma}
	
	\noindent
	The proof of \Lem~\ref{Lemma_stability}, based on a bit of calculus, can be found in \Sec~\ref{Sec_Lemma_stability}.
	Additionally, for $d\leq\eul$ we define $\malpha=\alpha_*$, while for $d>\eul$ we let $\malpha$ be the minimiser of $\Phi_d$ on the interval $[\alpha_*,\alpha^*]$.
	The following lemma, which we prove in \Sec~\ref{Sec_Lemma_contract}, shows that the $t$-fold iteration $\phi_d^{\circ t}(x)$ converges to one of the stable fixed points, except if we start right at $x=\malpha$.
	
	\begin{lemma}\label{Lemma_contract}
		For any $d>0$ we have
		\begin{align*}
			\lim_{t\to\infty}\phi^{\circ t}_d(x)&=\alpha_* \quad\mbox{for any }x<[0,\malpha),&
			\lim_{t\to\infty}\phi^{\circ t}_d(x)&=\alpha^* \quad\mbox{for any }x\in(\malpha,1].
		\end{align*}
	\end{lemma}
	
	The fixed point characterisation of the maximisers of $\Phi_d$ enables us to show that the global maxima of~$\Phi_d$ occur precisely at $\alpha_*=\alpha_*(d),\alpha^*=\alpha^*(d)$, the smallest and the largest fixed points of $\phi_d$.
	
	\begin{proposition}\label{Prop_theta=1}
		\begin{enumerate}[(i)]
			\item\label{Prop_theta=1dle} If $d\leq\eul$ then $\phi_{d}$ has a unique fixed point, which is the unique global maximiser of~$\Phi_{d}$.
			\item\label{Prop_theta=1dge} If $d>\eul$ then the function $\phi_{d}$ has precisely two stable fixed points, namely $0<\alpha_*<\alpha^*<1$, and
			\begin{align*}
				\Phi_{d}(\alpha_*)=\Phi_{d}(\alpha^*)&>\Phi_{d}(\alpha)&&\mbox{ for all }\alpha\in[0,1]\setminus\{\alpha_*,\alpha^*\}.
			\end{align*}
			In addition, $\phi_d$ has its unique unstable fixed point at $\malpha$, which satisfies the equation
			\begin{align}\label{eqmalphaeq}
				1-\malpha=\exp(-d(1-\malpha)).
			\end{align}
		\end{enumerate}
	\end{proposition}
	
	\noindent
	Although both the functions $\phi_d,\Phi_d$ are explicit, the proof of \Prop~\ref{Prop_theta=1}, which can be found in \Sec~\ref{Sec_Prop_theta=1},  turns out to be mildly involved.
	
	\subsection{Warning Propagation}\label{Sec_WP}
	One of our principal tools is an enhanced version of the Warning Propagation message passing algorithm that identifies variables as frozen, unfrozen or slush.
	Specifically, we will see that WP identifies about $\alpha_*n$ coordinates as positively frozen and another  $(1-\alpha^*)n$ as likely unfrozen \whp\
	Because \Prop~\ref{Prop_theta=1} shows that $\alpha_*=\alpha^*$ for $d<\eul$, this already nearly suffices to establish the first part of \Thm~\ref{Thm_square}.
	By contrast, in the case $d>\eul$, where $\alpha_*<\alpha^*$, we need to conduct a more detailed investigation of the $(\alpha^*-\alpha_*+o(1))n$ coordinates that WP declares as slush.
	
	To introduce WP, for a given $m\times n$ matrix $A$ over $\FF_2$ we represent the matrix by its bipartite {\em Tanner graph} $G(A)$.
	One of its vertex classes $V(A)=V(G(A))=\{v_1,\ldots,v_n\}$ represents the columns of $A$;
	we refer to the $v_i$ as the {\em variable nodes}.
	The second vertex class $C(A)=C(G(A))=\{a_1,\ldots,a_m\}$ represents the rows of~$A$; 
	we refer to them as {\em check nodes}.
	There is an edge present between $a_i$ and $v_j$ iff $A_{ij}=1$.
	Let $E(A)$ denote the edge set of $G(A)$.
	Moreover, let $\partial u$ signify the set of neighbours of vertex $u\in V(A)\cup C(A)$.
	Further, let $\cF(A)$ be the set of frozen coordinates $i\in[n]$, i.e., coordinates such that $x_i=0$ for all $x\in\ker A$.
	By abuse of notation we identify $\cF(A)$ with the corresponding set $\{v_i:i\in\cF(A)\}$ of variable nodes.
	Also let $f(A)=|\cF(A)|/n$ be the fraction of frozen coordinates.
	Conversely, for a given Tanner graph $G$ we denote by $A(G)$ the adjacency matrix induced by $G$.
	
	Our enhanced WP algorithm associates a pair of $\{\frozen,\slush,\unfrozen\}$-valued messages with every edge of $G(A)$.
	Hence, let $\cW(A)$ be the set of all vectors 
	\begin{align*}
		w&=(w_{v\to a},w_{a\to v})_{v\in V(A), a\in C(A):a\in\partial v}&&\mbox{with entries }w_{v\to a},w_{a\to v}\in\{\frozen,\slush,\unfrozen\}.
	\end{align*}
	We define the operator $\WP_A:\cW(A)\to\cW(A)$, $w\mapsto \opw$, encoding one round of the message updates, by letting
	(for adjacent $v \in V(A)$ and $a\in C(A)$)
	
		\begin{align}\label{eqWP1}
			\opw_{v\to a}&=\begin{cases}
				\unfrozen&\mbox{ if $w_{b\to v}=\unfrozen$ for all $b\in\partial v\setminus\cbc a$},\\
				\frozen&\mbox{ if $w_{b\to v}=\frozen$ for some $b\in\partial v\setminus\cbc a$},\\
				\slush&\mbox{ otherwise,}
			\end{cases}&
			\opw_{a\to v}&=\begin{cases}
				\frozen&\mbox{ if $w_{y\to a}=\frozen$ for all $y\in \partial a\setminus\{v\}$},\\
				\unfrozen&\mbox{ if $w_{y\to a}=\unfrozen$ for some $y\in\partial a\setminus\cbc v$},\\
				\slush&\mbox{ otherwise}
			\end{cases}
	\end{align}%
	as illustrated in Figure~\ref{fig2:WPrules}. 
	Further, let $w(A,t)=\WP_A^t(\slush,\ldots,\slush)$ comprise the messages that result after $t$ iterations of $\WP_A$ launched from the all-$\slush$ message vector $w(A,0)$.
	Additionally, let $w(A)=\lim_{t\to\infty}w(A,t)$ be the fixed point to which $\WP_A$ converges;
	the (pointwise) limit always exists because $\WP_A$ only updates an $\slush$-message to a $\unfrozen$-message or to an  $\frozen$-message,
	while $\unfrozen$-messages and $\frozen$-messages will never change again.

	\begin{figure}\centering 
		\begin{tikzpicture}[decoration={ 
				markings,
				mark=at position 0.5 with {\arrow{>}}}
			] 
			\tikzstyle{var}=[circle,thick,draw,minimum size=6mm]
			\tikzstyle{check}=[rectangle,thick,draw=black,minimum size=5mm]
			\tikzstyle{unfrozen} = [postaction={decorate}, thick,green!75!black]
			\tikzstyle{frozen} = [postaction={decorate}, thick,red!75!black]
			\tikzstyle{slush} = [postaction={decorate}, thick,blue!90!black]
			\draw (1,3)  node[var] (vp0) {$v$};
			\draw (1,1.5) node[check] (ap0) {$a$};
			\draw (0,0) node[var] (vp1) {};
			\draw (1,0) node[var] (vp2) {};
			\draw (2,0) node[var] (vp3) {};
			\draw[frozen] (vp1)--node[left]{\colf}(ap0);
			\draw[frozen] (vp2)--node[left]{\colf}(ap0);
			\draw[frozen] (vp3)--node[left]{\colf}(ap0);
			\draw[frozen] (ap0) -- node[left] {\textcolor{BrickRed}{$\frozen$}} (vp0);
			\draw (4,3)  node[var] (up0) {$v$};
			\draw (4,1.5) node[check] (bp0) {$a$};
			\draw (3,0) node[var] (up1) {};
			\draw (4,0) node[var] (up2) {};
			\draw (5,0) node[var] (up3) {};
			\draw[slush] (up1) --node[left]{\cols}(bp0);
			\draw[unfrozen] (up2) --node[left]{\colu}(bp0);
			\draw[frozen] (up3) --node[left]{\colf} (bp0);
			\draw[unfrozen] (bp0) -- node[left] {\textcolor{ForestGreen}{$\unfrozen$}} (up0);
			\draw (7,3)  node[var] (wp0) {$v$};
			\draw (7,1.5) node[check] (cp0) {$a$};
			\draw (6,0) node[var] (wp1) {};
			\draw (7,0) node[var] (wp2) {};
			\draw (8,0) node[var] (wp3) {}; 
			\draw[slush] (wp1)-- node[left]{\cols}(cp0) {};
			\draw[frozen] (wp2)-- node[left]{\colf}(cp0) {};
			\draw[frozen] (wp3) --node[left]{\colf} (cp0);
			\draw[slush] (cp0) --  node[left] {\textcolor{Blue}{$\slush$}}  (wp0);
			\draw (1,-1.5)  node[check] (a0) {$a$};
			\draw (1,-3) node[var] (v0) {$v$};
			\draw (0,-4.5) node[check] (a1) {};
			\draw (1,-4.5) node[check] (a2) {};
			\draw (2,-4.5) node[check] (a3) {};
			\draw[unfrozen] (a1) -- node[left] {\colu} (v0) ; 
			\draw[unfrozen] (a2) -- node[left] {\colu} (v0);
			\draw[unfrozen] (a3) -- node[left] {\colu}(v0);
			\draw[unfrozen] (v0) -- node[left] {\colu} (a0);
			\draw (3,-4.5) node[check]  (b1) {};
			\draw (4,-4.5) node[check]  (b2) {};
			\draw (5,-4.5) node[check]  (b3) {};
			\draw (4,-3) node[var]  (u0) {$v$};
			\draw (4,-1.5) node[check]  (b0) {$a$};
			\draw[frozen] (b1) --node[left]{\colf} (u0);
			\draw[slush] (b2) -- node[left]{\cols} (u0);
			\draw[slush] (b3) -- node[left]{\cols} (u0);
			\draw[frozen] (u0) -- node[left] {\textcolor{BrickRed}{$\frozen$}} (b0);
			\draw (6,-4.5) node[check]  (c1) {};
			\draw (7,-4.5) node[check]  (c2) {};
			\draw (8,-4.5) node[check]  (c3) {};
			\draw (7,-3) node[var]  (w0) {$v$};
			\draw (7,-1.5) node[check]  (c0) {$a$};
			\draw[unfrozen] (c1) -- node[left]{\colu}(w0);
			\draw[slush] (c2)--node[left]{\cols} (w0);
			\draw[unfrozen] (c3)--node[left]{\colu}(w0);
			\draw[slush] (w0) -- node[left] {\textcolor{Blue}{$\slush$}} (c0);
		\end{tikzpicture}
		\caption{A local snapshot of the Warning Propagation rules. The check and variable nodes are represented by squares and circles respectively. } \label{fig2:WPrules}
	\end{figure}
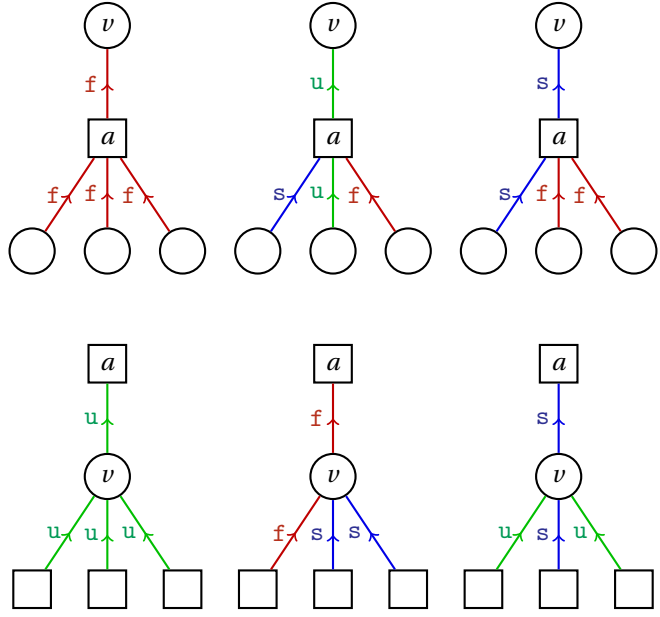
	
	What is the combinatorial idea behind WP?
	The intended semantics of the messages is that $\frozen$ stands for `frozen', $\unfrozen$ for `unfrozen' and $\slush$ for `slush'.
	Since we launch from all-$\slush$ messages, \eqref{eqWP1} shows that in the first round $\frozen$-messages only emanate from check nodes of degree one, where the `for all'-condition on the left of \eqref{eqWP1} is empty and therefore trivially satisfied.
	Hence, if a check node $a_i$ is adjacent to $v_j\in V(A)$ only, then $w_{a_i\to v_j}(A,1)=\frozen$.
	This message reflects that the $i$-th row of $A$, having only one single non-zero entry, fixes the $j$-th entry of every vector of $\ker A$ to zero.
	Further, turning to the updates of the variable-to-check messages, if $w_{a_i\to v_j}(A,1)=\frozen$, then $v_j$ signals its being forced to zero by passing to all its other neighbours $a_h\neq a_i$ the message  $w_{v_j\to a_h}(A,1)=\frozen$.
	Now suppose that check $a_i$ is adjacent to $v_h$ and  $w_{v_k\to a_i}(A,1)=\frozen$ for all $v_k\in\partial a_i\setminus\{v_h\}$.
	Thus, the $k$-th coordinate of every vector in $\ker A$ equals zero for all neighbours $v_k\neq v_h$ of $a_i$.
	Then the only way to satisfy the $i$-th row of $A$ is by setting the $h$-th coordinate to zero as well.
	Accordingly, \eqref{eqWP1} provides that $w_{a_i\to v_h}(A,2)=\frozen$, and so on.
	Hence, defining
	\begin{align}\label{eqLemma_WP1}
		\Vp(A)&=\cbc{v\in V(A):\exists a\in\partial v:w_{a\to v}(A)=\frozen},
		&\mbox{we see that}&&
		\Vp(A)\subset\cF(A).
	\end{align}
	
	The mechanics of the $\unfrozen$-messages is similar.
	In the first round any variable node $v_j$ of degree one, for which the `for all' condition on the right of \eqref{eqWP1} is trivially satisfied, starts to send out $\unfrozen$-messages.
	Subsequently, any check node $a_i$ with an adjacent variable $v_j$ of degree one will send a message $w_{a_i\to v_k}(A,2)=\unfrozen$ to all its other neighbours $v_k\neq v_j$.
	Further, if a variable node $v_j$ adjacent to a check $a_i$ receives $\unfrozen$-messages from all its other neighbours $a_h\neq a_i$, then $v_j$ sends a $\unfrozen$-message to $a_i$.
	Consequently, WP deems the variables
	\begin{equation}\label{eqLemma_WP2}
		\Vu(A)=\cbc{v\in V(A):\forall a\in\partial v:w_{a\to v}(A)=\unfrozen}
	\end{equation}
	unfrozen.
	But while \eqref{eqLemma_WP1} shows that WP's designation of the variables in the set $\Vp(A)$ as frozen is deterministically correct, matters are more subtle when it comes to the set $\Vu(A)$.
	For example, short cycles might lead  WP to include a variable in the set $\Vu(A)$ that is actually frozen.
	Yet the following lemma shows that on the random matrix $\vA$ such misclassifications are rare.
	
	\begin{proposition}\label{claim_littleintersection}
		For any $d>0$ we have	$ \vert \cF(\vA)\cap \Vu(\vA)  \vert=o(n)$ \whp\ 
	\end{proposition}
	
	\noindent
	Further, tracing WP on the random graph $G(\vA)$, we will establish the following bounds.
	
	\begin{proposition}\label{Prop_frozen}
		For any $d>0$ we have $|V_\frozen(\vA)|/n\geq\alpha_*+o(1)$ and  $|V_\unfrozen(\vA)|/n\geq 1-\alpha^*+o(1)$\whp{}
	\end{proposition}
	
	\noindent
	The proofs of \Prop~\ref{claim_littleintersection} and \Prop~\ref{Prop_frozen} can be found in \Sec~\ref{Sec_local}.
	
	\Prop s~\ref{claim_littleintersection} and~\ref{Prop_frozen} confine the number of frozen coordinates to the (scaled) interval between the two stable fixed points, $[\alpha_*n+o(n),\alpha^*n+o(n)]$.
	In particular, the first part of \Thm~\ref{Thm_square}, covering the regime $d<\eul$, is an immediate consequence of \Prop s~\ref{Prop_theta=1}, \ref{claim_littleintersection} and~\ref{Prop_frozen}.
	
	The case $d>\eul$ is not quite so simple since $\alpha_*<\alpha^*$ for $d>\eul$ by \Prop~\ref{Prop_theta=1}.
	Hence, \Prop~\ref{Prop_frozen} merely confines $f(\vAnp)$ to the interval $[ \alpha_*+o(1),\alpha^*+o(1) ]$.
	As we saw in \Sec~\ref{Sec_intro_outline}, a vital step is to prove that $f(\vAnp)$ is actually close to one of the boundary points $\alpha_*,\alpha^*$ \whp{}
	To prove this statement we need to take a closer look at the minor induced by the variables that are neither identified as frozen nor unfrozen, i.e., the variables in the slush.
	
	\subsection{The slush}\label{Sec_intro_slush}
	To this end we need to take a closer look at the inconclusive $\slush$-messages.
	Indeed, the $\slush$-messages naturally induce a minor $\vA_\slush{}$ of $\vA$.
	Generally, for a given matrix $A$ let
	\begin{align}\label{eqSlushIntro1}
		V_{\slush}(A)&=\cbc{v\in V(A):\bc{\forall a\in\partial v:w_{a\to v}(A)\neq\frozen},
			\abs{\cbc{a\in \partial v:w_{a\to v}(A)=\slush}}\geq2},\\
		C_{\slush}(A)&=\cbc{a\in C(A):\bc{\forall v\in\partial a:w_{v\to a}(A)\neq\unfrozen},
			\abs{\cbc{v\in\partial a:w_{v\to a}(A)=\slush}}\geq2}.\label{eqSlushIntro2}
	\end{align}
	Hence, none of the variable nodes in $V_{\slush}(A)$ receive any $\frozen$-messages, but each receives at least two $\slush$-messages.
	Analogously, the check nodes in $C_{\slush}(A)$ do not receive $\unfrozen$-messages but get at least two $\slush$-messages.
	Let $G_\slush(A)$ be the subgraph of $G(A)$ induced on $V_{\slush}(A)\cup C_\slush{}(A)$.
	Moreover, let $A_\slush{}$ be the minor of $A$ comprising the rows and columns whose corresponding variable or check nodes belong to $V_{\slush}(A)$ and $C_{\slush}(A)$, respectively.
	We observe that $G_\slush(A)$ admits an alternative construction that resembles the construction of the 2-core of a random hypergraph.
	Indeed, $G_\slush(A)$ results from $G(A)$ by repeating the following peeling operation:
	\begin{align}\label{eqslushpeel}
		\parbox{14cm}{while there is a variable or check node of degree at most one, remove that node along with its neighbour (if any).}
	\end{align}
	
	Formally, we remove all appropriate nodes and their neighbours simultaneously, although it is an elementary exercise
		to check that, if we only make such removals one-by-one, the order makes no difference to the end result.
	To determine the size and the degree distribution of $G_\slush(\vA)$ we employ a general result about
	WP-like message passing algorithms from~\cite{CLR_WP},
	which we will use in Section~\ref{sec:Prop_frozen} to prove the following result.
	
	\begin{proposition}\label{Prop_slush}
		Define
		\begin{align}\label{eqlambdanu}
			\lambda=\lambda(d)&=d(\alpha^*-\alpha_*),&\nu=\nu(d)=\exp(-d\alpha_*)-\exp(-d\alpha^*)(1+d(\alpha^*-\alpha_*)).
		\end{align}
		For any $d>\eul$ we have $\nu >0$ and
		\begin{align}\label{eqProp_slush1}
			\lim_{n\to\infty}|V_\slush\bc{\vA}|/n=\lim_{n\to\infty}|C_\slush\bc{\vA}|/n=\nu&&\mbox{in probability}.
		\end{align}
		Moreover, for any integer $\ell\geq2$ we have, in probability,
		\begin{align}\label{eqProp_slush2}
			\lim_{n\to\infty}\frac1n\sum_{x\in V_\slush{}(\vA)}\vecone\cbc{|\partial x\cap C_\slush{}(\vA)|=\ell}
			=\lim_{n\to\infty}\frac1n\sum_{a\in C_\slush{}(\vA)}\vecone\cbc{|\partial a\cap V_\slush{}(\vA)|=\ell}=\pr\brk{\Po_{\geq2}(\lambda)=\ell}.
		\end{align}
	\end{proposition}
	
	Based on what we have learned about Warning Propagation, we are now in a position to establish items {\bf FIX} and {\bf STAB} from the outline from \Sec~\ref{Sec_intro_outline}.
	
	\begin{proposition}\label{lem:3fp}
		For all $d\in(\eul,\infty)$ we have $\displaystyle\lim_{n\to\infty}\Erw\brk{\abs{f(\vAnp)-\alpha_*}\wedge\abs{f(\vAnp)-\malpha}\wedge\abs{f(\vAnp)-\alpha^*}}=0.$
	\end{proposition}

	\begin{proposition}\label{Prop_dd}
		For any $d\in(\eul,\infty)$ there exists $\eps>0$ such that $\displaystyle\lim_{n\to\infty}\pr\brk{\abs{f(\vAnp)-\malpha}<\eps}=0.$
	\end{proposition}
	
	\noindent
	The proofs of \Prop s~\ref{lem:3fp}--\ref{Prop_dd} can be found in \Sec s~\ref{sec:3fp} and~\ref{Sec_dd}.
	
	\subsection{The aspect ratio}
	We are left to deliver on item {\bf EQ} from the proof outline.
	Thus, we need to show that $f(\vAnp)$ takes either value $\alpha_*$, $\alpha^*$ with about equal probability if $d>\eul$.
	The description \eqref{eqslushpeel} of $G_\slush{}(\vA)$ in terms of the peeling process underscores that $|V_\slush(\vA)|$ and $|C_\slush(\vA)|$ are identically distributed.
	Yet in order to prove the second part of \Thm~\ref{Thm_square} we need to know that \whp\ the slush matrix is not close to square. 
	In \Sec~\ref{Sec_slush} we prove the following.
	
	\begin{proposition}\label{Prop_sym}
		For any $d_0>\eul$ there exists a function $\omega=\omega(n)\gg1$ such that for all $d>d_0$ we have
		\begin{align*}
			\lim_{n\to\infty}\pr\brk{|V_\slush\bc{\vA}|-|C_\slush\bc{\vA}|\geq\omega}=
			\lim_{n\to\infty}\pr\brk{|C_\slush\bc{\vA}|-|V_\slush\bc{\vA}|\geq\omega}=\frac12.
		\end{align*}
	\end{proposition}
	
	\subsection{Moments and expansion}
	
	Finally, to complete step {\bf EQ} in \Sec~\ref{Sec_Prop_expansion} we prove that  $f(\vA)$ is about equal to the higher possible value $\alpha^*$ if $\A_\slush$ has more rows than columns, and equal to the lower value $\alpha_*$ otherwise.
	
	\begin{proposition}\label{Prop_expansion}
		For any $d>\eul$, $\eps>0$, $\omega=\omega(n)\gg1$ we have
		\begin{align*}
			\limsup_{n\to\infty}\pr\brk{|f(\vA)-\alpha^*|<\eps,\ |V_\slush(\vA)|-|C_\slush(\vA)|\geq\omega}&=0,&
			\limsup_{n\to\infty}\pr\brk{|f(\vA)-\alpha_*|<\eps,\ |C_\slush(\vA)|-|V_\slush(\vA)|\geq\omega}&=0.
		\end{align*}
	\end{proposition}
	
	\noindent
	We now have all the ingredients in place to complete the proof of the main theorem.
	
	\begin{proof}[Proof of \Thm~\ref{Thm_square}]
		(i)	Suppose $d<\eul$.
		Combining \Prop s~\ref{claim_littleintersection} and~\ref{Prop_frozen} 
		with \eqref{eqLemma_WP1} and~\eqref{eqLemma_WP2}, we conclude that  $\alpha_*-o(1)\leq f(\vA)\leq\alpha^*+o(1)$ \whp{} 
		Since \Prop~\ref{Prop_theta=1} yields $\alpha_*=\alpha^*$, the assertion follows.
		
		(ii)	Fix $d>\eul$ and $\eps>0$ and let $\cE_*=\cbc{|f(\vA)-\alpha_*|<\eps}$, $\cE^*=\cbc{|f(\vA)-\alpha^*|<\eps}$.
		Then Propositions~\ref{lem:3fp} and~\ref{Prop_dd} imply that $\pr\brk{\cE_*\cup\cE^*}=1-o(1)$.
		Moreover, \Prop s~\ref{Prop_sym} and \ref{Prop_expansion} show that
		$\pr\brk{\cE_*}\leq1/2+o(1)$ and $\pr\brk{\cE^*}\leq1/2+o(1)$.
		Hence, we conclude that $\pr\brk{\cE_*},\pr\brk{\cE^*}=1/2+o(1)$, as claimed.
	\end{proof}

	\subsection{The overlap}\label{Sec_overview_overlap}
	\Thm~\ref{Cor_square} concerning the overlap follows relatively easily from \Thm~\ref{Thm_square}.
	The single additional ingredient that we need is the following statement that provides asymptotic independence of the first few coordinates $\vx_1,\ldots,\vx_\ell$ of a vector $\vec x$ drawn from the posterior distribution \eqref{eqposterior}.
	
	\begin{proposition}\label{Prop_pin}
		For every $\ell\geq1$ there exists $\gamma>0$ such that for all $d>0$ and all $\sigma\in\field^\ell$ we have
		\begin{align*}
			\lim_{n\to\infty}\Erw\brk{n^\gamma\abs{\pr\brk{\vx_1=\sigma_1,\ldots,\vx_\ell=\sigma_\ell\mid\vAnp}
					-\prod_{i=1}^\ell\pr\brk{\vx_i=\sigma_i\mid\vAnp}}}&=0.
		\end{align*}
	\end{proposition}
	
	\noindent
	\Prop~\ref{Prop_pin}, whose proof we defer to Appendix~\ref{Sec_pino}, is a corollary to a random perturbation of the matrix $\vAnp$ developed in~\cite{Ayre}.
	As an easy consequence of \Prop~\ref{Prop_pin} we obtain the following expression for the overlap.
	The proof can also be found in Appendix~\ref{Sec_pino}.
	
	\begin{corollary}\label{Cor_pin}
		For all $d>0$ we have
		$\lim_{n\to\infty}\Erw\abs{R(\vx,\vx')-(1+f(\vA))/2}=0.$
	\end{corollary}
	
	\begin{proof}[Proof of \Thm~\ref{Cor_square}]
		The assertion is an immediate consequence of \Thm~\ref{Thm_square} and \Cor~\ref{Cor_pin}.
	\end{proof}
	
	\subsection{Preliminaries and notation}\label{Sec_preliminaries}
	Throughout the paper, we use the standard Landau notations for asymptotic orders and all asymptotics are taken as $n\to\infty$. 
	Where asymptotics with respect to another additional parameter are needed, we indicate this fact by using an index.
	For example, $g(\eps,n)=o_\eps(1)$ means that
	$$
	\limsup_{\eps\to0}\limsup_{n\to\infty}|g(\eps,n)|=0.
	$$
	We call an $O(1)$ variable \emph{bounded} and an $\Omega(1)$ variable \emph{bounded away from zero}
		(although note in the latter case that it could take the value $0$ finitely often).
	We ignore floors and ceilings whenever they do not significantly affect the argument.
	
	Any $m\times n$ $\field$-matrix $A$ is perfectly represented by its Tanner graph $G(A)$, as defined
	in Section~\ref{Sec_WP}.
	We simply identify $A$ with its Tanner graph $G(A)$.
	For instance, we take the liberty of writing $f(G(A))$ instead of $f(A)$.
	Conversely, a bipartite graph $G$ with designated sets of check nodes $C(G)$ and variable nodes $V(G)$ induces a $|C(G)|\times|V(G)|$ matrix $A(G)$.
	Once again we tacitly identify $G$ with this matrix.
	Recall that for a Tanner graph $G$ and a node $z\in C(G)\cup V(G)$ we let $\partial z=\partial_Gz$ signify the set of neighbours.
	We further define $\partial^tz=\partial_G^tz$ to be the set of nodes at distance exactly $t$ from $z$.
	
	For a matrix $A$ we generally denote by $\cF(A)=\cF(G(A))$ the set of frozen variables.
	In addition, we let $\hat\cF(A)$ be the set of frozen checks, where a check node $a\in C(A)$ is called {\em frozen} if $\partial a\subset\cF(A)$.
	Let $\hat f(A)=|\hat\cF(A)|/|C(A)|$ be the fraction of frozen checks.
	
	For a matrix $A$ with Tanner graph $G$ and a node $z$ of $G$ let $d_A(z)=d_G(z)$ denote the degree of $z$.
	Furthermore, let $d_A=(d_A(z))_{z\in C(A)\cup V(A)}$ signify the degree sequence of $G(A)$ and let $d_{A,\slush}=(d_{A,\slush}(z))_{z\in C(A)\cup V(A)}$ encompass the degrees of the subgraph $G_\slush(A)$.
	Note that this sequence includes degrees of vertices
	which are not actually in $G_\slush(A)$, whose degree in $G_\slush(A)$ we define to be $0$.
	
	Returning to the random matrix $\vA$, let $\cG_\slush$ be a random multigraph drawn from the pairing model with degree distribution $d_{\vA,\slush}$,
	i.e., we generate a degree sequence according to the distribution $d_{\vA,\slush}$, assign to each vertex the appropriate number of half-edges
		according to this degree sequence, and construct a multigraph by selecting a uniformly random perfect matching on the set of half-edges.
	
	\begin{lemma}\label{Lemma_contig}
		The probability that $\cG_\slush$ is a simple graph is bounded away from $0$.
		Furthermore, conditioned on being simple the graph $\cG_\slush$ has exactly the same distribution as $G_\slush(\vA)$.
	\end{lemma}
	
	The proof of this lemma is a standard exercise, which we include in Appendix~\ref{Sec_contig} for completeness.
	We further need a routine estimate of the degree distribution of the random bipartite graph $G(\vA)$, whose proof can be found in
	Appendix~\ref{Sec_Lemma_tails}.
	
	\begin{lemma}\label{Lemma_tails}
		Let $d>0$.
		\Whp{} the random graph $G(\vA)$ satisfies
		\begin{align}\label{eqLemma_tails}
			\max_{v\in V(\vA)\cup C(\vA)}|\partial v|&\leq\log n,&
			\frac1n\sum_{x\in V(\vA)}\binom{|\partial x|}{\ell}&\leq(2d)^\ell&\mbox{for any integer $\ell\geq1$.}
		\end{align}
	\end{lemma}

	Throughout the paper all logarithms are to the base $\eul$.
	
	The {\em entropy} of a probability distribution $\mu$ on a finite set $\Omega\neq\emptyset$ is denoted by
	\begin{align*}
		H(\mu)&=-\sum_{\omega\in\Omega}\mu(\omega)\ln\mu(\omega).
	\end{align*}

	As a further important tool we need the cut metric for probability measures on $\field^n$.
	Following~\cite{KathrinWill}, we define the {\em cut distance} of two probability measures $\mu,\nu$ on $\FF_2^n$ as
	\begin{align}\label{eqCutMetric}
		\cutm(\mu,\nu)=\frac{1}{n}\min_{\substack{\vsigma\sim\mu\\\vtau\sim\nu}}\max_{\substack{U\subset\field^n\times\field^n\\I\subset[n]}}\abs{\sum_{i\in I}\pr\brk{(\vsigma,\vtau)\in U,\vsigma_i=1}-\pr\brk{(\vsigma,\vtau)\in U, \vtau_i=1}}.
	\end{align}
	In words, we first minimise over couplings $(\vsigma,\vtau)$ of the probability measures $\mu,\nu$.
	Then, given such a coupling an adversary points out the largest remaining discrepancy.
	Specifically, the adversary puts their finger on the event $U$ and the set of coordinates $I$ where the frequency of 1-entries in $\vsigma,\vtau$ differ as much as possible.
	
	The cut metric is indeed a (very weak) metric.
	We need to point out a few of its basic properties.
	For a probability measure $\mu$ on $\FF_2^n$ let $\vsigma^{(\mu)}$ denote a sample from $\mu$. 
	Moreover, let $\bar\mu$ be the product measure with the same marginals, i.e.,
	\begin{align*}
		\bar\mu(\sigma)&=\prod_{i=1}^n\mu\bc{\cbc{\vsigma^{(\mu)}_i=\sigma_i}}&&(\sigma\in\field^n).
	\end{align*}
	It is easy to see that upper bounds on the cut distance of $\mu,\nu$ carry over to $\bar\mu,\bar\nu$, i.e., 
	\begin{align}\label{eqMargContract}
		\cutm(\bar\mu,\bar\nu)\leq\cutm(\mu,\nu).
	\end{align}
	Moreover, upper bounds on the cut distance carry over to upper bounds on the marginal distributions, i.e.,
	\begin{align}\label{eqMargBound}
		\frac1n\sum_{i=1}^n\abs{\mu\bc{\cbc{\vsigma_i^{(\mu)}=1}}-\nu\bc{\cbc{\vsigma_i^{(\nu)}=1}}}\leq\cutm(\mu,\nu).
	\end{align}
	
	The distribution $\mu$ is {\em $\eps$-extremal} if $\cutm(\mu,\bar\mu)<\eps$.
	Furthermore, $\mu$ is {\em $\eps$-symmetric} if 
	\begin{align*}
		\sum_{1\leq i<j\leq n}\abs{\mu\bc{\cbc{\vsigma_i^{(\mu)}=\vsigma_j^{(\mu)}=1}}-\mu\bc{\cbc{\vsigma_i^{(\mu)}=1}}\mu\bc{\cbc{\vsigma_j^{(\mu)}=1}}}<\eps n^2.
	\end{align*}
	Hence, for most pairs $i,j$ the entries $\vsigma_i,\vsigma_j$ are about independent.
	More generally, $\mu$ is {\em $(\eps,\ell)$-symmetric} if 
	\begin{align*}
		\sum_{\tau\in\field^\ell}\sum_{1\leq i_1<\cdots<i_\ell\leq n}\abs{\mu\bc{\cbc{\forall j\leq\ell:\vsigma_{i_j}^{(\mu)}=\tau_j}}-\prod_{j=1}^\ell\mu\bc{\cbc{\vsigma_{i_j}^{(\mu)}=\tau_j}}}<\eps n^\ell.
	\end{align*}
	The following statement summarises a few results about the cut metric from~\cite{Victor,KathrinWill}.
	
	\begin{proposition}\label{Prop_cut}
		For any $\ell,\eps>0$ there exist $\delta>0$ and $n_0>0$ such that for all $n>n_0$ and all probability measures $\mu$ on $\field^n$ the following statements hold.
		\begin{enumerate}[(i)]
			\item If $\mu$ is $\delta$-extremal, then $\mu$ is $(\eps,\ell)$-symmetric.
			\item If $\mu$ is $\delta$-symmetric, then $\mu$ is $\eps$-extremal.
		\end{enumerate}
	\end{proposition}
	
	Furthermore, extremality of measures carries over to conditional measures so long as we do not condition on events that are too unlikely
	(see, e.g., \cite[Corollary~3.18]{BetheLattices}).
	More generally, we call two probability measures $\mu,\nu$ on $\field^n$ {\em mutually $c$-contiguous} if $c^{-1}\mu(\sigma)\leq\nu(\sigma)\leq c\mu(\sigma)$ for all $\sigma\in\field^n$.
	
	\begin{proposition}[{\cite[Lemma~3.17]{BetheLattices}}]\label{Prop_contig}
		For any $\eps>0$ there exist $\delta>0$ and $n_0>0$ such that for all $n>n_0$,
		for any $\delta$-extremal probability measure $\mu$ on $\field^n$ and
		for any probability measure $\nu$ on $\field^n$ such that $\mu,\nu$ are mutually $(1/\eps)$-contiguous, we have $\cutm(\mu,\nu)<\eps$.
	\end{proposition}
	
	Moreover, we need an elementary observation about the kernel of $\field$-matrices.
	
	\begin{fact}[{\cite[Lemma 2.3]{Ayre}}]\label{fact_ker}
		Let $A$ be an $m\times n$-matrix over $\FF_2$ and choose $\vxi=(\vxi_{1},\ldots,\vxi_{n})\in\ker A$ uniformly at random.
		Then for any $i,j\in[n]$ we have $\pr[\vxi_i=0]\in\{1/2,1\}$ and $\pr[\vxi_i=\vxi_j]\in\{1/2,1\}$.
	\end{fact}
	
	Finally, in Appendix~\ref{Sec_Lem_weighted_tails} we will prove the following auxiliary statement about weighted sums.
	
	\begin{lemma}\label{Lemma_weighted_tails}
		For any $c_0,c_1>0$ there exists $c_2>0$ such that for all $n>0$ the following is true.
		Suppose that $w:[n]\to(0,\infty)$ is any function such that
		\begin{align*}
			\frac1n\sum_{i=1}^n w_i\ \vecone\cbc{w_i>t}&\leq c_0\exp(-c_1t)&&\mbox{for any }t\geq1.
		\end{align*}
		Moreover, assume that $\cP=(P_1,\ldots,P_\ell)$ is any partition of $[n]$ into pairwise disjoint sets such that
		\begin{align*}
			\frac1n\sum_{j=1}^\ell |P_j|\ \vecone\cbc{|P_j|>t}&\leq c_0\exp(-c_1t)&&\mbox{for any }t\geq1.
		\end{align*}
		Then $\frac1n\sum_{j=1}^\ell\bc{\sum_{i\in P_j}w_i}^2\leq c_2.$
	\end{lemma}
	
	\section{Fixed points and local maxima}\label{Sec_basics}
	
	\noindent
	In this section we prove \Lem~\ref{Lemma_stability} and \Prop~\ref{Prop_theta=1}.
	We begin with a bit of trite calculus.
	
	\subsection{Getting started}\label{Sec_start}
	We introduce
	$D_{d}(\alpha)=\exp(-d(1-\alpha))$
	so that
	\begin{equation}\label{eqPhiDK}
		\begin{split}
			\phi_{d}(\alpha)&= 1-\exp(-d\exp(-d(1-\alpha)))=1-D_{d}(1-D_{d}(\alpha)),\\
			\Phi_{d}(\alpha)&=D_{d}(1-D_{d}(\alpha)) + (1+d(1-\alpha))D_{d}(\alpha)-1. 
		\end{split}
	\end{equation}
	We need two derivatives of $\Phi_d(\alpha)$ and $\phi_d(\alpha)$:
	\begin{align}
		\Phi_{d}'(\alpha) &= d^2 D_{d}(\alpha) \bc{\phi_{d}(\alpha)-\alpha},&
		\phi_d' (\alpha) &= d^2 D_d\bc{1-D_d(\alpha)} D_d(\alpha)\label{eqlittlephi_single}, \\
		\Phi_{d}''(\alpha)&= d^3 D_{d}(\alpha)  \bc{\phi_{d}(\alpha)-\alpha} + d^2 D_{d}(\alpha)  \bc{ \phi_{d}'(\alpha)- 1 } ,&
		\phi_d'' (\alpha) &= d^3 D_d\bc{1-D_d(\alpha)}D_d(\alpha)\bc{1-dD_d(\alpha)}.\label{eqlittlephi_double} 
	\end{align}
	Since $D_{d}(\alpha)$ is strictly increasing for all $d>0$, so is $\phi_{d}(\alpha)$ due to \eqref{eqPhiDK}.
	Thus,
	\begin{align}\label{eqphiinc}
		\phi'_{d}(\alpha)&>0&&\mbox{for all }\alpha\in[0,1].
	\end{align}
	Moreover, \eqref{eqlittlephi_double} shows that the sign of $\phi_d''$ only depends on the last term, denoted by 
	\begin{align} \label{psi:sign}
		\psi_{d, \text{sign}}(\alpha) = 1-dD_d(\alpha).
	\end{align}
	We denote the unique zero of $\psi_{d, \text{sign}}(\alpha)$ by $\bar{\alpha} = 1-\frac{\log d}{d}$.
	The following claim comes down to an exercise in calculus.
	
	\begin{claim}\label{claim:calculus1}\label{claim:calculus2}
		\begin{enumerate}[(i)]
			\item \label{item:d=e_1} $\bar{\alpha}$ is a fixed point of $\phi_{d}$ iff $d=\eul$.
			\item \label{item:phi''pos} $\phi_d''(0)>0$.
			\item \label{item:phi''zero} $\phi_d''(\alpha)$ has one zero at $\bar \alpha$ in the interval $[0,1]$ if $d\geq1$, none otherwise.
			\item \label{item:critical} $\phi_{\eul}'(\bar \alpha)=1$ and $\Phi_{\eul}''(\bar{\alpha}) = 0$.
			\item \label{item:onlyfp} $\bar{\alpha}$ is the only fixed point of $\phi_{\eul}(\alpha)$.
			\item \label{item:fixedstationary} The fixed points of $\phi_{d}$ coincide with the stationary points of $\Phi_{d}$.
			\item \label{item:Phi'signs} $\Phi_{d}'(0)>0>\Phi_{d}'(1)$.
			\item \label{item:fpexists} For any $d>0$ the function $\phi_{d}$ has at least one stable fixed point.
			\item \label{item:fewfp} For any $d>0$ the function $\phi_d$ has at most three fixed points, no more than two of which are stable.
			\item \label{item:phi'<1} For $d <\eul$, we have  $\phi_d'(\alpha) <1$ for all $\alpha \in [0,1]$.
			\item \label{item:uniquePhimax} For $d < \eul$, the function $\Phi_{d}$ attains a unique local maximiser $\alpha_{d}\in(0,1)$.  
			\item \label{item:secondfp} For $d>\eul$, if $\alpha\in(0,1)$ is a fixed point of $\phi_{d}$ then so is $\hat\alpha=1-\exp(-d(1-\alpha))\in(0,1)$.
		\end{enumerate}
	\end{claim}
	\begin{proof}
		\begin{enumerate}[(i)]
			\item Observe that $\phi_d(\bar\alpha) = 1-1/\eul$, which is a fixed point iff $\bar \alpha=1-\frac{\log d}{d}=1-\frac{1}{\eul}$,
			i.e., iff $d=\eul$.
			\item Recall that the sign of $\phi_d''(\alpha)$ is determined by the sign of $\psi_{d, \text{sign}}(\alpha)$, and we have
			$\psi_{d, \text{sign}} (0) =1 -~d\exp(-d) >0$ for all $d>0$.
			\item Since $\psi_{d, \text{sign}}'(\alpha) = -d^2\exp(-d(1-\alpha))<0$, we see that $\psi_{d, \text{sign}}$ is a decreasing function that has its unique zero at $\bar{\alpha}$.
			Furthermore, $\bar \alpha \le 1$ iff $d \ge 1$.
			\item By~\eqref{item:d=e_1}, when $d=\eul$ and $\alpha=\bar \alpha$, Equation~\eqref{eqlittlephi_double} reduces to
			$\Phi_{\eul}''(\bar \alpha) = \eul^2 D_{\eul}(\bar\alpha)  \bc{ \phi_{\eul}'(\bar\alpha)- 1 }$. Since also $D_e(\bar \alpha)=1/\eul$,
			by~\eqref{eqlittlephi_single} we have $\phi_{\eul}'(\bar \alpha)=1$,
			and therefore also $\Phi_{\eul}''(\bar \alpha)=0$.
			\item Due to~\eqref{item:d=e_1} $\bar \alpha$ is a fixed point, and $\phi_{\eul}'(\bar \alpha)=1$ by~\eqref{item:critical}.
			Since $\phi_{\eul}(\alpha)$ is convex for $\alpha<\bar{\alpha}$ and concave for $\alpha>\bar{\alpha}$ by \eqref{eqlittlephi_double},
			we deduce that $\phi_\eul(\alpha)>\alpha$ for $\alpha <\bar \alpha$ and $\phi_{\eul}(\alpha)<\alpha$ for $\alpha>\bar \alpha$,
			so $\bar{\alpha}$ is the unique fixed point of $\phi_{\eul}(\alpha)$.
			\item Since $d^2 D_{d}(\alpha)>0$, \eqref{eqlittlephi_single} implies that $\Phi_{d}'(\alpha) =0$ iff $\phi_{d}(\alpha)=\alpha$.
			\item This follows from~\eqref{eqlittlephi_single} since $\phi_{d}(0)>0$ and $\phi_{d}(1)<1$.
			\item 
			Since $\phi_d(0) >0$ and $\phi_d(1)<1$, and since $\phi_d$ is a continuous function,
			there must be at least one fixed point in $(0,1)$.
			Setting $\alpha_1:= \sup \{\alpha : \phi_d(\alpha)>\alpha\}$, we have that $\alpha_1$ is a fixed point by continuity.
			Furthermore, $\alpha_1$ is stable since there are points $\alpha<\alpha_1$ arbitrarily close to $\alpha_1$ for which $\phi_d(\alpha)>\alpha$, but also for any $\alpha>\alpha_1$ we have $\phi_d(\alpha)\le \alpha$, and therefore $\phi_d'(\alpha_1)\le 1$.
			\footnote{Note that at this point we could also have observed that $\Phi_d$ attains its maximum in the interior of $(0,1)$ and then applied
				Lemma~\ref{Lemma_stability} to prove the existence of a stable fixed point. This would be permissible since the proof of
				Lemma~\ref{Lemma_stability} only uses earlier points from this Claim and not~\eqref{item:fpexists} or any later points,
				therefore the argument is not a circular one.} 
			\item This is a consequence of \eqref{item:phi''zero}: between any two fixed points there must be a point with $\phi'(\alpha)=1$, and between any two such points there must be a point with $\phi''(\alpha)=0$; furthermore, between any two stable fixed points, there must be an unstable fixed point.
			\item If $d<1$, \eqref{item:phi''pos} and~\eqref{item:phi''zero} imply that $\phi''(\alpha) >0$ on $[0,1]$.
			Therefore $\phi_d'(\alpha) \leq \phi_d'(1) = d^2 \eul^{-d} < 1$.
			For $1 \leq d < \eul$, Property~\eqref{item:phi''zero} proves that for all $\alpha \in [0,1]$ we have
			$\phi_d'(\alpha)<\phi_d'(\bar{\alpha})= d/\eul <1$.
			\item
			By \eqref{item:fixedstationary}, we may consider stable fixed points of $\phi_d$ rather than maximisers of $\Phi_d$.
			The difference $h(\alpha):=\phi_d(\alpha) - \alpha$ is a decreasing function since $h'(\alpha) = \phi_d'(\alpha) -1<0$ by~\eqref{item:phi'<1}.
			Since $h(0)>0$ and $ h(1)<0$,  $h(\alpha)$ has only one zero for $d<\eul$. This shows that the stable fixed point from~\eqref{item:fpexists}
			is the unique fixed point.
			\item 
			Using $\alpha=\phi_d(\alpha)=1-\exp(-d\exp(-d(1-\alpha)))$, we obtain
			\begin{align*}
				\exp(-d(1-\hat\alpha))=\exp(-d\exp(-d(1-\alpha)))=1-\alpha=
				-\log(1-\hat\alpha)/d.
			\end{align*}
			Rearranging this inequality shows that $\hat\alpha=\phi_d(\hat\alpha)$.\qedhere
		\end{enumerate}
	\end{proof}
	
	\subsection{Proof of \Lem~\ref{Lemma_stability}}\label{Sec_Lemma_stability}
	At a fixed point $\alpha$ of $\phi_d$, \eqref{eqlittlephi_double} simplifies to 
	\begin{align}\label{eqLemma_stability_theta1}
		\Phi_{d}''(\alpha)& =  d^2 D_{d}(\alpha) \bc{ \phi_{d}'(\alpha) - 1 }. 
	\end{align}
	This shows $\Phi_{d}''(\alpha) < 0$ iff $ \phi_{d}'(\alpha) < 1$.   
	Hence, for $d>0, d\neq \eul$, \eqref{eqphiinc} and Claim~\ref{claim:calculus1}~\eqref{item:fixedstationary}
	imply that the stable fixed points of $\phi_{d}$ are precisely the local maximisers of $\Phi_{d}$.
	Claim~\ref{claim:calculus1}~\eqref{item:onlyfp} proves the second assertion in the case $d=\eul$.

	\subsection{Proof of \Prop~\ref{Prop_theta=1}}\label{Sec_Prop_theta=1}
	We make further observations on the existence and stability of fixed points of $\phi_{d}$.
	
	\begin{lemma}\label{Lemma_two}
		If $d>\eul$ then $\Phi_d$ attains its unique local minimum $\malpha\in[\alpha_*,\alpha^*]$ at the root of the \linebreak expression $1-\alpha-\exp(-d(1-\alpha))$.
	\end{lemma}
	\begin{proof}
		The concave function $\alpha\in[0,1]\mapsto1-\exp(-d(1-\alpha))$ has a unique fixed point $\beta=\beta(d)\in(0,1)$, which satisfies
		\begin{align*}
			\phi_{d}(\beta)&= 1-\exp(-d\exp(-d(1-\beta))=\beta,&
			\phi_{d}'(\beta)&=d^2\exp(-d(1-\beta))\exp(-d\exp(-d(1-\beta)))=d^2(1-\beta)^2.
		\end{align*}
		Hence, Claim~\ref{claim:calculus1}~\eqref{item:fixedstationary} and~\eqref{eqLemma_stability_theta1} yield
		\begin{align}\label{eqLemma_two1}
			\Phi_d'(\beta)&=0,&
			\Phi_d''(\beta)&=d^2\exp(-d(1-\beta))\bc{d^2(1-\beta)^2-1}.
		\end{align}
		In order to determine the sign of the last expression we differentiate with respect to $d$, keeping in mind that $\beta=\beta(d)$ is a function of $d$.
		Rearranging the fixed point equation $\beta=1-\exp(-d(1-\beta))$, we obtain $d=-(1-\beta)^{-1}\log(1-\beta)$.
		The inverse function theorem therefore yields
		\begin{align*}
			\frac{\partial\beta}{\partial d}&=\frac{(1-\beta)^2}{1-\log(1-\beta)}.
		\end{align*}
		Combining the chain rule with the fixed point equation $\beta=1-\exp(-d(1-\beta))$, we thus obtain
		\begin{align}\label{eqLemma_two2}
			\frac{\partial}{\partial d}d^2(1-\beta)^2&=2d(1-\beta)^2-2d^2(1-\beta)\frac{\partial\beta}{\partial d}=2d(1-\beta)^2\bc{1-\frac{d(1-\beta)}{1-\log(1-\beta)}}
			=\frac{2d(1-\beta)^2}{1+d(1-\beta)}>0.
		\end{align}
		As in Claim~\ref{claim:calculus1}, at $d=\eul$ we obtain $\beta=\bar{\alpha}=1-1/\eul$ and thus $d^2(1-\beta)^2=1$.
		Therefore, \eqref{eqLemma_two2} implies that $d^2(1-\beta)^2>1$ for all $d>\eul$, and thus  \eqref{eqLemma_two1} shows that $\Phi_d$ attains its local minimum $\malpha$ precisely at the point $\beta$.
		Finally, by Claim~\ref{claim:calculus1} (vi) and (ix) there is precisely one local minimum in the interval $[\alpha_*,\alpha^*]$.
	\end{proof}
	
	\begin{corollary}\label{Lemma_Lambert}
		For $d>\eul$ the function $\Phi_d$ attains its local maxima at the fixed points $0<\alpha_*<\alpha^*<1$ of $\phi_d$.
		Moreover, $\Phi_d(\alpha_*)=\Phi_d(\alpha^*)$.
	\end{corollary}
	\begin{proof}
		Since by Claim~\ref{claim:calculus1}~\eqref{item:Phi'signs} we have $\Phi_d'(0)>0>\Phi_d'(1)$, the existence of the local minimiser $\malpha\in(0,1)$ provided by \Lem~\ref{Lemma_two} implies that $\Phi_d$ has at least two local maximisers $0<\alpha_1<\malpha<\alpha_2<1$.
		\Lem~\ref{Lemma_stability} and Claim~\ref{claim:calculus1}~\eqref{item:fixedstationary} show that $\malpha,\alpha_1,\alpha_2$ are fixed points of $\phi_d$.
		Hence, Claim~\ref{claim:calculus2}~\eqref{item:fewfp} implies that $\alpha_1=\alpha_*$ is the smallest fixed point of $\phi_d$ and that $\alpha_2=\alpha^*>\alpha_*$ is the largest fixed point.
		Additionally, \Lem~\ref{Lemma_stability} and Claim~\ref{claim:calculus2}~\eqref{item:fewfp}
		imply that $\alpha_*,\alpha^*$ are the only local maximisers of $\Phi_d$.
		
		It remains to prove that $\Phi_d(\alpha_*)=\Phi_d(\alpha^*)$.
		Claim~\ref{claim:calculus2}~\eqref{item:secondfp} implies that 
		\begin{align*}
			\hat\alpha_*&=1-\exp(-d(1-\alpha_*))&\mbox{and}&&
			\hat\alpha^*&=1-\exp(-d(1-\alpha^*))
		\end{align*}
		are fixed points of $\phi_d$.
		Because $\malpha\neq \alpha_*,\alpha^*$ is the unique root of $1-\alpha-\exp(-d(1-\alpha))$, we conclude that $\hat\alpha_*=\alpha^*$ and $\hat\alpha^*=\alpha_*$.
		Hence,
		\begin{equation}\label{eqPreLambert}
			1-\alpha^*=\exp(-d(1-\alpha_*)),\quad
			1-\alpha_*=\exp(-d(1-\alpha^*)).
		\end{equation}
		Consequently,
		\begin{align}\label{eqLemma_Lambert1}
			(1-\alpha_*)\exp(-d(1-\alpha_*))&=(1-\alpha^*)\exp(-d(1-\alpha^*))\quad\mbox{and}\\
			1-\alpha_*+\exp(-d(1-\alpha_*))&=1-\alpha^*+\exp(-d(1-\alpha^*))
			\label{eqLemma_Lambert2}
		\end{align}
		Finally, combining \eqref{eqLemma_Lambert1}--\eqref{eqLemma_Lambert2} with the fixed point equations $\phi_d(\alpha_*)=\alpha_*$, $\phi_d(\alpha^*)=\alpha^*$, we obtain 
		\begin{align*}
			\Phi_d(\alpha^*)-\Phi_d(\alpha_*)&=\exp(-d\exp(-d(1-\alpha^*)))+\exp(-d(1-\alpha^*))\\
			&\quad-\brk{\exp(-d\exp(-d(1-\alpha_*)))+\exp(-d(1-\alpha_*))}\\
			&\quad+d\brk{(1-\alpha^*)\exp(-d(1-\alpha^*))-(1-\alpha_*)\exp(-d(1-\alpha_*))}\\
			&=1-\alpha^*+\exp(-d(1-\alpha^*))-\bc{1-\alpha_*+\exp(-d(1-\alpha_*))}=0,
		\end{align*} 
		thereby completing the proof.
	\end{proof}
	
	\begin{proof}[Proof of \Prop~\ref{Prop_theta=1}]
		The first part follows immediately from \Lem~\ref{Lemma_stability} and Claim~\ref{claim:calculus2}~\eqref{item:uniquePhimax}.
		The second assertion follows from \Lem~\ref{Lemma_stability}, \Lem~\ref{Lemma_two} and \Cor~\ref{Lemma_Lambert}.
	\end{proof}
	
	\subsection{Proof of \Lem~\ref{Lemma_contract}}\label{Sec_Lemma_contract}
	By a straightforward computation, we get that $\phi_d(0)>0$ and $ \phi_d(1)<1$ for all $d>0$. Moreover, $\phi_d(\alpha)$ is a continuously differentiable function.  
	For $d<\eul$, by Claim~\ref{claim:calculus1}~\eqref{item:fixedstationary}  and~\eqref{item:uniquePhimax} (or Proposition~\ref{Prop_theta=1}~\eqref{Prop_theta=1dle}) there is one fixed point $\alpha_*= \malpha=\alpha^*$.  
	This implies $\phi_d(\alpha)>\alpha$ for $\alpha\in[0,\alpha_*)$ and $\phi_d(\alpha)<\alpha$ for $\alpha\in(\alpha_*, 1]$ .  
	By Equation~\eqref{eqphiinc}, $\phi_d(\alpha)$ is strictly increasing so $\phi_d(\phi_d(\alpha)) >\phi_d(\alpha)$ for $\alpha\in[0, \alpha_*)$ and $\phi_d(\phi_d(\alpha))<\phi_d(\alpha)$ for $\alpha\in(\alpha_*,1]$. 
	By induction, for all $t>0$, $\phi_d^{\circ t}(\alpha)>\phi_d^{\circ t-1}(\alpha)$ for $\alpha\in[0, \alpha_*)$ and $\phi_d^{\circ t}(\alpha)<\phi_d^{\circ t-1}(\alpha)$ for $\alpha\in(\alpha_*, 1]$.  
	In addition, the fact that $\alpha_*$ is a fixed point of $\phi_{d}$ implies that $\alpha_*=\phi_d(\alpha_*)>\phi_d^{\circ t}(\alpha)$ for $\alpha\in[0,\alpha_*)$ and $\alpha_*=\phi_d(\alpha_*)<\phi_d^{\circ t}(\alpha)$ for $\alpha\in(\alpha_*, 1]$.  
	Hence, for $\alpha\in[0, \alpha_*)$, the sequence $\bc{\phi_d^t(\alpha)}_{t\geq 0}$ is monotonically increasing and bounded above
	by $\phi_d(\alpha_*)=\alpha_*$, and therefore
	$\lim_{t\to\infty}\phi_d^{\circ t}(\alpha)$ exists. Furthermore, since $\phi_d$ is continuous, this limit must be a fixed point of $\phi_d$.
	Since $\alpha_*$ is the smallest fixed point, we must have $\lim_{t\to\infty}\phi_d^{\circ t}(\alpha)=\alpha_*$, as required.
	Similarly, for $\alpha\in(\alpha_*, 1]$, the sequence $\bc{\phi_d^t(\alpha)}_{t\geq 0}$ is monotonically decreasing and bounded below thus $\lim_{t\to\infty}\phi_d^{\circ t}(\alpha) = \alpha^*$.  
	
	For $d>\eul$, by Proposition~\ref{Prop_theta=1}~\eqref{Prop_theta=1dge},  there are three fixed points, $\alpha_*<\malpha<\alpha^*$ where $\alpha_*, \alpha^*$ are stable fixed points and $\malpha$ is unstable.  For the intervals $[0,\alpha_*), (\alpha^*, 1]$, the proof is exactly the same as in the case $d<\eul$.  
	Similarly, $(\alpha_*, \malpha)$ comes down to the case of a monotonically decreasing sequence converging \linebreak to $\alpha_*$ while $(\malpha, \alpha^*)$ comes down to the case of a monotonically increasing sequence converging to $\alpha^*$.  
	
	\section{Tracing Warning Propagation}\label{Sec_local}
	
	\noindent
	In this section we will analyse the local structure of $G(\vA)$ together with WP messages, and show that locally the graph has a rather simple structure.
	For this argument we will make use of the results of~\cite{CLR_WP}.%
	\footnote{The article~\cite{CLR_WP} deals with a much more general multi-type random graph model,
			an arbitrary finite alphabet of possible messages and an arbitrary message update rule. It covers
			the bipartite random graph $G(\vA)$ and our instance of the WP algorithm as a special case, and checking that all of the necessary assumptions
			are satisfied is an easy exercise.}
	The study of WP messages will enable us to prove Propositions~\ref{claim_littleintersection}, \ref{Prop_frozen} and~\ref{Prop_slush}.
	
	\subsection{Message distributions and the local structure}
	To investigate the link between the local graph structure and the WP messages we need a few definitions.
	Let us first define a \emph{message distribution} to be a vector 
	\begin{align*}
		\vec q = \bc{\vec q^{(v)},\vec q^{(c)}} & \quad\mbox{with}\quad\vec q^{(v)} = \bc{q_\frozen^{(v)},q_\slush^{(v)},q_\unfrozen^{(v)}},\ \vec q^{(c)} = \bc{q_\frozen^{(c)},q_\slush^{(c)},q_\unfrozen^{(c)}}\in[0,1]^3\\
		&\quad\mbox{s.t.}\quad\sum_{s\in\{\frozen,\slush,\unfrozen\}}q_s^{(v)}=\sum_{s\in\{\frozen,\slush,\unfrozen\}}q_s^{(c)}=1.
	\end{align*}
	Intuitively, $q^{(v)},q^{(c)}$ model the probability distribution of an incoming message at a check/variable node,
	so for example $q_\frozen^{(v)}$ is the probability that an incoming message at a variable node is $\frozen$.
	
	Given a message distribution $\vec q$, we define $\Po(d\vec q)$ to be a distribution
	of half-edges with incoming messages. Specifically, at a variable node, this generates
	$\Po\bc{dq_\frozen^{(v)}}$ half-edges whose in-message is $\frozen$ and similarly
	(and independently) generates half-edges whose in-message is $\slush$ or $\unfrozen$.
	At a check node, the generation of half-edges with incoming messages is analogous.
	Let us define the message distribution
	\begin{align*}
		\vec q_* := \bc{\vec q_*^{(v)},\vec q_*^{(c)}}\quad\mbox{with}\quad\vec q_*^{(v)} = \bc{q_{*,\frozen}^{(v)},q_{*,\slush}^{(v)},q_{*,\unfrozen}^{(v)}} & := \bc{1-\alpha^*,\alpha^*-\alpha_*,\alpha_*},\\
		\vec q_*^{(c)} = \bc{q_{*,\frozen}^{(c)},q_{*,\slush}^{(c)},q_{*,\unfrozen}^{(c)}} & := \bc{\alpha_*,\alpha^*-\alpha_*,1-\alpha^*}.
	\end{align*}
	which is our conjectured limiting distribution of a randomly chosen message after the completion of WP,
	which motivates the following definitions.

	\begin{definition}\label{Def_T}
		We define branching processes $\cT,\hat\cT$ which will generate rooted trees decorated with messages along edges towards the root.
		\begin{enumerate}[(i)]
			\item 
			The root of the first process $\cT$ is a variable node $v_0$.
			The root spawns $\Po(d)$ children, which are check nodes.
			The edges from the children to the root independently carry an $\frozen$-message with probability $1-\alpha^*$, an $\slush$-message with probability $\alpha^*-\alpha_*$, and a $\unfrozen$-message with probability $\alpha_*$.
			The process then proceeds such that each check node spawns variable nodes and each variable node spawns check nodes as its offspring such that the messages sent from the children to their parents abide by the rules from Figure~\ref{fig2:WPrules}.
			To be precise, a check node $a$ that sends its parent message $\s\in\{\frozen,\slush,\unfrozen\}$ has offspring
			\begin{description}
				\item[$\s=\frozen$] $\Po(\alpha_*d)$ children that send an $\frozen$-message.
				\item[$\s=\slush$] $\Po(\alpha_*d)$ children that send an $\frozen$-message and $\Po_{\geq1}(d(\alpha^*-\alpha_*))$ children that each send an $\slush$-message.
				\item[$\s=\unfrozen$] $\Po(\alpha_*d)$ children that send an $\frozen$-message, $\Po(d(\alpha^*-\alpha_*))$ children that send an $\slush$-message and $\Po_{\geq1}(d(\alpha^*-\alpha_*))$ children that send a $\unfrozen$-message.
			\end{description}
			Analogously, a variable node $v$ that sends its parent message $\s\in\{\frozen,\slush,\unfrozen\}$ has offspring
			\begin{description}
				\item[$\s=\frozen$] $\Po_{\geq1}((1-\alpha_*)d)$ children that send an $\frozen$-message, $\Po(d(\alpha^*-\alpha_*))$ children that send an \linebreak $\slush$-message, and $\Po(d\alpha_*)$ children that send a $\unfrozen$-message.
				\item[$\s=\slush$] $\Po(\alpha_*d)$ children that each send a $\unfrozen$-message and $\Po_{\geq1}(d(\alpha^*-\alpha_*))$ children that send an $\slush$-message.
				\item[$\s=\unfrozen$] $\Po(\alpha_*d)$ children that send a $\unfrozen$-message.
			\end{description}
			\item 	The root of the second process $\hat\cT$ is a check node $a_0$.
			The root spawns $\Po(d)$ children, which are variable nodes.
			They independently send messages $\frozen,\slush,\unfrozen$ with probabilities $\alpha_*,\alpha^*-\alpha_*,1-\alpha^*$.
			Apart from the root, the nodes have offspring as under (i).
		\end{enumerate}
	\end{definition}

	Let us note that the processes $\cT,\hat\cT$,
	when truncated at depth $t\in \NN$, are equivalent to the following: generate a $2$-type branching tree up to depth $t$
	from the appropriate type of root in which each variable node has $\Po(d)$ children which are check nodes and vice versa,
	generate messages from the leaves at depth $t$ at random according to $\vec q_*$ and generate all other messages up
	the tree from these according to the WP update rule.
	The equivalence follows from the fact that $\vec q_*$ is a distributional fixed point of WP in a $\Po(d)$ branching tree. This means
		that because the messages in the second construction at the lowest level of the tree are distributed according to $\vec q_*$ independently,
		so are the messages one level higher (independently of each other, but certainly not independently
		of the messages coming up from below). Recursively, each level has messages up which are distributed according to $\vec q_*$ independently,
		including those sent up to the root, so the constructions are certainly equivalent down to depth~$1$.
		Now consider the messages coming up to a particular vertex $u$. With no additional information,
		these are distributed according to $\vec q_*$ independently, but if we already know the message that $u$ sends to its parent,
		we must condition on the children of $u$ sending messages which are compatible with this information. This is precisely
		what the distributions in Definition~\ref{Def_T} do.

	The following is the critical lemma describing the local structure.
	Given an integer $t$, let us define $\cS_t$ to be the set of messaged trees
	rooted at a variable node and with depth at most $t$, and similarly $\hat\cS_t$
	for trees rooted at a check node.
	For any $T \in \cS_t$ and matrix $A$, let us define 
	$$\graphprop_T(A) := \frac{1}{n}\sum_{v \in V(A)}\vec 1\cbc{\delta_{G(A)}^t v \cong T}$$
	to be the empirical fraction of variable nodes whose rooted depth $t$ neighbourhood $G(A)$ with edges towards the root annotated by the WP messages $(w_{a\to y}(A),w_{y\to a}(A))_{a,y}$ is isomorphic to $T$.
	For $\hat T \in \hat \cS_T$, the parameter $\graphprop_{\hat T}(A)$ is defined similarly.
	We also define $\branchprob_T := \Pr\sqbc{\cT_t \cong T}$ and $\branchprob_{\hat T} := \Pr\sqbc{\hat\cT_t \cong \hat T}$
	to be the probabilities that the appropriate branching process is isomorphic to $T$ or $\hat T$ respectively.
	
	\begin{lemma}\label{lem:locallimit}
		For any constant $t$ and any trees $T\in \cS_t$ and $\hat T \in \hat \cS_t$ we have
		$$
		\lim_{n\to\infty}|\graphprop_T(\vA) - \branchprob_T|=0 \qquad \mbox{and} \qquad \lim_{n\to\infty}|\graphprop_{\hat T}(\vA) - \branchprob_{\hat T}|=0 \qquad \mbox{in probability.}
		$$
	\end{lemma}
	
	In other words, picking a random vertex and looking at its local neighbourhood gives asymptotically the same result as generating a $\Po(d)$ branching tree to the appropriate depth and initialising messages at the leaves according to $\vec q_*$.
	
	Lemma~\ref{lem:locallimit} states that messages at the end of WP are roughly distributed according to $\vec q_*$, but of course, $\vec q_*$ does not reflect the messages at the start of the WP algorithm; our initialisation, in which all messages are $\slush$, is represented by the message distribution $\vec q_0 = (\vec q_0^{(v)},\vec q_0^{(c)}) := ((0,1,0),(0,1,0))$, but as the WP algorithm proceeds, the distribution will change, which motivates the following definition of an update function on message distributions.
	
	\begin{definition}
		Given a message distribution
		$\vec q = \bc{\bc{q_\frozen^{(v)},q_\slush^{(v)},q_\unfrozen^{(v)}},\bc{q_\frozen^{(c)},q_\slush^{(c)},q_\unfrozen^{(c)}}}$,
		let us define the message distribution $\varphi(\vec q)$ by setting
		\begin{align*}
			\varphi(\vec q)_\frozen^{(v)} & := \Pr\sqbc{\Po\bc{d\bc{q_\unfrozen^{(c)}+q_\slush^{(c)}}}=0},&
			\varphi(\vec q)_\frozen^{(c)} & := \Pr\sqbc{\Po\bc{dq_\frozen^{(v)}}\ge 1},\\
			\varphi(\vec q)_\slush^{(v)} & := \Pr\sqbc{\Po\bc{dq_\unfrozen^{(c)}}=0}\cdot \Pr\sqbc{\Po\bc{dq_\slush^{(c)}}\ge 1},&
			\varphi(\vec q)_\slush^{(c)} & := \Pr\sqbc{\Po\bc{dq_\frozen^{(v)}}=0}\cdot \Pr\sqbc{\Po\bc{dq_\slush^{(v)}}\ge 1},\\
			\varphi(\vec q)_\unfrozen^{(v)} & := \Pr\sqbc{\Po\bc{dq_\unfrozen^{(c)}}\ge 1},&
			\varphi(\vec q)_\unfrozen^{(c)} & := \Pr\sqbc{\Po\bc{d\bc{q_\frozen^{(v)}+q_\slush^{(v)}}}=0}.
		\end{align*}
		We further recursively define $\varphi^{\circ t}(\vec q) := \varphi\bc{\varphi^{\circ(t-1)}(\vec q)}$ for $t \ge 2$, and define $\varphi^*(\vec q) := \lim_{t\to \infty}\varphi^{\circ t}(\vec q)$ if this limit exists.
	\end{definition}
	
	The function $\varphi$ represents an update function of the WP message distributions in an idealised scenario,
	but it turns out that this idealised scenario is close to the truth.
	The following lemma is critical in order to be able to apply the results of~\cite{CLR_WP}.
	Let us define the total variation distance between message distributions $\vec q_1,\vec q_2$ by
	$$
	d_{TV} \bc{\vec q_1,\vec q_2} := d_{TV} \bc{\vec q_1^{(v)},\vec q_2^{(v)}} + d_{TV} \bc{\vec q_1^{(c)},\vec q_2^{(c)}}.
	$$
	\begin{lemma}\label{lem:stablelimit}
		We have $\varphi^*\bc{\vec q_0} = \vec q_*$. Furthermore, there exist $\eps,\delta>0$ such that for any message distribution~$\vec q$ which satisfies $d_{TV} \bc{\vec q,\vec q_*} \le \eps$,
		we have
		$d_{TV} \bc{\varphi\bc{\vec q},\vec q_*} \le (1-\delta)d_{TV} \bc{\vec q,\vec q_*}$.
	\end{lemma}
	
	In the language of~\cite{CLR_WP}, this lemma states that $\vec q_*$ is the \emph{stable limit} of $\vec q_0$.
	Before proving this lemma, we first show how to use it to prove Lemma~\ref{lem:locallimit}.
	We begin with the critical application of the main result of~\cite{CLR_WP}.
	Recall that $w(A,t)$ denote the messages after $t$ iterations of WP on the Tanner graph $G(A)$ with all initial messages set as $\slush$, and $w(A)=\lim_{t\to \infty} w(A,t)$.
	
	\begin{lemma}\label{lem:fewfurtherchanges}
		For any $d,\delta>0$ there exists $t_0 \in \NN$ such that \whp\ $w(\vec A)$ and $w(\vec A,t_0)$ are
		identical except on a set of at most $\delta n$ edges.
	\end{lemma}
	
	\begin{proof}
		Since $\vec q_*$ is the stable limit of $\vec q_0$, this follows directly from~\cite[Theorem~1.3]{CLR_WP}.
	\end{proof}
	
	Using Lemma~\ref{lem:fewfurtherchanges}, we can determine the local limit of the graph with final WP messages.
	
	\begin{proof}[Proof of Lemma~\ref{lem:locallimit}]
		Fix $t_0$ sufficiently large, and in particular large enough that Lemma~\ref{lem:fewfurtherchanges} can be applied,
		and furthermore $d_{TV}\bc{\varphi^{t_0}\bc{\vec q_0},\vec q_*} \ll \eps$.
		
		For any matrix $A$, we view $\graphprop(A)=\graphprop^{(t)}(A)$ and $\branchprob=\branchprob^{(t)}$
			as probability distributions over rooted, messaged graphs of
			depth at most $t$, and also define $\plaingraphprop(A)=\plaingraphprop^{(t)}(A)$ and $\plainbranchprob=\plainbranchprob^{(t)}$
			as the corresponding probability
			distributions over \emph{unmessaged} graphs
			(in which case the isomorphism classes are unions of isomorphism classes in the messaged version).

		The local structure of the graph $G(\vA)$ is that of a $\Po(d)$ branching tree,
			in the sense that \whp\
			$$
			d_{TV}\bc{\plaingraphprop^{(t+t_0)}(\vA),\plainbranchprob^{(t+t_0)}}=o(1).
			$$
			We can therefore couple the depth $t+t_0$ neighbourhoods in $G(\vA)$ with the outcomes of
			$n$ independent copies of the branching process $\cT_{t+t_0}$ in such a way that \whp\ they
			agree on all but $o(n)$ nodes.

		Next, we initialize WP messages in each of these neighbourhoods with all messages being $\slush$
			and run the WP process for $t_0$ rounds. In the branching trees, the messages at depth at most $t$ are distributed
			according to $\varphi^{t_0}\bc{\vec q_0}$, which is very close to $\vec q_*$. Therefore we can couple with an alternative
			way of generating the messages in the branching trees -- namely generating according to $\vec q_*$ at depth $t$
			and tracing the messages upwards -- in such a way that \whp\ the two alternatives agree except on a set of at
			most $\eps n/2$ nodes.

		Subsequently, Lemma~\ref{lem:fewfurtherchanges}
		implies that almost all messages at time $t_0$ are the final ones,
		and in particular there are at most $\eps n/2$ nodes whose depth $t_0$ neighbourhood will change.

		This shows that \whp\ $|\graphprop_T(\vA) - \branchprob_T|\le \eps/2$ and
			$|\graphprop_{\hat T}(\vA) - \branchprob_{\hat T}|\le \eps/2$.
			Since this holds for any constant $\eps$, the result follows.
	\end{proof}
	
	\begin{proof}[Proof of Lemma~\ref{lem:stablelimit}]
		For convenience, we will actually prove that $\vec q_*$ is the stable limit of $\vec q_0$ under the operator $\varphi^{\circ 2}$
		rather than $\varphi$ -- the advantage is that this $2$-step operator acts on the coordinates
		(corresponding to variable and check nodes) independently of each other.
		The analogous statement for $\varphi$ follows from that for $\varphi^{\circ 2}$ due to continuity.
		
		Furthermore, by symmetry we may prove the appropriate statements just for the first coordinate, i.e., for $\vec q_*^{(v)}$ -- 
		the corresponding proof for $\vec q_*^{(c)}$ is essentially identical.
		
		As a final reduction, let us observe that since for any message distribution we have $q_\frozen^{(v)}+q_\slush^{(v)} + q_\unfrozen^{(v)} = 1 $, it is sufficient to consider just two of the three coordinates. In this case it will be most convenient
		to consider $q_\frozen^{(v)}$ and $q_\unfrozen^{(v)}$, so let us restate what we are aiming to prove.
		
		Consider the operator $\tilde \varphi:[0,1]^2 \to [0,1]^2$ defined by
		$
		\tilde \varphi (x_1,x_2) := \bc{\tilde\varphi_1(x_1),\tilde\varphi_2(x_2)},
		$
		where
		\begin{align*}
			\tilde\varphi_1(x_1) & :=\exp\bc{-d\exp\bc{-d x_1}},&
			\tilde\varphi_2(x_2) & := 1- \exp\bc{-d \exp \bc{-d\bc{1-x_2}}}.
		\end{align*}
		This corresponds precisely to the action of $\varphi^{\circ 2}$ on $\bc{q_\frozen^{(v)},q_\unfrozen^{(v)}}$.
		Thus our goal is to prove that $(1-\alpha^*,\alpha_*)$ is the stable limit of $(0,0)$ under $\tilde \varphi$.
		
		Now observe that $\tilde \varphi_1(x_1) = 1-\phi_d(1-x_1)$ and recall that $\phi_d$ was
		defined in \eqref{eqphid}.
		By Lemma~\ref{Lemma_contract} and Proposition~\ref{Prop_theta=1}, $\phi_d$ is a contraction on $[\alpha^\ast,1]$ with unique fixed point $\alpha^\ast$,
		and so correspondingly $\tilde \varphi_1$ is a contraction on $[0,1-\alpha^*]$ with
		unique fixed point $1-\alpha^*$.
		
		On the other hand, $\tilde \varphi_2$ is exactly the function $\phi_d$.
		Therefore, similarly, by Lemma~\ref{Lemma_contract} and Proposition~\ref{Prop_theta=1},
		$\tilde \varphi_2$ is a contraction on $[0,\alpha_*]$ with unique fixed point $\alpha_*$.
		It follows that $(1-\alpha_*,\alpha_*)$ is the limit $\tilde \varphi^*(0,0)$.
		
		To show that it is the \emph{stable} limit, we simply observe that $\tilde \varphi_1'(1-\alpha^*) = \phi_d'(\alpha^*) <1$
		by Proposition~\ref{Prop_theta=1}, and similarly $\tilde \varphi_2'(\alpha_*) = \phi_d'(\alpha_*) <1$.
		This implies that each coordinate function is a contraction in the neighbourhood of the corresponding limit point,
		and therefore so is $\tilde \varphi$.
	\end{proof}

	\subsection{Proof of \Prop~\ref{Prop_frozen}}\label{sec:Prop_frozen}
	To determine the asymptotic proportion of vertices in $V_\frozen(\vA)$, by Lemma~\ref{lem:locallimit} it suffices to determine the probability that in $\cT$ the root receives at least one $\frozen$-message.
	This event has probability
	\begin{align*}
		\Pr\sqbc{\Po(d(q_{*,\frozen}^{(v)}))\ge 1}=1-\exp(-d(1-\alpha^*))=\alpha_*
	\end{align*}
	since $q_{*,\frozen}^{(v)}=1-\alpha^*$ and by~\eqref{eqPreLambert}.

	An analogous argument yields the statement for $V_\unfrozen(\vA)$.\qed

	\subsection{Proof of Proposition~\ref{Prop_slush}}
	To determine the asymptotic proportion of vertices in $V_\slush(\vA)$, by Lemma~\ref{lem:locallimit} it suffices to determine the probability that in $\cT$ the root receives at least two $\slush$-messages and no $\frozen$-messages.
	This occurs with probability
	\begin{align*}
		\Pr\sqbc{\Po(d(\alpha^*-\alpha_*))\ge 2} &\cdot \Pr\sqbc{\Po(d\alpha_*) = 0}\\
		& = \bc{1-\exp(-d(\alpha^*-\alpha_*)) - d(\alpha^*-\alpha_*)\exp(-d(\alpha^*-\alpha_*))}\cdot \exp\bc{-d\alpha_*}\\
		& = \exp(-d\alpha_*) - \exp(-d\alpha^*)(1+d(\alpha^*-\alpha_*)),
	\end{align*}
	as claimed. The analogous statement for $C_\slush(\vA)$ can be proved similarly, or follows from the statement for $V_\slush(\vA)$ by symmetry.
	
	The statement on degree distributions follows directly from the approximation using $\cT$ or $\hat \cT$: conditioned on a node lying in $V_\slush$ or $C_\slush$, it must certainly receive at least two $\slush$-messages from its neighbours.
	Furthermore, a neighbour is in $C_\slush$ or $V_\slush$ respectively if and only if it sends an $\slush$-message to this vertex.
	The distribution of neighbours sending $\slush$ is $\Po(\lambda)$ without the conditioning (where recall that $\lambda = d(\alpha^*-\alpha_*)$), therefore with the conditioning it is $\Po_{\ge 2}(\lambda)$, as required.\qed
	
	\subsection{Proof of \Prop~\ref{claim_littleintersection}}\label{littleintersection}
	
	For a matrix $A$ we let
	\begin{align}\label{eqVpAt}
		\Vp(A,t)&=\cbc{v\in V(A):\exists a\in\partial v:w_{a\to v}(A,t)=\frozen},&
		\Vu(A,t)&=\cbc{v\in V(A):\forall a\in\partial v:w_{a\to v}(A,t)=\unfrozen},\\
		\Cp(A,t)&=\cbc{a\in C(A):\forall v\in\partial a:w_{v\to a}(A,t)=\frozen},&
		\Cu(A,t)&=\cbc{a\in C(A):\exists v\in\partial a:w_{v\to a}(A,t)=\unfrozen}\label{eqCpAt}
	\end{align}
	be the sets of nodes of $G(A)$ classified as frozen or unfrozen after $t$ iterations of WP.
	Furthermore, let $B(v, t)$ denote the nodes that are within distance $t$ of $v$.  
	Let $\cB_t$ be the set of variable nodes $v$ such that $B(v,t)$ contains at least one cycle.

	\begin{claim} \label{DefunFrozen} Let $t_0\geq 1$.
		If $v_0 \in \Vu(A,t_0)$ and $v_0\notin\cB_{t_0}$, then $v_0 \notin \cF(A)$.
	\end{claim}
	\begin{proof}
		Let $v_0 \in \Vu(A,t_0)$. We will consider a subtree $T$ of $G(A)$ rooted at $v_0$ which we produce in the
		following way.
		All of the neighbours of $v_0$ are added to $T$ as children of $v_0$. Furthermore, since each such neighbour
		$a$ is a check node which sends $v_0$ a $\unfrozen$-message at time $t_0$, the check node $a$
		has at least one further neighbour (apart from $v_0$) from which it receives
		a $\unfrozen$-message at time $t_0-1$ -- we choose one such neighbour arbitrarily and add it to $T$ as a child of $a$.
		We continue recursively, for each variable node adding all neighbours (apart from the parent) if there are any,
		and for each check node at depth $i$ adding one neighbour (distinct from the parent) from which it receives message $\unfrozen$
		at time $t_0-i$.
		
		Since the leaves at depth $t_0$ send out $\unfrozen$-messages at time $1$, they must be unary variables
		(if they exist at all which is not the case if, for example, $t_0$ is odd).
		Therefore $T$ has the property that for any of its variable nodes, all its neighbours are also in $T$,
		while all checks have precisely two neighbours in $T$.
		
		Therefore we can obtain a vector in the kernel of $A$ that sets $x_{v_0}$ to~$1$
		by simply setting all the variable nodes in $T$ to $1$ and all other variables to zero.
		This shows that $v_0\not\in \cF(A)$.
	\end{proof}
	
	\begin{proof}[Proof of \Prop\ \ref{claim_littleintersection}]
		First observe that Claim \ref{DefunFrozen} implies $\Vu(A,t_0) \bigcap \cF(A) \subset \cB_{t_0}.$   Calculating the expectation of the number of vertices lying on cycles of length up to $2t_0$ and applying Markov inequality gives us that indeed $\abs{\cB_{t_0}} = o(n)$.  
		By choosing $t_0$ sufficiently large according to Lemma \ref{lem:fewfurtherchanges} we have $\abs{\Vu(A,t_0)} = \abs{\Vu(A)} + o(n)$ \whp \ which concludes the proof.  
	\end{proof}
	
	\section{The standard messages}\label{sec:3fp}\label{Sec_FP}
	
	\noindent
	In this section we prove \Prop~\ref{lem:3fp}, which states that the proportion of frozen variables is likely close to one of the fixed points
	of $\phi_d$.
	Along the way we will establish auxiliary statements that will pave the way for the proof of \Prop~\ref{Prop_dd}
	(which rules out the unstable fixed point) in \Sec~\ref{Sec_dd} as well.
	
	\subsection{Perturbing the Tanner graph}\label{Sec_perturb}
	A key observation toward \Prop~\ref{lem:3fp} is that if we make some minor alterations to $G(\vA)$, the resulting graph $G'(\vA)$ is essentially indistinguishable from $G(\vA)$.
	Let $\TT=\TT(d)$ be the tree generated by a Galton-Watson process with the two types `variable node' and `check node'.
	The root is a variable node $v_0$.
	Each variable node spawns $\Po(d)$ check nodes as offspring.
	Similarly, the offspring of a check node consists of $\Po(d)$ variable nodes.
	In addition, let $\hat\TT=\hat\TT(d)$ be the tree generated by a Galton-Watson process with the same offspring distribution whose root is a check node $a_0$.
	Given an integer $t$, we obtain $\TT_t$ and $\hat\TT_t$ from $\TT$ and $\hat\TT$, respectively, by deleting all nodes whose distance from the root exceeds $t$, so these are trees of depth (at most) $t$.
	(Unlike the branching processes from \Def~\ref{Def_T}, the trees $\TT,\hat\TT$ do not incorporate messages.)
	
	\begin{definition}\label{Def_Olly}
		Let $0\leq\omega_1=\omega_1(n)=o(\sqrt n)$, $0\leq\omega_2=\omega_2(n)=n^{1/2-\Omega(1)}$
		and obtain $G'(\vA)$ from $G(\vA)$ as follows.
		\begin{enumerate}[(i)]
			\item Generate $\omega_1$ many $\TT_2$ trees and $\omega_2$ many $\hat \TT_1$ trees independently.
			\item For each node $v$ in the final layer of these trees (which is a variable node), embed $v$ onto a variable node of $G(\vA)$ chosen uniformly at random and independently. 
			\item 
			Embed the remaining nodes of the trees randomly onto nodes which were previously isolated such that variable nodes are embedded onto variable nodes and checks onto checks.
		\end{enumerate}
		Let $G'(\vA)$ denote the resulting graph and let $\vA'$ be its adjacency matrix. (Thus $G'(\vA)=G(\vA')$ is the Tanner graph of $\vA'$.)
	\end{definition}

	\begin{figure}\centering 
		\begin{tikzpicture}[scale=0.75] 
			\tikzstyle{var}=[circle,thick,draw,minimum size=4.5mm]
			\tikzstyle{check}=[rectangle,thick,draw,minimum size=3.75mm]
			\draw[blue!90!black] (0,0) node[var] (vr1) {};
			\draw (-0.8,2) node[check] (ci1) {};
			\draw (0.8,2) node[check] (ci2) {};
			\draw[green!75!black] (-1.2,4) node[var] (vl1) {};
			\draw[green!75!black] (0,4) node[var] (vl2) {};
			\draw[green!75!black] (0.8,4) node[var] (vl3) {};
			\draw[green!75!black] (1.6,4) node[var] (vl4) {};
			\draw[very thick,red!75!black] (vr1)--(ci1);
			\draw[very thick,red!75!black] (vr1)--(ci2);
			\draw[very thick,red!75!black] (ci1)--(vl1);
			\draw[very thick,red!75!black] (ci2)--(vl2);
			\draw[very thick,red!75!black] (ci2)--(vl3);
			\draw[very thick,red!75!black] (ci2)--(vl4);
			\draw (vl1)--(-1.6,5.5);
			\draw (vl1)--(-0.8,5.5);
			\draw (vl4)--(1,5.5);
			\draw (vl4)--(1.6,5.5);
			\draw (vl4)--(2.2,5.5);
			\draw[blue!90!black] (4,0) node[var] (vr2) {};
			\draw (3.2,2) node[check] (ci5) {};
			\draw (4.8,2) node[check] (ci6) {};
			\draw[green!75!black] (4.8,4) node[var] (vl5) {};
			\draw[very thick,red!75!black] (vr2)--(ci5);
			\draw[very thick,red!75!black] (vr2)--(ci6);
			\draw[very thick,red!75!black] (ci6)--(vl5);
			\draw[blue!90!black] (6,0) node[var] (vr3) {};

			\draw[blue!90!black] (9,2) node[check] (cr1) {};
			\draw[green!75!black] (9,4) node[var] (vl6) {};
			\draw[very thick,red!75!black] (cr1)--(vl6);
			\draw (vl6)--(8.6,5.5);
			\draw (vl6)--(9.4,5.5);
			\draw[blue!90!black] (11,2) node[check] (cr2) {};
			\draw[blue!90!black] (13,2) node[check] (cr3) {};
			\draw[green!75!black] (12.2,4) node[var] (vl7) {};
			\draw[green!75!black] (13,4) node[var] (vl8) {};
			\draw[green!75!black] (13.8,4) node[var] (vl9) {};
			\draw[very thick,red!75!black] (cr3)--(vl7);
			\draw[very thick,red!75!black] (cr3)--(vl8);
			\draw[very thick,red!75!black] (cr3)--(vl9);
			\draw (vl7)--(12.2,5.5);
			\draw (vl9)--(13.4,5.5);
			\draw (vl9)--(14.2,5.5);

			\draw[blue!90!black] (-1.5,0) node (Vtilde) {$\tilde V$};
			\draw[blue!90!black] (14.5,2) node (Ctilde) {$\tilde C$};
			\draw[green!75!black] (-2.5,4) node (Utilde) {$\tilde U$};
			\draw[rounded corners] (-3,5) rectangle (16,8);
		\end{tikzpicture}
		\caption{An instance of the randomly generated trees added to $G(\vA)$ to produce $G'(\vA)$ in Definition~\ref{Def_Olly}:
			the variable and check root sets $\tilde V,\tilde C$ are shown in blue; the attachment nodes $\tilde U$ in green;
			the thick red edges are those in the trees, which are added to $G(\vA)$;
			the thin black edges were already present in $G(\vA)$;
			all explicitly drawn nodes were already present but, apart from possibly the attachment nodes (i.e., those in $\tilde U$),
			were previously isolated in $G(\vA)$.
		} \label{fig:thimblerig}
	\end{figure}
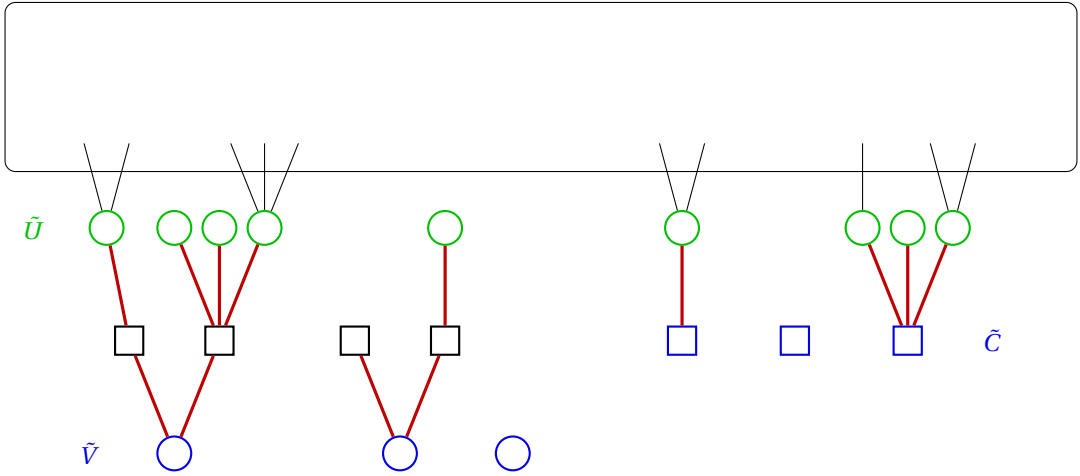
	
	Let $\tilde V,\tilde C$ denote the sets of variable and check nodes of $G'(\vA)$ respectively onto which the roots of the $\TT_2$ and $\hat\TT_1$ branching trees from \Def~\ref{Def_Olly} (i) are embedded.
	Similarly, let $\tilde U=(\partial\tilde C\cup\partial^2\tilde V)\setminus\tilde V$ be the set of variable nodes of $G(\vA)$ where the checks from \Def~\ref{Def_Olly} attach to the bulk of the Tanner graph in Step~(ii). 
	An example is shown in Figure~\ref{fig:thimblerig}.

	Note that it is possible that this process fails, for example if there are not enough isolated nodes available, in which case we simply set $G'(\vA):=G(\vA)$.
	However, since \whp\ the total size of all trees is $O(\omega_1+\omega_2)$, and \whp\ there are $\Omega(n)$ isolated variables and check nodes available, the failure probability is $\exp(-\Omega(n))$ and thus negligible for our purposes.
	For the same reason \whp\ no two nodes from the trees are embedded onto the same node of $G(\vA)$.
	
	\begin{fact}\label{fact:TVsmall}
		If $\omega_1+\omega_2=n^{1/2-\Omega(1)}$, then $\dTV(G(\vA),G'(\vA))=n^{-\Omega(1)}$.
	\end{fact}
	
	\begin{proof}[Proof sketch]
		We will focus on proving the case when $\omega_1=0$, i.e., the case
			where we add only $\hat \TT_1$ trees attached to isolated check nodes.
			The argument for adding $\TT_2$ trees is conceptually similar,
			but more technically involved.
		
		Let us first observe that
			$\dTV\bc{\Bin(n,d/n),\Po(d)} = O\bc{n^{-1}} = o\bc{\omega_2^{-1}n^{-\Omega(1)}}$,
			and therefore with probability $1-n^{-\Omega(1)}$
			the effect of adding the $\omega_2$ many $\hat \TT_1$ trees to $G(\vA)$
			would be identical if the trees were constructed using the $\Bin(n,d/n)$ offspring distribution.
			For convenience, we will therefore assume in the rest of the proof that this is indeed
			how $G'(\vA)$ is constructed from $G(\vA)$.
		
		For each $i \in \NN_0$, let $\cC_{i}$ denote the set of rows of degree $i$ in $\vA$,
			and set $c_i:= |\cC_i|$. We also let $\cC_i',c_i'$ denote the analogous variables in $\vA'$.
			Now let us sample $\omega_2$ independent $\Po(d)$ random variables
			and let $c_i^{*}$ be the number of these that are equal to $i$.
			These represent the number of rows corresponding to checks in $\tilde C$ with degree $i$,
			and therefore $c_i' = c_i+c_i^*$ for $i\ge 1$,
			while $c_0' = c_0-\sum_{i\ge 1}c_i^* = c_0 - (\omega_2-c_0^*)$.  
		
		We will reveal the information about $\vA,\vA'$ a little at a time. Our aim is to show
			that in most steps of this revealing process, we perform an identical operation
			regardless of which matrix we are constructing, while in the remaining steps the difference
			in the construction only leads to a total variation distance $o(1)$.
			It will be helpful to use a superscript $\square$ to denote a quantity or set in either
			$\vA$ or $\vA'$ as appropriate, so for example $c_i^\square$ signifies either
			$c_i$ or $c_i'$ depending on the context.
			We now construct $\vA,\vA'$ as follows.
		\begin{enumerate}
				\item Reveal $c_0$ and $c_0^*$, which also determines $c_0'=c_0- (\omega_2-c_0^*)$; 
				\item Reveal $\cC_0^\square$;
				\item For each $i \in \NN$, reveal $\cC_i^\square$;
				\item For each $i \in \NN$ and each row $j\in \cC_i^\square$, reveal the $j$ positions of its $1$s.
		\end{enumerate}
		We claim that the first step is the only step at which the processes differ
			for $\vA,\vA'$.
			Indeed, in the second step we simply reveal $c_0^\square$ rows
			uniformly at random;
			in the third, for each row not in $\cC_0^\square$ we reveal its degree, which
			now has distribution $\Bin_{\ge 1}(n,d/n)$;
			and in the fourth, we simply reveal the positions of the required number of $1$s
			in a row uniformly at random.
		
		Thus the first step is the only one at which there is a difference.
			However, let us observe that $c_0 \sim \Bin\bc{n,(1-d/n)^n}$,
			which has standard deviation $\Theta\bc{n^{1/2}}$,
			while $\omega_2- c_0^* \le \omega_2 = O\bc{n^{1/2-\Omega(1)}}$.
			It is a standard fact about the binomial distribution that
			$\dTV\bc{c_0-(\omega_2-c_0^*),c_0} = O\bc{n^{-\Omega(1)}}$,
			and since $c_0'=c_0-(\omega_2-c_0^*)$, this completes the proof.
	\end{proof}

	We point out that $\tilde V,\tilde C$ are representative of $G'(\vA)$ as a whole.

	\begin{fact}\label{prop:representative}
		Let $\Lambda:(G,u)\mapsto\Lambda(G,u)\in[0,1]$ be any function that maps a pair consisting of a graph and a node to a number.
		If $1\ll\omega_1,\omega_2=n^{1/2-\Omega(1)}$, then 
		\begin{align*}
			&\Erw\abs{\frac{1}{n}\sum_{v\in V(G'(\vA))}\Lambda(G'(\vA),v) -\frac{1}{|\tilde V|}\sum_{v\in\tilde V}\Lambda(G'(\vA),v)} = o(1),\\
			&\Erw\abs{\frac1n\sum_{a\in C(G'(\vA))}\Lambda(G'(\vA),a) - \frac1{|\tilde C|}\sum_{a\in\tilde C}\Lambda(G'(\vA),a)} = o(1).
		\end{align*}
	\end{fact}
	\begin{proof}
		The statement for $\tilde V$ follows since the local structure of $G(\vA)$, and therefore also of $G'(\vA)$ by Fact~\ref{fact:TVsmall}, is that of a $\Po(d)$ branching tree, and this is clearly also the case at the variables of $\tilde V$.
		Formally, if $\vv$ is a variable node chosen uniformly at random from $V(G'(\vA))$ and $\tilde{\vv}$ is a random element of $\tilde V$, then Fact~\ref{fact:TVsmall} implies that $(G'(\vA),\vv)$ and $(G'(\vA),\tilde\vv)$ have total variation distance $o(1)$ given $G'(\vA)$ \whp\ 
		
		To justify this last statement in more detail, consider the process of attempting to distinguish between the two,
			i.e., of attempting to determine whether a variable $v$ of $G(\vA)$ was chosen uniformly at random from $\tilde V$ (as for $\tilde \vv$)
			or from $V$ (as for $\vv$), with only the graph $G'(\vA)$ to hand and no knowledge of $G(\vA)$ or of where edges were added.
			We can reveal the isomorphism class of the depth~$2$ neighbourhood of~$v$, but can couple both distributions such that the outcome is identical
			$\whp$, since both have asymptotically the distribution of a $\Po(d)$ branching tree truncated at depth~$2$.
			Conditioned on the isomorphism class, the location of the nodes in this depth~$2$ neighbourhood is also uniformly random in 
			both distributions. Finally, the remainder of the graph in both models is simply that of $G(\vA)$ conditioned on
			having no edges incident to the depth one neighbourhood of~$v$.

		Therefore, the empirical average of $\Lambda$ on the entire set $V(G'(\vA))$ is well approximated by the average on $\tilde V$ \whp\
		The second statement concerning  $\tilde C$ follows similarly.
	\end{proof}
	
	\subsection{Construction of the standard messages}
	In \Sec~\ref{Sec_WP} we defined Warning Propagation messages via an explicit combinatorial construction that captured our intuition as to the causes of freezing.
	In the following we pursue a converse path.
	We define a set of messages implicitly, purely in terms of algebraic reality.
	We call these $\{\frozen,\unfrozen\}$-valued messages the {\em standard messages}.
	The battle plan is to ultimately match this implicit definition with the explicit construction from \Sec~\ref{Sec_WP}.
	
	The standard messages can be defined for any $m\times n$-matrix $A$.
	Given a subset $U$ of nodes of a graph~$G$, we denote by $G-U$ the graph
	obtained from $G$ by deleting $U$ and all incident edges. For a node $x$, we write $G-x$ instead of $G-\{x\}$.
	For each adjacent variable/check pair $(v,a)$ of $G(A)$ we define
	\begin{align}\label{eqStandard}
		\fm_{v\to a}(A)&=\begin{cases}
			\frozen&\mbox{ if $v$ is frozen in $G(A)-a$,}\\
			\unfrozen&\mbox{ otherwise,}
		\end{cases}\quad
		&
		\fm_{a\to v}(A)&=\begin{cases}
			\frozen&\mbox{ if $v$ is frozen in $G(A)-(\partial v\setminus\cbc a)$,}\\
			\unfrozen&\mbox{ otherwise.}
		\end{cases}
		\qquad
	\end{align}
	
	\noindent
	Hence, $\fm_{v\to a}(A)=\frozen$ iff $v$ is frozen in the matrix obtained from $A$ by deleting the $a$-row.
	Moreover, $\fm_{a\to v}(A)=\frozen$ iff $v$ is frozen in the matrix obtained by removing the rows of all $b\in\partial v$ except $a$.
	Let $\fm(A)=(\fm_{v\to a}(A),\fm_{a\to v}(A))_{v\in\partial a}$.
	
	Further, we define $\{\frozen,\fu,\unfrozen\}$-valued marks for the variables and checks by letting
	\begin{align}\label{eqMarks1}
		\fm_v(A)&=
		\begin{cases}
			\frozen&\mbox{ if $\fm_{a\to v}(A)=\frozen$ for at least two $a\in\partial v$,}\\
			\fu&\mbox{ if $\fm_{a\to v}(A)=\frozen$ for precisely one $a\in\partial v$,}\\
			\unfrozen&\mbox{ otherwise,}
		\end{cases}\quad\\
		\fm_a(A)&=\begin{cases}
			\frozen&\mbox{ if $\fm_{v\to a}(A)=\frozen$ for all $v\in\partial a$,}\\
			\fu&\mbox{ if $\fm_{v\to a}(A)=\frozen$ for all but precisely one $v\in\partial a$,}\\
			\unfrozen&\mbox{ otherwise.}
		\end{cases}\qquad
		\label{eqMarks2}
	\end{align}
	The intended semantics is that, barring intuitively unlikely dependencies, $\frozen$ and $\fu$ both represent frozen variables/checks, meaning that a variable $v$ is frozen if $\fm_v(A)\neq\unfrozen$
	while for any check $a$ we have $\fm_a(A)\neq\unfrozen$ if all variables $v\in\partial a$ are frozen.
	But for checks or variables with mark $\fu$, freezing hangs by a thread since, for instance, a variable $v$ with $\fm_v(A)=\fu$ receives just a single `freeze' message.
	This manifests itself in the messages sent out by $\fu$-variables or checks: a variable or check with mark $\frozen$ can only
		ever send out messages of $\frozen$, and similarly for $\unfrozen$ instead of $\frozen$, but a variable or check with a mark of $\fu$ can send out
		different messages to different neighbours.
		We will see a much more detailed description of this phenomenon in \Cor~\ref{Cor_stats} below.

		Although the standard messages are derived from algebraic reality rather than the WP algorithm,
		we can still consider what happens when applying the Warning Propagation operator $\WP_A$ from \Sec~\ref{Sec_WP}
		to them. Since the standard messages involve only $\frozen$ and $\unfrozen$ messages, the updated messages $\hat\fm$ given
		by~\eqref{eqWP1} can be described more simply:
	\begin{align}\label{eqSimpleWP1}
		\hat\fm_{v\to a}(A)&=\begin{cases}
			\frozen&\mbox{ if $\fm_{b\to v}(A)=\frozen$ for some $b\in\partial v\setminus\cbc a$,}\\
			\unfrozen&\mbox{ otherwise,}
		\end{cases}\\\label{eqSimpleWP2}
		\hat\fm_{a\to v}(A)&=\begin{cases}
			\frozen&\mbox{ if $\fm_{y\to a}(A)=\frozen$ for all $y\in\partial a\setminus\cbc v$,}\\
			\unfrozen&\mbox{ otherwise.}
		\end{cases}\quad
	\end{align}
	We next show that the standard messages constitute an approximate fixed point of the $\WP_{\vA}$ operator, i.e., that 
		$\hat\fm$ and $\fm$ are almost always equal, and that the marks mostly match their intended semantics \whp
	
	\begin{lemma}\label{Lemma_standard}
		For all $d>0$ we have
		\begin{align}\label{eqLemma_standard1}
			\Erw\sum_{\substack{v\in V(\vA)\\a\in\partial v}}\vecone\cbc{\fm_{v\to a}(\vA)\neq\hat\fm_{v\to a}(\vA)}+\vecone\cbc{\fm_{a\to v}(\vA)\neq\hat\fm_{a\to v}(\vA)}&=o(n),\\
			\Erw\abs{\cbc{v\in V(\vA):\fm_v(\vA)\neq\unfrozen}\triangle\cF(\vA)}=o(n),\qquad
			\Erw\abs{\cbc{a\in C(\vA):\fm_a(\vA)\neq\unfrozen}\triangle\hat\cF(\vA)}&=o(n).
			\label{eqLemma_standard2}
		\end{align}
	\end{lemma}
	
	We prove \Lem~\ref{Lemma_standard} by way of the perturbation from \Sec~\ref{Sec_perturb}.
	Specifically, in light of Fact~\ref{prop:representative} it suffices to consider $G'(\vA)$ and the sets of variables/checks $\tilde V,\tilde C$
	onto which the roots of the $\TT_2$ and $\hat\TT_1$ branching trees from \Def~\ref{Def_Olly} are embedded.
	The following claim summarises the main step of the argument.
	Recall that $\tilde U$ is the set of variable nodes where the trees from \Def~\ref{Def_Olly} attach to the bulk of the Tanner graph in Step~(ii) (see Figure~\ref{fig:thimblerig}). 
	
	\begin{claim}\label{Lemma_msg}
		There exists $1\ll\omega^*=\omega^*(n)\leq n^{1/2-\Omega(1)}$ such that for all $\omega_1,\omega_2\leq\omega^*$ and every $d>0$ \whp\ we have
		\begin{align}\label{eqLemma_msg1}
			\fm_{y\to a}(\vA')&=\frozen\ \Leftrightarrow\ y\in\cF(\vA)&&\mbox{for all }a\in\tilde C\cup\partial\tilde V,\,y\in\tilde U\cap\partial a.
		\end{align}
		Furthermore, \whp\ a random vector $\vx\in\ker\vA$ satisfies
		\begin{align}\label{eqLemma_msg2}
			\pr\brk{\forall y\in \tilde U\setminus\cF(\vA):\vx_y=\sigma_y\mid G(\vA),G'(\vA)}&=2^{-|\tilde U\setminus\cF(\vA)|}&&\mbox{for all }\sigma\in\field^{\tilde U\setminus\cF(\vA)}.
		\end{align}
		Finally, $\cF(\vA)\subset\cF(\vA')$ and \whp\ we have $f(\vA')=f(\vA)+o(1)$.
	\end{claim}
	
	In words,~\eqref{eqLemma_msg1} states that the standard messages in $G'(\vA)$ sent down from the attachment nodes
		into the added trees are determined by whether the corresponding attachment node was frozen in $\vA$. This is intuitively
		natural to expect (and would be trivial if the trees we added consisted of just one additional edge), and it may well be that
		an added tree causes an attachment node to freeze in an obvious way without needing to consider the rest of $G(\vA)$,
		but the main take-away message of~\eqref{eqLemma_msg1} is that
		added trees do not cause strange and unexpected freezing due to restrictions passing from the added trees
		into the rest of the graph and returning back to the attachment nodes.

		On the other hand,~\eqref{eqLemma_msg2} states that the attachment nodes which were not frozen
		in $G(\vA)$ behave independently in $\ker(\vA)$, in the sense that any combination of assignments
		on these vertices is equally likely in a random vector $\vx$. (Of course, this will certainly \emph{not}
		be true after the addition of trees since, for example, we will likely add some check nodes in $\tilde C$
		of degree~$2$, which force their two neighbours to take the same value.)

	\begin{proof}
		Let us begin with the last statement.
		The inclusion $\cF(\vA)\subset\cF(\vA')$ is deterministically true because $\vA'$ is obtained from $\vA$ by effectively adding checks (viz.\ ``activating'' formerly dormant isolated checks).
		Moreover, \Prop~\ref{Prop_pin} shows that the distribution of a random $\vx\in\ker\vA$ is $n^{-\Omega(1)}$-symmetric \whp\
		Since $\vA'$ is obtained from $\vA$ by adding no more than $O(\omega^*)$ checks \whp\ and since any additional check reduces the nullity by at most one, the distributions of a uniformly random $\vx'\in\ker\vA'$ and of $\vx$ are mutually $2^{O(\omega^*)}$-contiguous \whp\
		Therefore, \Prop~\ref{Prop_contig} implies that 	\whp
		\begin{align}\label{eqLemma_msg3}
			\cutm(\vx,\vx')=o(1),
		\end{align}
		provided that $\omega^*=\omega^*(n)$ grows sufficiently slowly (so in particular setting $\eps=2^{-O(\omega^*)}$,
			the $\delta$ given by \Prop~\ref{Prop_contig} is larger than $n^{-\Omega(1)}$).
		Finally, since the marginals of the individual entries $\vx_i,\vx_i'$ are either uniform or place all mass on zero by Fact~\ref{fact_ker}, \eqref{eqMargBound} and \eqref{eqLemma_msg3} yield
		\begin{align}\label{eqLemma_msg3a}
			f(\vA')-f(\vA)=\frac1n\sum_{i=1}^n\vecone\{v_i\in\cF(\vA')\}-\vecone\{v_i\in\cF(\vA)\}
			\leq\frac2n\sum_{i=1}^n\dTV(\vx_i,\vx_i')\leq4\cutm(\vx,\vx')=o(1).
		\end{align}
		
		The other two assertions \eqref{eqLemma_msg1} and \eqref{eqLemma_msg2} follow from similar deliberations.
		Indeed, to prove \eqref{eqLemma_msg2} we observe that given $G(\vA)$ the set $\tilde U$ of variable nodes where the bottom layers of the trees from \Def~\ref{Def_Olly} attach in Step~(ii) is just a uniformly random set of $O(\omega^*)$ variable nodes of $G(\vA)$.
		Therefore, 
		provided $\omega^*\to\infty$ sufficiently slowly, \Prop~\ref{Prop_pin} shows that \whp\
		\begin{align}\label{eqLemma_msg5}
			\pr\brk{\forall y\in \tilde U\setminus\cF(\vA):\vx_y=\sigma_y\mid G(\vA),G'(\vA)}-2^{-|\tilde U\setminus\cF(\vA)|}
			= O\bc{n^{-\Omega(1)}}&&\mbox{ for any }\sigma\in\field^{\tilde U\setminus\cF(\vA)}.
		\end{align}
		Now, the projections of the vectors $x\in\ker\vA$ onto the coordinates in $\tilde U\setminus\cF(\vA)$ form a subspace of $\field^{\tilde U\setminus\cF(\vA)}$. 
		Assuming that $|\tilde U|=O(\omega^*)$ and that $\omega^*\to\infty$ sufficiently slowly, \eqref{eqLemma_msg5} implies that the dimension of this subspace equals $|\tilde U\setminus\cF(\vA)|$.
			It follows that there exists a set $\tilde B$ of $ |\tilde U\setminus\cF(\vA)|$ vectors in $\ker \vA$ whose restriction to
			$\tilde U\setminus\cF(\vA)$ forms the standard basis. This set $\tilde B$ can be extended to a basis $B$ of $\ker \vA$
			in which each vector of $B\setminus\tilde B$ takes value $0$ on $\tilde U\setminus\cF(\vA)$.
			A random vector of $\ker \vA$ can be generated by taking each element of $B$ with probability $1/2$ independently
			and taking the sum. Since in particular each element of $\tilde B$ is included in the sum with probability $1/2$ independently,
			\eqref{eqLemma_msg2} follows.
	
	Regarding \eqref{eqLemma_msg1}, fix some check $a\in\tilde C\cup\partial\tilde V$ and think of $G'(\vA)$,
	and therefore also its adjacency matrix~$\vA'$, as being constructed from $G(\vA)$ in two steps.
	In the first step we select all the nodes onto which the additional trees will be embedded except for $\partial a \cap \tilde U$,
		and add all of the additional edges except those incident to $a$. (Note that if $a\notin \tilde C$, we do \emph{not} add the edge between
		$a$ and the corresponding variable node of $\tilde V$).
	Let $G''(\vA)$ be the outcome of this first step and let $\vA''$ be its adjacency matrix.
	Subsequently we independently choose the set of neighbours $\partial a\setminus\tilde V$ among the variable nodes of $G(\vA)$
	and add the edges incident to $a$ to obtain $G'(\vA)$.
	Let $\vx''$ be a random element of $\ker\vA''$.
	Repeating the argument towards \eqref{eqLemma_msg3}  we see that $\cutm(\vx,\vx'')=o(1)$ \whp\
	Hence, repeating the steps of \eqref{eqLemma_msg3a} we conclude that 
	$|\cF(\vA)\triangle\cF(\vA'')|=o(n)$ \whp\
	Since in our two-round exposure $\partial a\setminus\tilde V$ is independent of $\vA''$, we thus conclude that $\partial a\cap\cF(\vA'')\setminus\tilde V=\partial a\cap\cF(\vA)\setminus\tilde V$ \whp\
	Hence, the definition \eqref{eqStandard} of the standard messages implies \eqref{eqLemma_msg1}.
\end{proof}

\begin{proof}[Proof of \Lem~\ref{Lemma_standard}]
	By Fact~\ref{prop:representative} it suffices to prove the fixed point conditions for the variables and checks $\tilde V,\tilde C$ of $G'(\vA)$
	which are the roots of the $\TT_2$ and $\hat\TT_1$ branching processes added in \Def~\ref{Def_Olly}.
	Hence, with $\omega^*$ from Claim~\ref{Lemma_msg} let $\omega_1=\omega_*$ and $\omega_2=0$ and assume that \eqref{eqLemma_msg1}--\eqref{eqLemma_msg2} are satisfied.
	We may also assume that the subgraph of $G'(\vA)$ induced on $\cX=\tilde V\cup\tilde U\cup\partial \tilde V$ is acyclic.
	Pick a variable $v\in\tilde V$ and an adjacent check $a\in\partial v$.
	We will show that under the assumptions the fixed point property is satisfied deterministically.
	
	The definition \eqref{eqStandard} of the standard messages provides that $\fm_{a\to v}(\vA')=\frozen$ iff $v$ is frozen in $G'-(\partial v\setminus\cbc a)$.
	A sufficient condition is that $\partial a\setminus\cbc v\subset\cF(\vA)$. 
	Conversely, if $\partial a\setminus(\cbc v\cup\cF(\vA))\neq\emptyset$, then \eqref{eqLemma_msg2} shows that $v$ is unfrozen in $G'(\vA)-(\partial v\setminus\cbc a)$. 
	For there exists $\sigma\in\ker\vA$ such that $\sum_{y\in\partial a\setminus\cbc v}\sigma_y=1$, and because the subgraph induced on $\cX$ is acyclic this vector $\sigma$ extends to a vector $\sigma'\in\ker\vA'$ with $\sigma'_v=1$.
	Hence, $v\not\in\cF(\vA')$.
	Furthermore, \eqref{eqLemma_msg1} ensures that $\partial a\setminus\cbc v\subset\cF(\vA)$ iff $\fm_{y\to a}(\vA')=\frozen$ for all $y\in\partial a\setminus\cbc v$.
	Hence, $\fm_{a\to v}(\vA')=\frozen$ iff $\fm_{y\to a}(\vA')=\frozen$ for all $y\in\partial a\setminus\cbc v$.
	In other words, we obtain 
	\begin{align}\label{eqLemma_standard_10}
		\fm_{a\to v}(\vA')&=\hat\fm_{a\to v}(\vA')&&\mbox{for all }v\in\tilde V,\,a\in\partial v.
	\end{align}
	
	A similar argument shows that	
	\begin{align}\label{eqLemma_standard_11}
		\fm_{v\to a}(\vA')&=\hat\fm_{v\to a}(\vA')&&\mbox{for all }v\in\tilde V,\,a\in\partial v.
	\end{align}
	Indeed, \eqref{eqStandard} guarantees that $\fm_{v\to a}(\vA')=\frozen$ if there is a check $b\in\partial v\setminus\cbc a$ such that $\partial b\setminus\cbc v\subset\cF(\vA)$.
	Such a check satisfies $\fm_{b\to v}(\vA')=\frozen$, and thus \eqref{eqSimpleWP1} shows that $\hat\fm_{v\to a}(\vA')=\frozen$.
	Conversely, suppose that $\fm_{v\to a}(\vA')=\unfrozen$.
	Then \eqref{eqStandard} shows that $v$ is unfrozen in $G'(\vA)-a$.
	Hence, the kernel of the matrix obtained from $\vA'$ by deleting the $a$-row contains a vector $\sigma''$ with $\sigma''_v=1$.
	Therefore, any check $b\in\partial v\setminus a$ features a variable $y\in\partial b\setminus(\cbc v\cup\cF(\vA))$.
	Consequently, because the subgraph induced on $\cX$ is acyclic, \eqref{eqLemma_msg2} implies that $v$ is unfrozen in the subraph $G'(\vA)-(\partial v\setminus\{b\})$ where the only check adjacent to $v$ is $b$.
	Thus, $\fm_{b\to v}(\vA')=\unfrozen$.
	Finally, \eqref{eqSimpleWP1} shows that $\hat\fm_{v\to a}(\vA')=\unfrozen$.

	The proof of \eqref{eqLemma_standard2} proceeds along similar lines.
	Indeed, $v\in\tilde V$ is frozen in $\vA'$ if there exists a check $a\in\partial v$ such that $\partial a\setminus\cbc v\subset\cF(\vA)$.
	Hence, \eqref{eqLemma_msg1} shows that the existence of a check $a\in\partial v$ with $\fm_{a\to v}(\vA')=\frozen$ is a sufficient condition for $v\in\cF(\vA')$.
	Conversely, \eqref{eqLemma_msg2} shows that the absence of such a check is a sufficient condition for $v\not\in\cF(\vA')$.
	Thus, recalling the definition \eqref{eqMarks1}, we obtain the first part of \eqref{eqLemma_standard2}.
	
	To prove the second part we combine \eqref{eqLemma_standard1}--\eqref{eqLemma_standard2} with \eqref{eqLemma_standard_11} to see that $a\in\hat\cF(\vA')$ iff there is at most one $y\in\partial a$ with $\fm_{y\to a}(\vA')=\unfrozen$.
	For clearly $a\in\hat\cF(\vA')$ if no such $y$ exists, while if there is precisely one such $y$ the presence of the check $a$ will freeze this variable.
	Conversely, if at least two $y,y'\in\partial a$ satisfy $\fm_{y\to a}(\vA'),\fm_{y'\to a}(\vA')\neq\frozen$, then $a\not\in\cF(\vA')$ due to \eqref{eqLemma_msg2}.
	Thus, a glance at the definition \eqref{eqMarks2} of $\fm_a(\vA')$ completes the proof of \eqref{eqLemma_standard2}.
\end{proof}

Proposition~\ref{lem:3fp} roughly states that the proportion of frozen variables is close to one of the fixed points of WP;
in order to prove this result,
we will need to analyse the distribution of the numbers of incoming and outgoing standard messages of each type at a node.
This motivates the following definitions.

Given a vector $\Lu =\bc{\ell_{\unfrozen\unfrozen},\ell_{\unfrozen\frozen},\ell_{\frozen\unfrozen},\ell_{\frozen\frozen}} \in \NN_0^4$ and $\s\in\{\frozen,\fu,\unfrozen\}$, let
\begin{align*}
	\vDelta_{A}(\s,\Lu)&=
	\sum_{v\in V(A)}\vecone\cbc{\fm_{v}(A)=\s}\prod_{\x,\y\in\{\unfrozen,\frozen\}}\vecone\cbc{\abs{\cbc{a\in\partial v:\fm_{a\to v}(A)=\x \mbox{ and }\fm_{v\to a}(A)=\y}}=\ell_{\x \y}},\\
	\vGamma_A(\s,\Lu)&=
	\sum_{a\in C(A)}\vecone\cbc{\fm_{a}(A)=\s}\prod_{\x,\y\in\{\unfrozen,\frozen\}}\vecone\cbc{\abs{\cbc{v\in\partial a:\fm_{v\to a}(A)=\x\mbox{ and }\fm_{a\to v}(A)=\y}}=\ell_{\x \y}}.
\end{align*}
These random variables count variables/checks with certain marks and given numbers of edges with specific incoming/outgoing messages.
For instance, $\ell_{\unfrozen\frozen}$ provides the number of edges with an incoming $\unfrozen$-message and an outgoing $\frozen$-message.
Of course, for some choices of $\s$ and $\Lu$ the variables $\vDelta_{A}(\s,\Lu)$ and $\vGamma_A(\s,\Lu)$ may equal zero deterministically.
We can think of $\vDelta$ and $\vGamma$ as generalised degrees,
giving information not just about the number of edges, but the number of edges with each type of message.
The following corollary pinpoints the generalised degree distribution.
For $\alpha,\hat \alpha \in [0,1]$ and $\Lu=(\ell_{\unfrozen\unfrozen},\ell_{\unfrozen\frozen},\ell_{\frozen\unfrozen},\ell_{\frozen\frozen}) \in \NN_0^4$,
we define
\begin{align}\label{eqCor_stats11}
	\bDelta(\hat\alpha,\unfrozen,\Lu)&=\vecone\cbc{\ell_{\frozen\unfrozen}=\ell_{\unfrozen\frozen}=\ell_{\frozen\frozen}=0} \cdot \pr\brk{\Po(d\hat\alpha)=0} \cdot \pr\brk{\Po(d(1-\hat\alpha))=\ell_{\unfrozen\unfrozen}},\\
	\bDelta(\hat\alpha,\fu,\Lu)&=\vecone\cbc{\ell_{\frozen\unfrozen}=1,\,\ell_{\unfrozen\unfrozen}=\ell_{\frozen\frozen}=0} \cdot \pr\brk{\Po(d\hat\alpha)=1} \cdot \pr\brk{\Po(d(1-\hat\alpha))=\ell_{\unfrozen\frozen}}\label{eqCor_stats12},\\
	\bDelta(\hat\alpha,\frozen,\Lu)&=\vecone\cbc{\ell_{\frozen\unfrozen}=\ell_{\unfrozen\unfrozen}=0,\ell_{\frozen\frozen}\geq2}\cdot \pr\brk{\Po(d\hat\alpha)=\ell_{\frozen\frozen}}\cdot \pr\brk{\Po(d(1-\hat\alpha))=\ell_{\unfrozen\frozen}}\label{eqCor_stats13},\\
	\bGamma(\alpha,\unfrozen,\Lu)&=\vecone\cbc{\ell_{\unfrozen\frozen}=\ell_{\frozen\frozen}=0,\ell_{\unfrozen\unfrozen}\geq2}\cdot \pr\brk{\Po(d(1-\alpha))=\ell_{\unfrozen\unfrozen}}\cdot \pr\brk{\Po(d\alpha)=\ell_{\frozen\unfrozen}}\label{eqCor_stats14},\\
	\bGamma(\alpha,\fu,\Lu)&=\vecone\cbc{\ell_{\unfrozen\frozen}=1,\,\ell_{\unfrozen\unfrozen}=\ell_{\frozen\frozen}=0}\cdot\pr\brk{\Po(d (1-\alpha)=1}\cdot\pr\brk{\Po(d\alpha)=\ell_{\frozen\unfrozen}}\label{eqCor_stats15},\\
	\bGamma(\alpha,\frozen,\Lu)&=\vecone\cbc{\ell_{\frozen\unfrozen}=\ell_{\unfrozen\frozen}=\ell_{\unfrozen\unfrozen}=0}\cdot\pr\brk{\Po(d (1-\alpha))=0}\cdot\pr\brk{\Po(d\alpha)=\ell_{\frozen\frozen}}.\label{eqCor_stats16}
\end{align}

\begin{corollary}\label{Cor_stats}
	Let $d>0$.
	For any $\s\in\{\frozen,\fu,\unfrozen\}$ and $\Lu  = (\ell_{\unfrozen\unfrozen},\ell_{\unfrozen\frozen},\ell_{\frozen\unfrozen},\ell_{\frozen\frozen}) \in \NN_0^4$ we have
	\begin{align}\label{eqCor_stats}
		\lim_{n\to\infty}\frac{1}{n}\ex\brk{\abs{\vDelta_{\vA}(\s,\Lu)-\bDelta(\hat f(\vA),\s,\Lu)}+\abs{\vGamma_{\vA}(\s,\Lu)-\bGamma(f(\vA),\s,\Lu)}}&=0,\\
		\lim_{n\to\infty}\ex\brk{\abs{f(\vA) - \phi_d(f(\vA)) }+\abs{\hat f(\vA)-(1+d(1-f(\vA)))\exp\bc{-d(1-f(\vA))}}}&=0.\nonumber
	\end{align}
\end{corollary}

\begin{proof}
	In light of Fact~\ref{prop:representative} it once again suffices to prove the various estimates for the variables/checks from $\tilde V,\tilde C$. 
	Hence, with $\omega^*$ from Claim~\ref{Lemma_msg} let $1\ll\omega_1,\omega_2\ll\omega^*$.
	
	To prove the second part of \eqref{eqCor_stats} we consider a check $a\in\tilde C$.
	The construction in \Def~\ref{Def_Olly} ensures that $a$ randomly selected $\vk(a)\sim\Po(d)$ random variable nodes of $G$ as neighbours.
	Each of them belongs to $\cF(\vA)$ with probability $f(\vA)$.
	Thus, $\vk(a)$ decomposes into two independent Poisson variables $\vk_{\frozen}(a)$ and $\vk_{\unfrozen}(a)$ with means $f(\vA)d$ and $(1-f(\vA))d$.
	Furthermore, the definition \eqref{eqMarks2} of the marks ensures that the mark of $a$ depends only on the incoming messages.
	Moreover, \eqref{eqMarks2} implies together with \eqref{eqLemma_msg1} that \whp\ over the choice of $\vA$ for any fixed integers $\ell_\unfrozen,\ell_\frozen\geq0$ we have
	\begin{align}\label{eqCor_stats10}
		\pr\brk{\fm_a(\vA')=\unfrozen,\;\vk_{\frozen}(a)=\ell_\frozen,\;\vk_{\unfrozen}(a)=\ell_\unfrozen\mid\vA}&=\vecone\cbc{\ell_\unfrozen\geq2}\pr\brk{\Po(d(1-f(\vA)))=\ell_\unfrozen}\pr\brk{\Po(df(\vA))=\ell_\frozen}+o(1).
	\end{align}
	Indeed, \eqref{eqMarks2} ensures that $\fm_a(\vA')=\unfrozen$ only if $a$ receives at least two $\unfrozen$-messages.
	Furthermore, as Fact~\ref{Lemma_msg} shows that $f(\vA')=f(\vA)+o(1)$ \whp, we can rewrite \eqref{eqCor_stats10} as 
	\begin{align}\label{eqCor_stats_11}
		\pr\brk{\fm_a(\vA')=\unfrozen,\;\vk_{\frozen}(a)=\ell_\frozen,\;\vk_{\unfrozen}(a)=\ell_\unfrozen\mid\vA}&=\vecone\cbc{\ell_\unfrozen\geq2}\pr\brk{\Po(d(1-f(\vA)))=\ell_\unfrozen}\pr\brk{\Po(df(\vA))=\ell_\frozen}+o(1).
	\end{align}
	Since by the fixed point property from \Lem~\ref{Lemma_standard} the reverse messages sent out by $a$ are determined by the incoming ones via \eqref{eqSimpleWP2} \whp, all messages returned by a check with mark $\unfrozen$ are $\unfrozen$ \whp\
	Therefore, \eqref{eqCor_stats_11} implies the second part of \eqref{eqCor_stats}.
	Finally, we observe that the identity  $$\lim_{n \to \infty}\Erw\abs{\hat f(\vA) - (1+d(1-f(\vA)))\exp\bc{-d(1-f(\vA))}}=0$$ is equivalent
	to the statement that \whp\ $\hat f(\vA) = (1+d(1-f(\vA)))\exp\bc{-d(1-f(\vA))} + o(1)$, which actually follows from
	\eqref{eqLemma_standard2}, \eqref{eqCor_stats14}
	and~\eqref{eqCor_stats}
	by summing over $\Lu \in \NN_0^4$.
	More precisely, \eqref{eqLemma_standard2} implies that \whp\ $\hat f(\vA) = n^{-1}\abs{\cbc{a : \fm_a(\vA) \neq \unfrozen}} + o(1)$. Furthermore, by~\eqref{eqCor_stats}, \whp\ for all but $o(n)$ check nodes $a$ we have $\fm_a(\vA)\neq\unfrozen$ if and only if $a$ is adjacent to
	no edge along which both messages are $\unfrozen$. 
	A glance at~\eqref{eqCor_stats14} 
	shows that the sum over all $\Lu \in \NN_0^4$ of $\bGamma(\alpha,\unfrozen,\Lu)$ is simply
	$\Pr\brk{\Po(d(1-\alpha))\ge 2} = 1-(1+d(1-\alpha))\exp\bc{-d(1-\alpha)}$.
	Considering the complement and substituting $\alpha=f(\vA)$, the result follows.

	The first part of \eqref{eqCor_stats} also follows from similar deliberations.
	For example, for $x\in\tilde V$ we have \linebreak $\fm_x(\vA')=\unfrozen$ iff $\fm_{a\to x}(\vA')=\unfrozen$ for all $a\in\partial x$.
	Furthermore, the fixed point property from \Lem~\ref{eqLemma_standard1} shows that \whp\ $\fm_{a\to x}(\vA')=\frozen$ iff $y\in\cF(\vA)$ for all $y\in\partial a\setminus\tilde V$.
	Since the variables $y$ are chosen randomly and independently, we see that
	$$ \pr[\fm_{a\to x}(\vA')=\frozen\mid\vA]=\pr\brk{\Po(d(1-f(\vA)))=0}+o(1)=\exp(-d(1-f(\vA)))+o(1)=\hat f(\vA)+o(1) \mbox{ \quad \whp\ }$$ 
	Because $x$ has a total of $\Po(d)$ independent adjacent checks, we obtain \eqref{eqCor_stats} for $\s=\unfrozen$; the cases $\s=\frozen$ and $\s=\fu$ are analogous.
	Finally, the identity $f(\vA)=\phi_d(f(\vA))+o(1)$ \whp\ follows from Fact~\ref{prop:representative}, \eqref{eqLemma_standard2} and \eqref{eqCor_stats} by summing on $\ell_{\unfrozen\unfrozen}$.
\end{proof}

\begin{proof}[Proof of \Prop~\ref{lem:3fp}]
	Fix a small $\eps>0$ and let $U(\eps)=\{\alpha\in[0,1]: |\alpha-\alpha_*|\wedge|\alpha-\alpha_0|\wedge|\alpha-\alpha^*|>\eps\}$.
	Then \Lem~\ref{Lemma_contract} shows that there exists an integer $t>0$ such that $\abs{\phi_d^{\circ t}(\alpha)-\alpha_*}\wedge\abs{\phi_d^{\circ t}(\alpha)-\alpha^*}<\eps/2$ for all $\alpha\in U(\eps)$.
	Hence, 
	\begin{align}\label{eqlem:3fp_1}
		\abs{\alpha-\phi^{\circ t}_d(\alpha)}&>\eps/2 &&\mbox{for all $\alpha\in U(\eps)$.}
	\end{align}
	By contrast, \Cor~\ref{Cor_stats} shows that $\abs{f(\vA)-\phi_d(f(\vA))}=o(1)$ \whp\
	Since $\phi_d(\nix)$ is uniformly continuous on $[0,1]$, this implies that $\abs{f(\vA)-\phi_d^{\circ t}(f(\vA))}=o(1)$ \whp\
	Hence, \eqref{eqlem:3fp_1} shows that $\pr\brk{f(\vA)\in U(\eps)}=o(1)$.
	Because this holds for arbitrarily small $\eps>0$, the assertion follows.
\end{proof}

\section{The unstable fixed point}\label{Sec_dd}

\noindent
\Prop~\ref{lem:3fp} shows that $f(\vA)$ is close to one of the fixed points of the function $\phi_d$ \whp\
The aim in this section is to prove \Prop~\ref{Prop_dd} by using the ``hammer and anvil''
strategy described in \Sec~\ref{Sec_formalheuristic} to rule out the unstable fixed point $\malpha$.
The proof is subtle and requires three steps.
First we show that a random $\vx\in\ker\vA$ sets about half the unfrozen variables to one.
Indeed, even if we weight the variable nodes by their degrees the overall weight of the one-entries comes to about half \whp\
Therefore, \eqref{eqGaoSimple} implies that $\ker\vA$ contains $2^{\Phi_d(\alpha_*)n+o(n)}$ such balanced vectors \whp\
This is the ``anvil'' part of the argument.

The ``hammer'' part consists of the next two steps showing that the existence of that many balanced solutions is actually unlikely if $f(\vA)\sim\malpha$.
We proceed by way of a sophisticated moment computation.
Specifically, we estimate the number of fixed points of the operator from \eqref{eqSimpleWP1}--\eqref{eqSimpleWP2} that mark about \linebreak $\malpha n$ variable nodes unfrozen as per \eqref{eqMarks1}.
This expectation turns out to be of order $\exp(o(n))$.
Subsequently we compute the expected number of actual balanced solutions compatible with such a WP fixed point.\linebreak
The answer turns out to be $2^{\Phi_d(\malpha)n+o(n)}$.
Since $\Phi_d(\malpha)<\Phi_d(\alpha_*)=\max_\alpha\Phi_d(\alpha)$, we conclude that a random matrix with $f(\vA)\sim\malpha$ would have far fewer ``balanced'' vectors in its kernel than the anvil part of the argument demands.
Consequently, the event $f(\vA)\sim\malpha$ is unlikely.

\subsection{Degree-weighted solutions}
Let us now carry this strategy out in detail.
A vector $x\in\ker\vA$ is called {\em $\delta$-balanced} if
\begin{align*}
	\abs{\sum_{v\notin\cF(\vA)}d_{\vA}(v)\bc{\vecone\cbc{x_v=1}-1/2}}<\delta n.
\end{align*}
The following observation is a simple consequence of \Prop~\ref{Prop_pin}.

\begin{lemma}\label{lem_balanced}
	\Whp\ the random matrix $\vA$ has $2^{\Phi_d(\alpha_*)n+o(n)}$ many $o(1)$-balanced solutions.
\end{lemma}
\begin{proof}
	Since \eqref{eqGaoSimple} and \Prop~\ref{Prop_theta=1} show that $\nul\vA\sim\Phi_d(\alpha_*)n$ \whp, it suffices to prove that
	\whp\ over the choice of $\vA$, a uniformly random $\vx\in\ker\vA$ is $o(1)$-balanced \whp\ over the choice of $\vx$.
	To see this, fix any integer $\ell>0$.
	\Prop~\ref{Prop_pin} implies together with \Prop~\ref{Prop_cut} that \whp\ over the choice of $\vA$,
	the distribution of a uniformly random $\vx\in\ker\vA$ is $o(1)$-extremal.
	Moreover, Fact~\ref{fact_ker} shows that the event $\{\vx_v=1\}$ has probability $1/2$ for all $v\not\in\cF(\vA)$.
	Therefore, the definition \eqref{eqCutMetric} of the cut metric implies that for any $\ell \in \NN$, \whp\ over the choice of $\vA$ we have
	\begin{align}\label{eqlem_balanced1}
		\Erw\brk{\abs{\sum_{v\not\in\cF(\vA)}\vecone\{d_{\vA}(v)=\ell\}\bc{\vecone\cbc{\vx_v=1}-\frac{1}{2}}}\mid\vA}=o(n).
	\end{align}
	To see that this is true, observe that in order for~\eqref{eqlem_balanced1} to fail,
		setting $I\subset [n]\setminus\cF(\vA)$ to be the subset of (indices of) unfrozen variables
		of degree $\ell$ in $G(\vA)$, there must be some constant $\gamma>0$ such that
		$|I|\ge \gamma n$ and furthermore
		$$
		\Pr\bc{\abs{\sum_{i \in I} \vx_i - \frac{|I|}{2}} \ge \gamma n \mid \vA}\ge 2\gamma.
		$$
		Suppose that $\Pr\bc{\sum_{i \in I} \vx_i  \ge \frac{|I|}{2} + \gamma n \mid \vA}\ge \gamma$ (the proof
		when the corresponding lower tail has probability at least $\gamma$ is similar);
		we will show that this property is incompatible with the distribution $\mu$ of $\vx$ being $o(1)$-extremal.
		For let $(\vsigma,\vtau)$ be any coupling of $(\mu,\bar\mu)$.
		We know that $\bar \mu$ has probability $2^{-|I|}$ on every possible assignment within $I$,
		so let us set $U_1\subset \field^n$ to be the set of vectors that
		assign at least $\frac{|I|}{2}+\gamma n$ many $1$s on $I$, set
		$U_2\subset \field^n$ to be the set of vectors that assign at most $\frac{|I|}{2}+\gamma n/2$ many $1$s on $I$,
		and let $U := U_1\times U_2$.
		Observe that
		$$
		\pr\brk{(\vsigma,\vtau)\in U} \ge \pr\brk{\vsigma \in U_1}-\pr\brk{\vtau \notin U_2} \ge \gamma - o(1)\ge \gamma/2.
		$$
		It follows that
		\begin{align*}
			\abs{\sum_{i\in I}\pr\brk{(\vsigma,\vtau)\in U,\vsigma_i=1}-\pr\brk{(\vsigma,\vtau)\in U, \vtau_i=1}}
			& \ge \frac{\gamma n}{2}\cdot \pr\brk{(\vsigma,\vtau)\in U} \ge \frac{\gamma^2 n}{2}.
		\end{align*}
		Since this is true for any coupling $(\vsigma,\vtau)$ of $(\mu,\bar\mu)$, the definition~\eqref{eqCutMetric} of the cut metric
		gives \linebreak $\cutm(\mu,\bar \mu) \ge \gamma^2/2 = \Theta(1)$, and therefore $\mu$ is not $o(1)$-extremal,
		which we know can only happen with probability $o(1)$ over the choice of $\vA$.

	As~\eqref{eqlem_balanced1} is true for every fixed $\ell$ \whp\ and the $\Bin(n,d/n)$ degree distribution of $G(\vA)$ has sub-exponential tails, the assertion follows from \eqref{eqlem_balanced1} by summing on $\ell$.
\end{proof}

\subsection{Counting WP fixed points}
Proceeding to the next step of our strategy, we now estimate the expected number of approximate WP fixed points that leave about $\malpha n$ variables unfrozen.
We call such fixed points $\malpha$-covers.
The precise definition, in which we condition on the degree sequence $d_{\vA}$ of $G(\vA)$, reads as follows.
For the rest of this section, let us fix a further parameter $\explicit = \explicit(n) \xrightarrow{n\to \infty} \infty$,
	which grows arbitrarily slowly. 

\begin{definition}\label{def_cover}
	Given $d_{\vA}$ let
	\begin{align*}
		\fV=\bigcup_{i=1}^n\cbc{v_i}\times[d_{\vA}(v_i)]&&\mbox{and}&&\fC=\bigcup_{i=1}^n\cbc{a_i}\times[d_{\vA}(a_i)]
	\end{align*}
	be sets of variable/check clones.
	An {\em $\alpha$-cover} is a pair $(\fm,\pi)$ consisting of a map $\fm:\fV\cup\fC\to\{\frozen,\unfrozen\}^2, $ \linebreak $\,(u,j)\mapsto(\fm_1(u,j),\fm_2(u,j))$ and a bijection $\pi:\fV\to\fC$ such that the following conditions are satisfied.
	\begin{description}
		\item[COV1] For all $i\in[n]$ and $j\in[d_{\vA}(v_i)]$ we have $\Big(\fm_1(\pi(v_i,j)),\fm_2(\pi(v_i,j))\Big)=\Big(\fm_2(v_i,j),\fm_1(v_i,j)\Big)$.
		\item[COV2] For all but at most $n/\explicit$ pairs $(v_i,j)$ with $i\in[n]$ and $j\in[d_{\vA}(v_i)]$ we have
		\begin{align*}
			\fm_2\bc{v_i,j}=\begin{cases}
				\frozen&\mbox{ if $\fm_1\bc{v_i,h}=\frozen$ for some $h\in[d_{\vA}(v_i)]\setminus\cbc j$},\\
				\unfrozen&\mbox{ otherwise.}
			\end{cases}
		\end{align*}
		\item[COV3] For all but at most $n/\explicit$ pairs $(a_i,j)$ with $i\in[n]$ and $j\in[d_{\vA}(a_i)]$ we have
		\begin{align*}
			\fm_2\bc{a_i,j}=\begin{cases}
				\frozen&\mbox{ if $\fm_1\bc{a_i,h}=\frozen$ for all $h\in[d_{\vA}(a_i)]\setminus\cbc j$},\\
				\unfrozen&\mbox{ otherwise.}
			\end{cases}
		\end{align*}
		\item[COV4] For any $\s\in\{\frozen,\star,\unfrozen\}$ and
		$\Lu = (\ell_{\unfrozen\unfrozen},\ell_{\unfrozen\frozen},\ell_{\frozen\unfrozen},\ell_{\frozen\frozen})\in\NN_0^4$ let
		\begin{align}\label{eqCOV5a}
			\fm(v_i)&=
			\begin{cases}
				\frozen&\mbox{ if $\fm_{1}(v_i,j)=\frozen$ for at least two $j\in[d_{\vA}(v_i)]$,}\\
				\fu&\mbox{ if $\fm_{1}(v_i,j)=\frozen$ for precisely one $j\in[d_{\vA}(v_i)]$,}\\
				\unfrozen&\mbox{ otherwise,}
			\end{cases}\quad\\
			\fm(a_i)&=\begin{cases}
				\frozen&\mbox{ if $\fm_{1}(a_i,j)=\frozen$ for all $j\in[d_{\vA}(a_i)]$,}\\
				\fu&\mbox{ if $\fm_{1}(a_i,j)=\frozen$ for all but precisely one $j\in[d_{\vA}(a_i)]$,}\\
				\unfrozen&\mbox{ otherwise,}
			\end{cases}\qquad
			\label{eqCOV5b}\\
			\vDelta(\s,\Lu{\tiny })&=\sum_{i=1}^n\vecone\cbc{\fm\bc{v_i}=\s}\prod_{\x,\y\in\{\unfrozen,\frozen\}}\vecone\cbc{\abs{\cbc{j\in[d_{\vA}(v_i)]:\fm_1(v_i,j)=\x, \; \fm_2(v_i,j)=\y}}=\ell_{\x \y}},			\label{eqCOV5c}\\
			\vGamma( \s,\Lu)&=
			\sum_{i=1}^n\vecone\cbc{\fm(a_i)=\s}\prod_{\x,\y \in\{\unfrozen,\frozen\}}\vecone\cbc{\abs{\cbc{j\in[d_{\vA}(a_i)]:\fm_1(a_i,j)=\x,\;\fm_2(a_i,j)=\y}}=\ell_{\x \y}}.			\label{eqCOV5d}
		\end{align}
		Then with $\bDelta(\nix),\bGamma(\nix)$ from \eqref{eqCor_stats11}--\eqref{eqCor_stats16} we have
		\begin{align}\label{eqCOV5punch}
			\vDelta(\s,\Lu)&=n\bDelta(1-\alpha,\s,\Lu) \; \pm \; n/\explicit,&
			\vGamma(\s,\Lu)&=n\bGamma(\alpha,\s,\Lu) \; \pm \; n/\explicit.
		\end{align}
	\end{description}
	We also define an \emph{$\alpha$-semi-cover to consist just of the map
		$\fm:\fV\cup\fC\to\{\frozen,\unfrozen\}^2,\,(u,j)\mapsto(\fm_1(u,j),\fm_2(u,j))$ satisfying {\bf COV2}--{\bf COV4}.}
\end{definition}

The intuition behind this definition is as follows. The sets $\fV,\fC$ represent clones of variable/check nodes, which
	will be used to model half-edges in the configuration model. The map $\fm$ will represent an assignment
	of messages to these half-edges, where each half-edge receives an incoming message given by $\fm_1$
	and an outgoing message given by $\fm_2$. Meanwhile, the bijection $\pi$ describes which half-edges will be matched together.

The condition {\bf COV1} ensures that the matching of half-edges is consistent in the sense that,
	say, a half-edge with incoming message $\frozen$ and outgoing message $\unfrozen$ must be matched with a half-edge
	with incoming message $\unfrozen$ and outgoing message $\frozen$.

Conditions~{\bf COV2} and~{\bf COV3} state that, in most cases, the outgoing message along a half-edge is exactly
	the one that would be produced by applying the WP update rule to the incoming messages along the other half-edges at this node.
	Thus these conditions ensure that, once the appropriate graph has been produced, the messages represent an approximate
	fixed point of WP.

Turning to the most complicated of the conditions,~{\bf COV4}, equations~\eqref{eqCOV5a} and~\eqref{eqCOV5b} extend
	$\fm$ to include not just messages on the half-edges, but also marks on the variables and checks 
	similar to those generated from the standard messages in~\eqref{eqMarks1}
	and~\eqref{eqMarks2}.
	Subsequently, $\vDelta(\s,\Lu)$ and $\vGamma(\s,\Lu)$ count the number of variables and checks respectively with mark $\s$
	and with numbers of messages along incident half-edges consistent with $\Lu$, thus describing a generalised degree sequence.
	Finally,~\eqref{eqCOV5punch} states that this generalised degree distribution is consistent with what one would heuristically expect
	if the incoming messages at a variable are $\frozen$ with probability $1-\alpha$ independently
	while the incoming messages at a check are $\frozen$ with probability $\alpha$ independently.
	The relation between these two probabilities is motivated by~\eqref{eqmalphaeq}, since we will be focussing
	on the case when $\alpha=\malpha$.

Let $\fZ(\alpha)$ be the number of $\alpha$-covers.
The main result in this section is the proof of the following bound.

\begin{proposition}\label{Prop_covers}
	For any $d>\eul$ \whp\ over the choice of the degree sequence $d_{\vA}$ we have
	\begin{align*}
		\frac{\fZ(\malpha)}{(dn)!\prod_{i=1}^nd_{\vA}(v_i)!d_{\vA}(a_i)!}&=\exp(o(n))\enspace.
	\end{align*}
\end{proposition}

The rest of this section is devoted to the proof of \Prop~\ref{Prop_covers}.
The following lemma decomposes $\fZ(\malpha)$ into a few factors that we will subsequently calculate separately.

\begin{lemma}\label{Lemma_covers2}
	\Whp\ over the choice of $d_{\vA}$ we have $\fZ(\malpha)=\exp(o(n))\fH^2\fL^2\fE$ where
	\begin{align*}
		\fH&=\binom n{n((\bDelta(1-\malpha,\s,\Lu))_{\s\in\{\frozen,\fu,\unfrozen\},\Lu\in\NN_0^4})},&
		\fL&=\prod_{\substack{\s\in\{\frozen,\fu,\unfrozen\}\\ \Lu = (\ell_{\unfrozen\unfrozen},\ell_{\unfrozen\frozen},\ell_{\frozen\unfrozen},\ell_{\frozen\frozen})\in\NN_0^4}}\binom{\ell_{\unfrozen\unfrozen}+\cdots+\ell_{\frozen\frozen}}{\ell_{\unfrozen\unfrozen},\ldots,\ell_{\frozen\frozen}}^{n\bDelta(1-\malpha,\s,\Lu)} \mbox{ and }\\
		\fE&=\bc{dn\malpha^2}!\bc{\bc{dn\malpha(1-\malpha)}!}^2\bc{dn(1-\malpha)^2}!.
	\end{align*}
\end{lemma}
\begin{proof}
	The first factor $\fH$ simply accounts for the number of ways of partitioning the $n$ variable nodes and the $n$ check nodes into the various types as designated by \eqref{eqCOV5c}--\eqref{eqCOV5d}.
	Since we need to select a type for each variable and check node, the number of possible designations actually reads
	\begin{align}\label{eqLemma_covers2_1}
		\binom n{n((\bDelta(1-\malpha,\s,\Lu))_{\s\in\{\frozen,\fu,\unfrozen\},\Lu\in\NN_0^4}}\binom n{n((\bGamma(\malpha,\s,\Lu))_{\s\in\{\frozen,\fu,\unfrozen\},\Lu\in\NN_0^4}}\exp(o(n));
	\end{align}
	the $\exp(o(n))$ error term accounts for the $n/\explicit$ error terms in~\eqref{eqCOV5punch}.
	But a glimpse at~\eqref{eqCor_stats11}--\eqref{eqCor_stats16} reveals that these two multinomial coefficients coincide.
	Hence, \eqref{eqLemma_covers2_1} is equal to $\fH^2\exp(o(n))$.
	Furthermore, the factor~$\fL$ accounts for the number of ways of selecting, for each variable/check node, the clones along which messages of the four types $\{\frozen,\unfrozen\}^2$ travel.
	Finally, $\fE$ counts the number of ways of matching up these clones so that {\bf COV2}--{\bf COV3} are satisfied.
	To be precise, since {\bf COV2--COV3} only provide asymptotic estimates rather than precise equalities, we incur an $\exp(o(n))$ error term; hence $\fZ(\alpha_0)=\exp(o(n))\fH^2\fL^2\fE$.
\end{proof}

\begin{lemma}\label{Lemma_covers3}
	We have  $\frac1n\log\fL=\fl'+\fl''+o(1)$, where 
	\begin{align*}
		\fl'&=\exp(-d)\sum_{\ell=0}^\infty\frac{d^\ell}{\ell!}\log(\ell!),&
		\fl''&=-\sum_{\substack{\s\in\{\frozen,\fu,\unfrozen\}\\ \Lu = (\ell_{\unfrozen\unfrozen},\ell_{\unfrozen\frozen},\ell_{\frozen\unfrozen},\ell_{\frozen\frozen})\in\NN_0^4}}\bDelta(1-\malpha,\s,\Lu)\log(\ell_{\unfrozen\unfrozen}!\ell_{\unfrozen\frozen}!\ell_{\frozen\unfrozen}!\ell_{\frozen\frozen}!).
	\end{align*}
\end{lemma}
\begin{proof}
	Choose $\vec{\s}\in\{\frozen,\fu,\unfrozen\}$ along with non-negative vector $\vLu\in\NN_0^4$ from the distribution
	\begin{align*}
		\pr\brk{\vec{\s}=\s,\vLu = \Lu}&=\bDelta(1-\malpha,\s,\Lu)&&(\s\in\{\frozen,\fu,\unfrozen\},\, \Lu\in\NN_0^4).
	\end{align*}
	Then due to {\bf COV4} we have
	\begin{align}\label{eqLemma_covers3_1}
		\frac1n\log\fL&=\Erw\brk{\log(\vell_{\unfrozen\unfrozen}+\cdots+\vell_{\frozen\frozen})!}-\Erw\brk{\log(\vell_{\unfrozen\unfrozen}!\cdots\vell_{\frozen\frozen}!)}+o(1)=\Erw\brk{\log(\vell_{\unfrozen\unfrozen}+\cdots+\vell_{\frozen\frozen})!}-\fl''+o(1).
	\end{align}
	Moreover, \eqref{eqCor_stats11}--\eqref{eqCor_stats13} show that $\vell_{\unfrozen\unfrozen}+\cdots+\vell_{\frozen\frozen}$ has distribution $\Po(d)$. Therefore, ${\Erw\brk{\log(\vell_{\unfrozen\unfrozen}+\cdots+\vell_{\frozen\frozen})!}=\fl'}$.
	Hence, the assertion follows from \eqref{eqLemma_covers3_1}.
\end{proof}

\begin{lemma}\label{Lemma_covers4}
	We have $\frac1n\log\fH=d\Big(1-\log(d)-\malpha\log\malpha-(1-\malpha)\log(1-\malpha)\Big)-\fl''.$
\end{lemma}
\begin{proof}
	This is a straightforward computation.
	For the sake of brevity we introduce $\newp(\lambda,i)=\pr\brk{\Po(\lambda)=i}$.
	Using Stirling's formula, we approximate $\fH$ in terms of entropy as
	\begin{align}\label{eqLemma_covers4_1}
		\frac1n\log\fH&=H((\bDelta(1-\malpha,\s,\Lu))_{\s\in\{\frozen,\fu,\unfrozen\},\Lu\in\NN_0^4})+o(1).
	\end{align}
	Depending on the choice of $\s\in\{\frozen,\fu,\unfrozen\}$, the definitions \eqref{eqCor_stats11}--\eqref{eqCor_stats13} of the $\bDelta(1-\malpha,\s,\Lu)$ constrain some of the values $\ell_{\unfrozen\unfrozen},\ldots,\ell_{\frozen\frozen}$ to be zero.
	Hence, using the identity \eqref{eqmalphaeq}, we can spell the right hand side of \eqref{eqLemma_covers4_1} out as
	\begin{align}
		H&((\bDelta(1-\malpha,\s,\Lu))_{\s\in\{\frozen,\fu,\unfrozen\},\Lu\in\NN_0^4})=-\sum_{\s,\Lu}\bDelta(1-\malpha,\s,\Lu)\log\bDelta(1-\malpha,\s,\Lu)\nonumber\\
		&=-\sum_{\ell_{\unfrozen\unfrozen}\geq0}\newp(d(1-\malpha),0)\newp(d\malpha,\ell_{\unfrozen\unfrozen})\log(\newp(d(1-\malpha),0)\newp(d\malpha,\ell_{\unfrozen\unfrozen}))\nonumber\\
		&\quad-\sum_{\ell_{\unfrozen\frozen}\geq0}\newp(d(1-\malpha),1)\newp(d\malpha,\ell_{\unfrozen\frozen})\log(\newp(d(1-\malpha),1)\newp(d\malpha,\ell_{\unfrozen\frozen}))\nonumber\\
		&\quad-\sum_{\ell_{\unfrozen\frozen}\geq0,\ell_{\frozen\frozen}\geq2}\newp(d(1-\malpha),\ell_{\frozen\frozen})\newp(d\malpha,\ell_{\unfrozen\frozen})\log(\newp(d(1-\malpha),\ell_{\frozen\frozen})\newp(d\malpha,\ell_{\unfrozen\frozen}))\nonumber\\ 
		&=d(1-\malpha)^2-(1-\malpha)\sum_{\ell_{\unfrozen\unfrozen}\geq0}\newp(d\malpha,\ell_{\unfrozen\unfrozen})\brk{\ell_{\unfrozen\unfrozen}\log(d\malpha)-d\malpha}\nonumber\\
		&\quad -d(1-\malpha)^2\log(d(1-\malpha)^2)-d(1-\malpha)^2\sum_{\ell_{\unfrozen\frozen}\geq0}\newp(d\malpha,\ell_{\unfrozen\frozen})\brk{\ell_{\unfrozen\frozen}\log(d\malpha)-d\malpha}\nonumber\\
		&\quad -\bc{\malpha-d(1-\malpha)^2}\sum_{\ell_{\unfrozen\frozen}\geq0}\newp(d\malpha,\ell_{\unfrozen\frozen})\brk{\ell_{\unfrozen\frozen}\log(d\malpha)-d\malpha}\nonumber\\
		&\quad-\sum_{\ell_{\frozen\frozen}\geq2}\newp(d(1-\malpha),\ell_{\frozen\frozen})
		\brk{\ell_{\frozen\frozen}\log(d(1-\malpha))-d(1-\malpha)}-\fl''\nonumber\\
		&=-\fl''+d(1-\malpha)^2+d\malpha(1-\malpha)-d\malpha(1-\malpha)\log(d\malpha)\nonumber\\
		&\quad-d(1-\malpha)^2\log(d(1-\malpha)^2)+d^2\malpha(1-\malpha)^2-d^2\malpha(1-\malpha)^2\log(d\malpha)\nonumber\\
		&\quad+d(1-\malpha)-d(1-\malpha)\log(d(1-\malpha))+(1-\malpha)\log(1-\malpha)+d(1-\malpha)^2\log(d(1-\malpha)^2)  \nonumber\\
		&\qquad +d\malpha(\malpha-d(1-\malpha)^2)-d\malpha(\malpha-d(1-\malpha)^2)\log(d\malpha) \nonumber\\
		&=-\fl''-d\log d-d\malpha\log\malpha-d(1-\malpha)\log(1-\malpha)+d+(1-\malpha)\log(1-\malpha)+d(1-\malpha)^2. \label{eqLemma_covers4_2}	
	\end{align} 
	Since $1-\malpha=\exp(-d(1-\malpha))$, the assertion is immediate from \eqref{eqLemma_covers4_2}.
\end{proof}

\begin{lemma}\label{Lemma_covers5}
	\Whp\ over the choice of $d_{\vA}$ we have $\frac1n\log\frac{\fE}{(dn)!}=2d\malpha\log\malpha+2d(1-\malpha)\log(1-\malpha)$.
\end{lemma}
\begin{proof}
	This follows immediately from Stirling's formula.
\end{proof}

\begin{proof}[Proof of \Prop~\ref{Prop_covers}]
	The proposition is an immediate consequence of \Lem s~\ref{Lemma_covers2}--\ref{Lemma_covers5}.
\end{proof}

\subsection{Extending covers}
While in the previous section we just estimated the number of covers, here we also count actual solutions to the random linear system encoded by a cover.
The following definition captures assignments $\sigma$ that, up to $o(n)$ errors, comply with the frozen/unfrozen designations of a cover $(\fm,\pi)$ and also satisfy the checks, again up to $o(n)$ errors.
We extend $\sigma:\{v_1,\ldots,v_n\}\to\field$ to the set of $\fV$ of clones by letting $\sigma(v_i,j)=\sigma(v_i)$.
Recall that $\explicit$ was an arbitrarily slowly growing function.

\begin{definition}\label{Def_extension}
	An {$\alpha$-extension} consists of an $\alpha$-cover $(\fm,\pi)$ together with an assignment $\sigma:\{v_1,\ldots,v_n\}\to\field$ such that the following conditions are satisfied.
	\begin{description}
		\item[EXT1] We have 
		$\sum_{i=1}^n(1+d_{\vA}(v_i))\vecone\cbc{\sigma(v_i)=1, \; \fm(v_i)\neq\unfrozen}\le n/\explicit.$
		\item[EXT2] We have
		$\sum_{i=1}^nd_{\vA}(v_i)\vecone\cbc{\sigma(v_i)=1, \;\fm(v_i)=\unfrozen}=\frac12\sum_{i=1}^nd_{\vA}(v_i)\vecone\cbc{\fm(v_i)=\unfrozen} \pm n/\explicit.$
		\item[EXT3] We have
		$\sum_{i=1}^n\vecone\cbc{\sum_{j\in[d_{\vA}(a_i)]}\sigma(\pi(a_i,j))\neq0}\le n/\explicit$.
	\end{description}
	Further, we define an $\alpha$-semi-extension to consist of an $\alpha$-semi-cover $\fm$ (see Definition~\ref{def_cover})
		and an assignment $\sigma:\{v_1,\ldots,v_n\}\to\field$ satisfying {\bf EXT1} and~{\bf EXT2}. In other words, an $\alpha$-semi-extension consists of $\fm$ and $\sigma$ satisfying all the properties of an $\alpha$-extension which do not involve the bijection~$\pi$.
\end{definition}

The first condition {\bf EXT1} posits that, when weighted according to their degrees,
all but $o(n)$ variables that are deemed frozen under $\fm$ are set to zero under~$\sigma$.
{\bf EXT2} provides that about half the variables that ought to be unfrozen according to $\fm$ are set to one, if we weight variables by their degrees.
Finally, {\bf EXT3} ensures that all but $o(n)$ checks are satisfied.

Let $\fX(\alpha)$ be the total number of $\alpha$-extensions.
The main result of this section reads as follows.

\begin{proposition}\label{Prop_extensions}
	Let $d>\eul$.
	\Whp\ over the choice of the degree sequence $d_{\vA}$ we have
	\begin{align*}
		\frac{\fX(\malpha)}{(dn)!\prod_{i=1}^nd_{\vA}(v_i)!d_{\vA}(a_i)!}&=\exp(n\Phi_d(\malpha)+o(n)).
	\end{align*}
\end{proposition}

The following lemma summarises the key step toward the proof of \Prop~\ref{Prop_extensions}.
For a fixed $\fm$ let $\vec\pi$ be a uniformly random matching of the clones $\fV,\fC$ such that $(\fm,\vec\pi)$ is an $\malpha$-cover.
Given $\fm$ and $\sigma$, let $\fp(\fm,\sigma)$ be the probability
	that the random matching $\vec\pi$ results in $\sigma$ satisfying all but $o(n)$ checks.

\begin{lemma}\label{Lemma_extensions1}
	\Whp\ over the choice of $d_{\vA}$,
		given any $\malpha$-semi-extension $(\fm,\sigma)$,
		we have 
	$$\fp(\fm,\sigma)\leq2^{-|\{i\in[n]:\fm(a_i)=\unfrozen\}|+o(n)}.$$
\end{lemma}
\begin{proof}
	Given the $\malpha$-semi-cover $\fm$, the precise matching $\vec\pi$ of the frozen/unfrozen clones
	remains random subject to condition {\bf COV1}.
	We will expose this matching in two steps.
	First we expose the degree-weighted fraction of occurrences of frozen/unfrozen variables set to one.
	Specifically, let $\vr_{\unfrozen}\sim1/2$ be the precise degree-weighted fraction of occurrences of unfrozen variables that are set to zero under $\sigma$; in formulae,
	\begin{align}\label{eqLemma_extensions1_1}
		\vr_{\unfrozen}&=\frac{\sum_{i=1}^n\abs{\cbc{j\in[d_{\vA}(a_i)]:\fm_1(a_i,j)=\unfrozen,\,\sigma(\pi(a_i,j))=0}}}{\sum_{i=1}^n\abs{\cbc{j\in[d_{\vA}(a_i)]:\fm_1(a_i,j)=\unfrozen}}}.
	\end{align}
	Similarly, let $\vr_{\frozen}\sim1$ be the degree-weighted fraction of frozen clones set to zero:
	\begin{align}\label{eqLemma_extensions1_2}
		\vr_{\frozen}&=\frac{\sum_{i=1}^n\abs{\cbc{j\in[d_{\vA}(a_i)]:\fm_1(a_i,j)=\frozen,\,\sigma(\pi(a_i,j))=0}}}{\sum_{i=1}^n\abs{\cbc{j\in[d_{\vA}(a_i)]:\fm_1(a_i,j)=\frozen}}}.
	\end{align}
	
	Once we condition on $\vr_{\unfrozen},\vr_{\frozen}$, the precise matching of the various clones remains random.
	To study the conditional probability that $\sigma$ satisfies all but $o(n)$ checks, we set up an auxiliary probability space.
	To be precise, let $\vec\chi=\bc{\vec\chi_{ij}}_{i\in[n],j\in[d_A(a_i)]}$ be a random sequence of mutually independent field elements $\vec\chi_{ij}\in\field$ such that
	\begin{align}\label{eqLemma_extensions1_3}
		\pr\brk{\vec\chi_{ij}=0}&=\begin{cases}
			\vr_{\unfrozen}&\mbox{ if }\fm_1(a_i,j)=\unfrozen,\\
			\vr_{\frozen}&\mbox{ if }\fm_1(a_i,j)=\frozen.
		\end{cases}
	\end{align}
	Further, consider the events 
	\begin{align*}
		\cR&=\cbc{\sum_{i=1}^n\sum_{j=1}^{d_{\vA}(a_i)}\vecone\cbc{\vec\chi_{ij}=0,\,\fm_1(a_i,j)=\s}= \vr_{\s}\sum_{i=1}^n\sum_{j=1}^{d_{\vA}(a_i)}\vecone\cbc{\fm_1(a_i,j)=\s}\mbox{ for }\s\in\{\frozen,\unfrozen\}},\\
		\cS&=\cbc{\sum_{i=1}^n\vecone\cbc{\sum_{j=1}^{d_{\vA}(a_i)}\vec\chi_{ij}\neq0}=o(n)}.
	\end{align*}
	(Note that since $\vec \chi_{ij}\in \field$, implicitly the central sum in the definition of $\cS$ is over $\field$.)
	Then because the matching $\vec\pi$ of the clones is random subject to {\bf COV1} we obtain
	\begin{align}\label{eqLemma_extensions1_4}
		\fp(\fm,\sigma)=\Erw[\pr\brk{\cS\mid\cR,\vr_{\frozen},\vr_{\unfrozen}}].
	\end{align}
	
	Hence, we are left to calculate $\pr\brk{\cS\mid\cR,\vr_{\frozen},\vr_{\unfrozen}}$.
	Calculating the unconditional probabilities is easy.
	Indeed, the choice \eqref{eqLemma_extensions1_1}--\eqref{eqLemma_extensions1_2} of $\vr_{\unfrozen},\vr_{\frozen}$ and the definition \eqref{eqLemma_extensions1_3} of $\vec\chi$ and the local limit theorem for the binomial distribution ensure that
	\begin{align}\label{eqLemma_extensions1_5}
		\pr\brk{\cR}&=\Omega(1/n).
	\end{align}
	Furthermore, we claim that
	\begin{align}\label{eqLemma_extensions1_6}
		\pr\brk{\cS}&=2^{-|\{i\in[n]:\fm(a_i)=\unfrozen\}|+o(n)}.
	\end{align}
	Indeed, consider a check $a_i$ such that $\fm(a_i)=\unfrozen$.
	Then there exists $j\in[d_{\vA}(a_i)]$ such that $\fm_1(a_i,j)=\unfrozen$.
	Therefore, the choice \eqref{eqLemma_extensions1_1} of $\vr_{\unfrozen}$ ensures that the event $\vec\chi_{ij}\neq0$ occurs with probability $1/2+o(1)$.
	Similarly, if $\fm(a_i)\neq\unfrozen$, then by the choice of $\vr_{\frozen}$ the event $\vec\chi_{ij}\neq0$ has probability at most $o(d_{\vA}(a_i))$.
	Since the definition \eqref{eqLemma_extensions1_3} of the $\vec\chi_{ij}$ ensures that these events are independent for the different checks $a_i$, we obtain \eqref{eqLemma_extensions1_6}.
	Finally, combining \eqref{eqLemma_extensions1_4}--\eqref{eqLemma_extensions1_6} with Bayes' rule, we obtain
	\begin{align*}
		\fp(\fm,\sigma)&=\Erw\Big[\pr\brk{\cS\mid\cR,\vr_{\frozen},\vr_{\unfrozen}}\Big]=\Erw\Big[\pr\brk{\cS\mid\vr_{\frozen},\vr_{\unfrozen}}\cdot\pr\brk{\cR\mid\cS,\vr_{\frozen},\vr_{\unfrozen}}/\pr\brk{\cR\mid\vr_{\frozen},\vr_{\unfrozen}}\Big]
		\leq2^{-|\{i\in[n]:\fm(a_i)=\unfrozen\}|+o(n)},
	\end{align*}
	as desired.
\end{proof}

To complete the proof of \Prop~\ref{Prop_extensions} we combine \Lem~\ref{Lemma_extensions1} with the following statement about the numbers of variables/checks of the various types. Given $\s \in \{\frozen,\unfrozen\}$, let us define $\indunf{\s}:= \vecone\cbc{\s=\unfrozen}$.

\begin{lemma}\label{Lemma_covers1}
	\Whp\ over the choice of $d_{\vA}$, any $\malpha$-cover $(\fm,\pi)$ satisfies
	\begin{align}\label{eqLemma_covers1}
		\frac1{dn}\sum_{i=1}^n\sum_{j=1}^{d_{\vA}(v_i)}\vecone\cbc{\fm(v_i,j)=(\x,\y)}&\sim \malpha^{1+\indunf{\x}-\indunf{\y}}(1-\malpha)^{1-\indunf{\x}+\indunf{\y}}\qquad(\x,\y\in\cbc{\frozen,\unfrozen}),\\
		\frac1{dn}\sum_{i=1}^n\sum_{j=1}^{d_{\vA}(a_i)}\vecone\cbc{\fm(a_i,j)=(\x,\y)}&\sim \malpha^{1-\indunf{\x}+\indunf{\y}}(1-\malpha)^{1+\indunf{\x}-\indunf{\y}}\qquad(\x, \y\in\cbc{\frozen,\unfrozen}),\label{eqLemma_covers1a}\\
		\frac1n\sum_{i=1}^n\vecone\cbc{\fm(v_i)=\frozen}\sim \malpha-d(1-\malpha)^2,&\qquad
		\frac1n\sum_{i=1}^n\vecone\cbc{\fm(v_i)=\unfrozen}\sim 1-\malpha,\qquad
		\frac1n\sum_{i=1}^n\vecone\cbc{\fm(v_i)=\fu}\sim d(1-\malpha)^2,\label{eqLemma_covers2}\\
		\frac1n\sum_{i=1}^n\vecone\cbc{\fm(a_i)=\unfrozen}\sim \malpha-d(1-\malpha)^2,&\qquad
		\frac1n\sum_{i=1}^n\vecone\cbc{\fm(a_i)=\frozen}\sim 1-\malpha,\qquad 
		\frac1n\sum_{i=1}^n\vecone\cbc{\fm(a_i)=\fu}\sim d(1-\malpha)^2. \label{eqLemma_covers3}
	\end{align}
\end{lemma}
\begin{proof}
	We first claim that {\bf COV4} implies the estimate 
	\begin{align*}
		\frac1n\sum_{i=1}^n\sum_{j=1}^{d_{\vA}(v_i)}\vecone\cbc{\fm(v_i,j)=(\x,\y)}&\sim d\malpha^{\indunf{\x}}(1-\malpha)^{1-\indunf{\x}}\exp(-d\indunf{\y}(1-\malpha))(1-\exp(-d(1-\malpha)))^{1-\indunf{\y}}.
	\end{align*}
	Let us first give some non-rigorous intuition for where the formula comes from. Condition~{\bf COV4} heuristically
		states that at a given variable, each incoming message is $\frozen$ with probability $1-\malpha$ independently.
		Thus there are on average about $nd(1-\malpha)$ incoming messages of $\frozen$ and $nd\malpha$ incoming messages of $\unfrozen$,
		which leads to the term $d\malpha^{\indunf{\x}}(1-\malpha)^{1-\indunf{\x}}$ above.
		Regardless of the incoming message, in order for the corresponding
		outgoing message to also be $\frozen$ (in which case $\indunf{\y}=0$),
		assuming we are not in one of the few (by~{\bf COV1}) cases when the outgoing messages
		do not come from the incoming messages,
		there would need to be at least one further incoming message of $\frozen$ at the appropriate variable,
		which occurs with probability $\pr\brk{\Po(d(1-\malpha))\ge 1} = 1-\exp\bc{d(1-\malpha)}$.
		Therefore we also have that the probability that the outgoing message is $\unfrozen$ (the case when $\indunf{\y}=1$)
		is $\exp\bc{d(1-\malpha)}$. This gives the appropriate term involving exponentials.
	
	For a more formal proof, let us restrict ourselves the case when $(\x,\y)=(\frozen,\frozen)$ -- other cases are very similar.
		We have 
		\begin{align*}
			\frac{1}{n}\sum_{i=1}^n\sum_{j=1}^{d_{\vA}(v_i)}\vecone\cbc{\fm(v_i,j)=(\frozen,\frozen)}
			& = \frac{1}{n}\sum_{\s \in \{\frozen,\star,\unfrozen\}} \sum_{\Lu \in \mathbb{N}_0^4} \ell_{\frozen\frozen} \vDelta(\s,\Lu) \\
			& \sim \sum_{\s \in \{\frozen,\star,\unfrozen\}} \sum_{\Lu \in \mathbb{N}_0^4} \ell_{\frozen\frozen} \bDelta(1-\malpha,\s,\Lu) \\
			& = \sum_{\Lu \in \mathbb{N}_0^4} \ell_{\frozen\frozen} \bDelta(1-\malpha,\frozen,\Lu) \\
			& = \sum_{\substack{\ell_{\frozen\frozen}\ge 2 \\ \ell_{\unfrozen\frozen}\ge 0}} \ell_{\frozen\frozen}\pr\brk{\Po(d(1-\malpha))=\ell_{\frozen\frozen}}\pr\brk{\Po(d\malpha)=\ell_{\unfrozen\frozen}} \\
			& = \Erw\brk{\Po(d(1-\malpha))} - \pr\brk{\Po(d(1-\malpha))=1} \\
			& = d(1-\malpha)\bc{ 1- \exp\bc{-d(1-\malpha)}}
		\end{align*}
		Since in the case when $(\x,\y)=(\frozen,\frozen)$ we have $\indunf{\x}=\indunf{\y}=0$, the claimed estimate follows.

	From this estimate,
	using the identity \eqref{eqmalphaeq}, we obtain \eqref{eqLemma_covers1}.
	The second identity \eqref{eqLemma_covers1a} follows from \eqref{eqLemma_covers1} and~{\bf COV1}.
	We note also that a similar heuristic argument to the one above motivates the expression in~\eqref{eqLemma_covers1a}:
		indeed, we can also argue by symmetrical considerations, switching $\frozen$ and $\unfrozen$ and switching the roles
		of variable and check nodes (which leaves the WP update rule unchanged). Switching $\frozen$ and $\unfrozen$ also
		negates the effect of replacing $1-\malpha$ by $\alpha$, and therefore in the expression we only need to replace $\indunf{\x}$ and $\indunf{\y}$ by
		$1-\indunf{\x}$ and $1-\indunf{\y}$ respectively.
	
	Finally, equations~\eqref{eqLemma_covers2}--\eqref{eqLemma_covers3} follow from the identity $\malpha=1-\exp(-d(1-\malpha))$ and {\bf COV2} by summing on~$\Lu$.
	Once again, the motivation comes from the heuristic of the distributions of incoming messages: for example, to have $\fm(v_i)=\frozen$, we
		require at least two incoming messages of $\frozen$, which occurs with probability
		\begin{align*}
			\pr\brk{\Po(d(1-\malpha)\ge 2)} &
			= 1-\exp\bc{-d(1-\malpha)}(1+d(1-\malpha)) 
			\stackrel{{\scriptsize \eqref{eqmalphaeq}}}{=} 1- (1-\malpha)(1+d(1-\malpha)) = \malpha - d(1-\malpha)^2
		\end{align*}
		as required.
\end{proof}

\begin{proof}[Proof of \Prop~\ref{Prop_extensions}]
	\Whp\ over the choice of $d_{\vA}$, for any $\malpha$-semi-cover $\fm$,
		$2^{|\{i\in[n]:\fm(v_i)=\unfrozen\}|+o(n)}$ many 
		vectors $\sigma$ are present such that $(\fm,\sigma)$ is an $\malpha$-semi-extension (since we must set almost all of the frozen variables
		to zero, but most of the assignments on the unfrozen variables will result in {\bf EXT1} and {\bf EXT2} being satisfied). Therefore
	\Lem s~\ref{Lemma_extensions1} and~\ref{Lemma_covers1} (in particular the second approximation of~\eqref{eqLemma_covers2}
		and the first approximation of~\eqref{eqLemma_covers3}) imply that \whp\ over the choice of $d_{\vA}$,
	\begin{align}\label{eqProp_extensions_1}
		p(\fm,\sigma)&\leq2^{|\{i\in[n]:\fm(v_i)=\unfrozen\}|-|\{i\in[n]:\fm(a_i)=\unfrozen\}|+o(n)}
		\leq2^{n(1-2\malpha+d(1-\malpha)^2+o(1))}.
	\end{align}
	Further, using the identity \eqref{eqmalphaeq}, we verify that $1-2\malpha+d(1-\malpha)^2=\Phi_d(\malpha)$.
	Thus, the assertion follows from \eqref{eqProp_extensions_1} and \Prop~\ref{Prop_covers}.
\end{proof}

\begin{proof}[Proof of \Prop~\ref{Prop_dd}]
	We can generate a random Tanner graph $G(\vA)$ with a given degree sequence $d_{\vA}$ by way of the pairing model.
	Specifically, we generate a random pairing $\vec\pi$ of the sets $\fV,\fC$ of clones and condition on the event $\fS$ that the resulting graph $G(\vec\pi)$ is simple.
	\Whp\ over the choice of the degree sequence $d_{\vA}$ we have $\pr\brk{\fS\mid d_{\vA}}=\Omega(1)$; but in fact, for the purposes of the present proof the trivial estimate
	\begin{align}\label{eqProp_dd_1}
		\pr\brk{\fS\mid d_{\vA}}&=\exp(o(n))	&&\mbox{\whp}
	\end{align}
	suffices.
	Now, let $\cE$ be the event that $G(\vec\pi)$ has at least $2^{\Phi_d(\alpha^*)n -n/\explicit}$ many $\malpha$-extensions.
	Recall that \whp\ over the choice of $d_{\vA}$ there are $\bc{\sum_{i=1}^nd_{\vA}(v_i)}!=(dn)!\exp(o(n))$ possible matchings of the $2\bc{\sum_{i=1}^nd_{\vA}(v_i)}$ clones in total, and that each Tanner graph extends to $\prod_{i=1}^nd_{\vA}(v_i)!d_{\vA}(a_i)!$ pairings.
	Therefore, \Prop s~\ref{Prop_theta=1} and~\ref{Prop_extensions}, \eqref{eqProp_dd_1} and Markov's inequality show that \whp\ over the choice of $d_{\vA}$,
	\begin{align*}
		\pr\Big[\cE\mid \fS,d_{\vA}\Big]&\leq
		2^{-\Phi_d(\alpha^*)n+ n/\explicit}\frac{\fX(\alpha_0)}{(dn)!\prod_{i=1}^nd_{\vA}(v_i)!d_{\vA}(a_i)!}
		\leq
		2^{n(\Phi_d(\malpha)-\Phi_d(\alpha^*))+o(n)}=\exp(-\Omega(n)),
	\end{align*}
	or in other words \whp\ over the choice of $d_{\vA}$,
		\begin{equation}\label{eqProp_dd_2}
			\pr\Big[\vA \in \cE\mid d_{\vA}\Big] \leq \exp(-\Omega(n)).
		\end{equation}
		Note that we no longer condition on $\fS$ because the random model for $\vA$ automatically results in a simple Tanner graph,
		or more precisely, $G(\vA)$ has the same distribution as $G(\vec\pi)$ conditioned on this graph being simple.

	To complete the proof, assume that $\pr\brk{f(\vA)=\malpha+o(1)}>\eps$ for some $\eps>0$.
	Observe that Lemma~\ref{lem_balanced} shows that, \whp\ over the choice of $\vA$, there are $2^{\Phi_d(\alpha^*)n+o(n)}$
		many $o(1)$-balanced vectors in $\ker A$,
		which in the case when $f(\vA)=\malpha$ are simply $\malpha$-extensions by Lemma~\ref{Lemma_standard}.
		It follows that with probability at least $\eps/2$ over the choice of~$\vA$, there are at least $2^{\Phi_d(\alpha^*)n+o(n)}$
		many $\malpha$-extensions, i.e., $\pr\brk{\vA \in \cE}\ge \eps/2$.
		Hence with probability at least $\eps/4$ over the choice of $d_{\vA}$ we have $\pr\brk{\vA \in \cE \mid d_{\vA}}\ge \eps/4$.
		However, this statement directly contradicts~\eqref{eqProp_dd_2}.
\end{proof}

\section{Symmetry and correlation}\label{Sec_slush}

\noindent
The aim in this section is to prove \Prop~\ref{Prop_sym}, which states that \whp\ the numbers of variables and checks
in the slush are not almost equal.
Thus, we study the subgraph $G_\slush(\vA)$ induced on $V_\slush(\vA) \cup C_\slush(\vA)$. 
We use the notation $n_\slush := |V_\slush(\vA)|$ and $m_\slush := |C_\slush(\vA)|$.
We exploit the symmetry of the distribution of $\vA$ by considering the transpose of the matrix.
While symmetry automatically implies that events are equally likely for $\vA$ and $\vA^\top$, we would like to be able to deduce that the event $|V_\slush\bc{\vA}|-|C_\slush\bc{\vA}|\geq\omega$ occurs with probability asymptotically $1/2$ for some $\omega = \omega(n) \gg 1$.
The main step is to prove the following.

\begin{lemma}\label{lem:excludedmiddle}
	There exists some $\omega_0 \xrightarrow{n\to \infty} \infty$ such that
	\whp\ $|n_\slush-m_\slush| \geq \omega_0$.
\end{lemma}

As indicated above, \Prop~\ref{Prop_sym} follows from this lemma and symmetry considerations.
We first describe the symmetry property more explicitly.

\begin{lemma}\label{prop:marksym}
	For any matrix $A$ we have $V_\slush\bc{A^{\top}} = C_\slush\bc{A}$ and $C_\slush\bc{A^{\top}} = V_\slush\bc{A}$.
\end{lemma}
\begin{proof}
	We can show by induction on $t\in \NN$ that the messages at time $t$ in the Tanner graphs of $A,A^\top$ are symmetric.
	More precisely, the Tanner graphs are identical except that variable nodes become check nodes and vice versa.
	At time $0$ all messages are $\slush$ in both graphs, while it can be easily checked that the update rules remain identical if we switch checks and variables and also switch the symbols $\frozen$ and $\unfrozen$.
	Therefore, introducing
	\begin{align*}
		V_{\slush}(A,t)&=\Big \{v\in V(A):\big(\forall a\in\partial v:w_{a\to v}(A,t)\neq\frozen\big)\mbox{ and }\abs{\cbc{a\in \partial v:w_{a\to v}(A,t)=\slush}}\geq2\Big\},\\
		C_{\slush}(A,t)&=\Big\{a\in C(A):\big(\forall v\in\partial a:w_{v\to a}(A,t)\neq\unfrozen\big)\mbox{ and }\abs{\cbc{v\in\partial a:w_{v\to a}(A,t)=\slush}}\geq2\Big\}.
	\end{align*}
	we conclude that $V_\slush(A,t)=C_\slush(A^\top,t)$ and $C_\slush(A,t)=V_\slush(A^\top,t)$ for all $t$.
	Recalling \eqref{eqSlushIntro1}--\eqref{eqSlushIntro2}, we see that $V_\slush(A)=\bigcap_{t\geq0}V_\slush(A,t)$ and $C_\slush(A)=\bigcap_{t\geq0}C_\slush(A,t)$, whence the assertion follows.
\end{proof}

\begin{proof}[Proof of \Prop~\ref{Prop_sym}]
	We apply \Lem~\ref{prop:marksym} to deduce that
	$$
	\Pr\Big[|V_\slush(\vA)|-|C_\slush(\vA)| \ge \omega_0\Big] = \Pr\Big[\abs{C_\slush\bc{\vA^\top}} - \abs{V_\slush\bc{\vA^\top}}\ge \omega_0\Big]
	= \Pr\big[|C_\slush(\vA)|-|V_\slush(\vA)| \ge \omega_0\big],
	$$
	where for the second equality we used the fact that $\vA,\vA^\top$ have identical distributions.
	Furthermore Lemma~\ref{lem:excludedmiddle} implies that 
	$\Pr\big[|V_\slush(\vA)|-|C_\slush(\vA)| \ge \omega_0\big] + \Pr\big[|C_\slush(\vA)|-|V_\slush(\vA)| \ge \omega_0\big] = 1-o(1)$,
	and the desired statement follows.
\end{proof}

The proof strategy for Lemma~\ref{lem:excludedmiddle} is similar to (but rather simpler than) the standard approach to proving a local limit theorem:
we will show that $n_\slush-m_\slush$ is almost equally likely to hit any value in a range much larger than $\omega_0$,
and therefore the probability of hitting the much smaller interval $[-\omega_0,\omega_0]$ is negligible.
Our strategy involves a carefully constructed version of a switching method, in which we make small modifications
	to the graph in such a way that $n_\slush-m_\slush$ is perturbed in a controlled way -- more precisely, it will be altered by an
	additive term $\omega_2$, which will be chosen shortly (see~\eqref{eq:ExcludedMiddleHierarchy}) --
	in particular, we will have $\omega_0\ll \omega_2 \ll \sqrt{n}$.
	However, the modifications we make to the graph are small enough
	that the distribution of the modified graph is essentially indistinguishable from that of the original graph (since $\omega_2 \ll \sqrt{n}$).
	Moreover, since $\omega_2 \gg \omega_0$,
	it turns out that the resulting discrepancy $|n_\slush-m_\slush|$ is unlikely to be smaller than $\omega_0$.

We begin by estimating the sizes of some special sets of vertices.
Recall $\lambda$ from \eqref{eqlambdanu}.

\begin{definition}
	\begin{enumerate}[(i)]
		\item Let $R=R(\A)$ be the set of check nodes $a$ of degree two such that $w_{v\to a}(\vA)=\slush$ for all $v\in\partial a$.
		\item Let $S=S(\A)$ be the set of isolated variable nodes.
		\item Let $T=T(\A)$ be the set of check nodes $a$ of degree three such that $w_{v\to a}(\vA)=\slush$ for all $v\in\partial a$.
		\item Let $U=U(\A)$ be the set of variable nodes which have precisely two neighbours, both in $T$.
		\item Let
		\begin{align*}
			r & =r(\A)  := |R|/n, & 
			s & =s(\A)  := |S|/n, & 
			u & =u(\A)  := |U|/n, \\
			\bar r & := \frac{\exp\bc{-d}\lambda^2}{2},
			& \bar s & := \exp\bc{-d} ,
			& \bar u & := \left(\frac{\exp\bc{-d}\lambda^2}{2}\right) \cdot
			\left(\frac{\exp\bc{-d\alpha^*} \lambda^2/2}{1-\exp\bc{-\lambda}} \right)^2.
		\end{align*}
	\end{enumerate}
\end{definition}

\begin{lemma}\label{prop:TVdegs}
	\Whp
	\begin{align*}
		r =(1+o(1)) \bar r, \qquad
		s =(1+o(1)) \bar s, \qquad
		u = (1+o(1)) \bar u.
	\end{align*}
	In particular, there exists some $\omega_1 \to \infty$ such that
	$$
	r= \bc{1+o\bc{\frac{1}{\omega_1}}} \bar r, \qquad
	s= \bc{1+o\bc{\frac{1}{\omega_1}}} \bar s, \qquad
	u= \bc{1+o\bc{\frac{1}{\omega_1}}} \bar u.
	$$
\end{lemma}
\begin{proof}
	Since whether a node lies in each of these sets is a fact about its depth (at most) $2$ neighbourhood (with messages), by Lemma~\ref{lem:locallimit}, it is enough to look at the probabilities that $\cT_2$ (for $S,U$) and $\hat \cT_2$ (for $R$) have the appropriate structure.
	(Indeed, the statement for $S$ could be proved directly using a Chernoff bound and without appealing to Lemma~\ref{lem:locallimit}.)
	An elementary check verifies that these probabilities are $\bar r,\bar s, \bar u$, as appropriate.
\end{proof}

Let $1\ll \omega_1 \ll n^{1/2}$ be a function such that \Lem~\ref{prop:TVdegs} holds.
For the remainder of this section,
we will fix further functions $\omega_0,\omega_2$ such that
\begin{equation}\label{eq:ExcludedMiddleHierarchy}
	1 \ll \omega_0 \ll  \omega_1 \ll n^{1/2}
\end{equation}
and such that $\omega_2$ is chosen uniformly at random from the interval $[\omega_1/2,\omega_1]$
independently of $\vA$. In particular, we will prove Lemma~\ref{lem:excludedmiddle} with this $\omega_0$.

\begin{claim}\label{claim:sneakysubset}
	If $|U| = \Theta(n)$, then for all but $o\bc{\binom{|U|}{\omega_1}}$ subsets $U'\subset U$ of size $\omega_1$,
	no node has more than one neighbour in~$U'$. 
\end{claim}

\begin{proof}
	It is a simple exercise to check that if a subset $U'\subset U$ of size $\omega_1$ is chosen
	uniformly at random, then the expected number of nodes of $T$ for which two of their three neighbours are chosen
	to be in $U'$
	is $O\bc{|T|\omega_1^2/n^2}=o(1)$. Therefore by Markov's inequality, \whp\ this does not occur for any check node.
\end{proof}

We will use the following notation for the remainder of the section.
Given a Tanner graph $G$ and a set of variable nodes $W$, let $G\bck W$ denote the graph obtained from $G$ by deleting the set of edges incident~to~$W$. Note that this amounts to replacing the columns of the matrix corresponding to nodes of $W$ with $0$ columns.

\begin{claim}\label{claim:robustslush}
	Let $G$ be any Tanner graph and $U' \subset U(G)$ be any subset whose nodes lie at distance
	greater than $2$.
	Let $U''\subset U'$ be any subset of $U'$. Then $V_\slush\bc{G\bck{U''}} = V_\slush(G) \setminus U''$.
\end{claim}
In other words, removing $U''$ from $G$
does not have any knock-on effects on the slush.

\begin{proof}
	Let $G':= G\bck{U''}$, and let us run WP on both $G'$ and $G$ simultaneously, initialising
	with all messages being $\slush$. 
	We verify by induction on $t$ that the messages on the common edge set (those in $G'$) are identical in both processes, since a discrepancy can only enter at edges incident to a deleted edge (i.e., in $G\setminus G'$), but our choice of $U''\subset U$ is such that the messages emanating from the vertices of $T$ incident to $U''$ remain~$\slush$.
\end{proof}

For any $r,s,u$, let $\cG_{r,s,u}$ denote the class of graphs with the appropriate parameters,
i.e., with $r(G)=r$, with $s(G)=s$ and with $u(G)=u$,
and let
$$
\cG'_{r,s,u} = \cG'_{r,s,u;\omega_2} := \cG_{r',s',u'}, \qquad \qquad \mbox{where }
r':=r+\frac{2\omega_2}{n}, \qquad s':=s+\frac{\omega_2}{n}, \qquad u':=u-\frac{\omega_2}{n}.
$$
The intuition behind this definition is that if we delete a set $U''\subset U'$ of size $\omega_2$
to obtain $G'$,
then by Claim~\ref{claim:sneakysubset} no remaining messages are changed,
and therefore
\begin{itemize}
	\item $|R(G')| = |R(G)|+2\omega_2$ (for each vertex of $U''$, its two neighbours
	are moved into $R$);
	\item $|S(G')| = |S(G)| + \omega_2$ (the vertices of $U''$ are moved into $S$);
	\item $|U(G')| = |U(G)| - \omega_2$.
\end{itemize}
Furthermore, for any integer $\ell \in \ZZ$, let $\cG_{r,s,u}(\ell) \subset \cG_{r,s,u}$
be the subset consisting of graphs such  \linebreak that  $n_\slush-m_\slush=\ell$,
and similarly define $\cG'_{r,s,u}(\ell) \subset \cG'_{r,s,u}$
to be the subset consisting of graphs such \linebreak that  $n_\slush-m_\slush=\ell':=\ell-\omega_2$.

\begin{proposition}\label{prop:TVclasssizes}
	Suppose that we have parameters $r,s,u$ satisfying $$
	r= \bc{1+o\bc{\frac{1}{\omega_1}}} \bar r, \qquad
	s= \bc{1+o\bc{\frac{1}{\omega_1}}} \bar s, \qquad
	u= \bc{1+o\bc{\frac{1}{\omega_1}}} \bar u.
	$$
	Then for any integer $\ell \in \ZZ$ we have
	$\Pr\sqbc{G(\vA) \in \cG_{r,s,u}(\ell)} = (1+o(1))\Pr\sqbc{G(\vA) \in \cG'_{r,s,u}(\ell)}$.
\end{proposition}

\begin{proof}
	We construct an auxiliary bipartite graph $H$ with classes
	$\cG_{r,s,u}(\ell),\cG'_{r,s,u}(\ell)$,
	and with an edge between $G \in \cG_{r,s,u}(\ell)$ and $G' \in \cG'_{r,s,u}(\ell)$
	if $G'$ can be obtained from $G$ by deleting the edges incident to a set $U'' \subset U(G)$ of size $\omega_2$.
	(Note that by Claim~\ref{claim:robustslush},
	$G'$ satisfies $n_\slush'= n_\slush-\omega_2$ and $m_\slush'=m_\slush$, 
	 \linebreak so $n_\slush'-m_\slush'= (n_\slush-m_\slush) -\omega_2 =\ell-\omega_2=\ell'$, so such an edge is plausible.)
	
	By Claim~\ref{claim:sneakysubset} (and the fact that $\omega_2\le \omega_1$),
	every graph $G \in \cG_{r,s,u}(\ell)$ is incident to 
	$
	(1+o(1))\binom{un}{\omega_2}
	$
	\linebreak edges of $H$, since almost every choice of $\omega_2$ nodes from $U$ will result in a graph from $\cG'_{r,s,u}(\ell)$.
	
	On the other hand, given a graph $G' \in \cG'_{r,s,u}(\ell)$,
	we may construct a graph $G \in \cG_{r,s,u}(\ell)$ by picking any set of $\omega_2$ nodes within $S(G')$,
	any set of $2\omega_2$ nodes within $R(G')$ and adding $2\omega_2$ edges between them in the appropriate way.
	Thus we may double-count the edges of $H$ and obtain
	$$
	\abs{\cG_{r,s,u}(\ell)} \binom{un}{\omega_2} = (1+o(1)) \abs{\cG'_{r,s,u}(\ell)} \binom{sn}{\omega_2}\binom{rn}{2\omega_2} \frac{(2\omega_2)!}{2^{\omega_2}}.
	$$
	Since $r,s,u$ are very close to their idealised values $\bar r,\bar s,\bar u$, some
	standard approximations lead to
	\begin{equation} \label{eq:cGratio}
		\frac{\abs{\cG_{r,s,u}(\ell)}}{\abs{\cG'_{r,s,u}(\ell)}}
		= (1+o(1))\bc{\frac{\bar s\bar r^2n^2}{2\bar u}}^{\omega_2}.
	\end{equation}
	Substituting in the definitions of 
	$\bar r, \bar s, \bar u$, some elementary calculations and~\eqref{eqPreLambert} show that
	$
	\frac{\bar s \bar r^2}{2\bar u} = \frac{1}{d^2} = \frac{1}{p^2 n^2}.
	$
	Substituting this into~\eqref{eq:cGratio}, we obtain
	\begin{equation}\label{eq:cGratio2}
		\abs{\cG_{r,s,u}(\ell)} = (1+o(1)) \abs{\cG'_{r,s,u}(\ell)} p^{-2\omega_2}.
	\end{equation}
	
	On the other hand, let us observe that for any graph $G \in \cG_{r,s,u}(\ell)$
	and any graph $G'$ constructed from $G$ as above,
	$G'$ has precisely $2\omega_2$ edges fewer than $G$, and therefore
	\begin{equation}\label{eq:compareprobs}
		\Pr\sqbc{G(\A) = G'} = \Pr\sqbc{G(\A)=G} p^{-2\omega_2}(1-p)^{2\omega_2}  = (1+o(1))\Pr\sqbc{G(\A)=G} p^{-2\omega_2}.
	\end{equation}
	Combining~\eqref{eq:cGratio2} and~\eqref{eq:compareprobs},
	we deduce the statement of the proposition.
\end{proof}
\begin{proof}[Proof of Lemma~\ref{lem:excludedmiddle}]
	For any
	$(r,s,u) = (1+o(\omega_1^{-1}))(\bar r, \bar s, \bar u)$ and for any $G \in \cG_{r,s,u}$,
	pick an arbitrary subset $U''\subset U'$ of size $\omega_2$, where $U'$ is as in Claim~\ref{claim:sneakysubset}
	and let $G':= G\bck{U''}$.

		For any real number $x$, let us define $\Dev{x}$ to be the \emph{complement} of the event that
		$1-x \le \frac{r}{\bar r} , \frac{s}{\bar s} , \frac{u}{\bar u} \le 1+x$, i.e., the event that at least one of $r,s,u$
		differs significantly from its ideal value, the amount of deviation required being described by $x$. 
		We fix a further parameter $\omega_3$ satisfying $\omega_1\ll \omega_3\ll\sqrt{n}$,
		and 
		define the set
		$$
		\cS = \left\{(r,s,u) : 1-\frac{1}{\omega_3} \le \frac{r}{\bar r} , \frac{s}{\bar s} , \frac{u}{\bar u} \le 1+\frac{1}{\omega_3}\right\}.
		$$
		The choice of $\omega_3$ is such that for $(r,s,u)\in \cS$ we can apply Proposition~\ref{prop:TVclasssizes}.
	
	Observe that since $\omega_2 \le \omega_1 \le \omega_3 = o(\sqrt{n})$ we have that $\frac{\omega_2}{n} \ll \frac{1}{\omega_3}$,
		and therefore if
		$(r,s,u) \in \cS$, then
		$r+\frac{2\omega_2}{n} = \bc{1\pm \frac{2}{\omega_3}} \bar r$, and similarly for $s+ \frac{\omega_2}{n},u-\frac{\omega_2}{n}$.
	Using this fact, we obtain
	\begin{align*}
		\Pr\brk{|n_\slush-m_\slush| \le \omega_0}
		& \stackrel{\mbox{\scriptsize \phantom{P.\ref{prop:TVclasssizes}}}}{=} \bc{ \sum_{(r,s,u) \in \cS} \sum_{|\ell| \le \omega_0}\Pr\sqbc{G(\A)\in \cG_{r,s,u}(\ell)}} + O\bc{\pr\brk{\Dev{\frac{1}{\omega_3}}}} \\
		& \stackrel{\mbox{\scriptsize P.\ref{prop:TVclasssizes}}}{=} \bc{ \sum_{(r,s,u) \in \cS} \sum_{|\ell| \le \omega_0}\Pr\sqbc{G(\A)\in \cG'_{r,s,u}(\ell)}}
		+ O\bc{\pr\brk{\Dev{\frac{2}{\omega_3}}}}
		= \Pr\brk{|n_\slush-m_\slush + \omega_2| \le \omega_0} + o(1).
	\end{align*}
	However, since $\omega_2$ is chosen uniformly at random from the interval $[\omega_1/2,\omega_1]$, and in particular independently of $\vA$,
	we may change our point of view and say that
	$$
	\Pr\brk{|n_\slush-m_\slush + \omega_2| \le \omega_0} = \Pr\brk{\omega_2 = |m_\slush-n_\slush| \pm \omega_0} \le \frac{2\omega_0+1}{\omega_1/2} = o(1),
	$$
	as required.
\end{proof}

\section{Moments and expansion}\label{Sec_Prop_expansion}

\subsection{Overview}
In this section we prove \Prop~\ref{Prop_expansion}.
The proofs of the two statements of the proposition proceed via two rather different arguments.
First we show that it is unlikely that $|V_\slush(\vA)|-|C_\slush(\vA)|$ is large and at the same time $f(\vA)\sim\alpha^*$,
which would imply that the slush is almost entirely frozen.
The proof relies on the fact that $G(\vA)$ is unlikely to contain a moderately large, relatively densely connected subgraph.
Specifically, let $A$ be a matrix.
A {\em flipper} of $A$ is a set of variable nodes $U\subset V(A)$ such that for all $a\in \partial U$ we have $|\partial a\cap U|\geq2$.
Let $\fF_\eps(A)$ be the set of all flippers $U$ of $A$ of size $|U|\leq\eps n$.
Moreover, let $F_\eps(A)=\sum_{U\in\fF_\eps(A)}|U|$ be the total size of all flippers of $A$ which individually each have size at most $\eps n$.

\begin{lemma} \label{PropFlipping}
	For any $d>0$ there exists $\eps>0$ such that for any function $\omega=\omega(n)\gg1$ \linebreak we have~$F_\eps(\vA_\slush)\leq\omega$~\whp{}
\end{lemma}

\noindent
The proof of \Lem~\ref{PropFlipping} can be found in \Sec~\ref{Sec_PropFlipping}.
We will combine \Lem~\ref{PropFlipping} with the following statement to bound the size of $V_\slush{}(\vA)\setminus\cF(\vA_\slush{})$.

\begin{lemma}\label{Lemma_nulslush}
	The set $U=V_\slush(\vA)\setminus\cF(\vA_\slush{})$ is a flipper of $\vAs$ of size $|U|\geq|V_\slush{}(\vA)|-|C_\slush{}(\vA)|$ and $U\cap\cF(\vA)=\emptyset$.
\end{lemma}
\begin{proof}
	Clearly, $\nul\vA_\slush{}\geq|V_\slush{}(\vA)|-|C_\slush{}(\vA)|$ and thus
	\begin{align*}
		2^{|V_\slush(\vA)|-|C_\slush(\vA)|}\leq2^{\nul\vA_\slush}=\abs{\ker\vA_\slush}\leq\abs{\cbc{\newx\in \field^{|\Vs(\vA)|}:\forall v\in\cF(\vAs):\newx_v=0}}= 2^{|U|}.
	\end{align*}
	Hence, $|U|\geq|V_\slush{}(\vA)|-|C_\slush{}(\vA)|$.
	
	To show that $U$ is a flipper of $A$ we consider a variable node $v\in U$ and an adjacent check node $a\in\Cs(\vA)$.
	Assume for a contradiction that $\partial a\cap U=\cbc v$.
	Then for all other variable nodes $u\in\partial a\cap\Vs(\vA)$ we have $u\in\cF(\vAs)$.
	Hence, the only way to satisfy check $a$ is by setting $v$ to zero, too.
	Thus, $v\in\cF(\vAs)$, which contradicts $v\in U$.
	
	Finally, to show that $U\cap\cF(\vA)=\emptyset$ it suffices to prove that any vector $\newx_\slush\in\ker\vAs$ extends to a vector $\newx\in\ker\vA$.
	To see this we recall the peeling process \eqref{eqslushpeel} that yields $\Vs(\vA)$.
	Let us actually run this peeling process in two stages.
	In the first stage we repeatedly remove check nodes of degree one or less from $G(\vA)$:
	\begin{quote}
		while there is a check node of degree one or less, remove it along with \linebreak its adjacent variable (if any).
	\end{quote}
	The set of variable nodes that this process removes is precisely $\Vf(\vA)$ and we extend $\newx_\slush$ by setting $\newx_v=0$ for all $v\in\Vf(\vA)$.
	Next we repeatedly delete variable nodes of degree one or less:
	\begin{quote}
		while there is a variable node of degree one or less, remove it along with \linebreak its adjacent check (if any).
	\end{quote}
	Observe that since in the second stage we also delete the neighbours (if any) of the chosen variables,
		any checks which remain at the end of this stage have all the neighbours they had before the stage began.
		In particular, we cannot create any new check nodes of degree at most one, and do not have to repeat the first stage.
		Thus the outcome of this two-stage process is the same as that described in~\eqref{eqslushpeel}.

	Let $y_1,\ldots,y_\ell$ be the variable nodes that this process deletes, and suppose that they were deleted in this order.
	Then we inductively extend $\newx_\slush$ by assigning the variables in the reverse order $y_\ell,\ldots,y_1$ as follows.
	At the time $y_k$ was deleted, where $1\leq k \leq \ell$, this variable node either had no adjacent check node at all, in which case we define $\newx_{y_k}=0$, or there was precisely one adjacent check node $b_k$.
	In the latter case we set $\newx_{y_k}$ to the (unique) value that satisfies $b_k$ given the previously defined entries of $\newx$.
	The construction ensures that $\newx\in\ker\vA$.
\end{proof}

\noindent
Second, we bound the probability that $|C_\slush(\vA)|-|V_\slush(\vA)|$ is large and at the same time $f(\vA)\sim\alpha_*$. 
The proof of the following lemma, which we postpone to \Sec~\ref{Sec_Lemma_expectedSolutions}, is based on a delicate moment calculation.

\begin{lemma}\label{Lemma_expectedSolutions}
	For any $d>\eul$ there exists $\eps>0$ such that for any $\omega=\omega(n)\gg1$  we have
	\begin{align*}
		\pr\brk{|C_\slush(\vA)|-|V_\slush(\vA)|\geq\omega\mbox{ and }
			|V_\slush(\vA)\cap\cF(\vA)|<\eps n}&=o(1).
	\end{align*}
\end{lemma}

\begin{proof}[Proof of \Prop~\ref{Prop_expansion}]
	Fix a small enough $\eps>0$ and suppose that $\omega\to\infty$.
	To prove the first statement let $\cE=\{|V_\slush(\vA)|-|C_\slush(\vA)|\geq\omega\}$ and $\cE'=\cbc{F_\eps(\vA)<\omega}$.
	\Lem~\ref{Lemma_nulslush} shows that if the event $\cE \cap\cE'$ occurs, then the set $U=V_\slush(\vA)\setminus\cF(\vA_\slush{})$, being a flipper of size at least $\omega$ (by $\cE$), cannot be included in $\cF_\eps(\vA)$ (because\linebreak of $\cE'$) and therefore has size at least $\eps n$.
	Additionally, we have $U\cap\cF(\vA)=\emptyset$ while $U\subset V_\slush{}(\vA)\subset V(\vA)\setminus V_\unfrozen(\vA)$.
	Hence, \Prop~\ref{claim_littleintersection}  implies $f(\vA)\leq|V(\vA)\setminus V_\unfrozen(\vA)|/n+o(1)-\eps$. 
	Consequently, \Prop~\ref{Prop_frozen} and \Lem~\ref{PropFlipping} yield
	\begin{align*}
		\pr\brk{\cE\cap\cbc{f(\vA)>\alpha^*-\eps/2}}
		&\leq\pr\brk{\{F_\eps(\vA)>\omega\}\cup
			\cbc{|V(\vA)\setminus V_\unfrozen(\vA)|/n> \alpha^*+\eps/3}}=o(1).
	\end{align*}
	Thus, \Prop s~\ref{lem:3fp} and \ref{Prop_dd} show that $\pr\brk{\cE\cap\cbc{\abs{f(\vA)-\alpha_*}>\eps}}=o(1)$.
	
	With respect to the second statement, let $\cA=\{|C_\slush(\vA)|-|V_\slush(\vA)|\geq\omega\}$ and $\cA'=\cbc{|V_\slush(\vA)\cap\cF(\vA)|<\eps n}$.
	Then \Lem~\ref{Lemma_expectedSolutions} shows that
	\begin{align}\label{eqProp_expansion10}
		\pr\brk{\cA\cap\cA'}&=o(1).
	\end{align}
	Moreover, \Prop~\ref{Prop_frozen} and \eqref{eqLemma_WP1} show that
	\begin{align}\label{eqProp_expansion11}
		\pr\brk{\cbc{f(\vA)\leq\alpha_*+\eps/2}\setminus\cA'}&=o(1),
	\end{align}
	and the assertion is immediate from \eqref{eqProp_expansion10}, \eqref{eqProp_expansion11} and \Prop s~\ref{lem:3fp} and~\ref{Prop_dd}.
\end{proof}

\subsection{Proof of \Lem~\ref{PropFlipping}}\label{Sec_PropFlipping}
A {\em $(u,c,m)$-flipper of $\vAs$} consists of a set $U\subset V_\slush{}(\vA)$ of size $|U|=u$ whose neighbourhood $C=\partial U\cap C_\slush{}(\vA)$ has size $|C|=c$ such that the number the number of $U$-$C$-edges in $G_\slush{}(\vA)$ is equal to $m$.
Let $\vZ(u,c,m)$ be the number of $(u,c,m)$-flippers.
As a first step we deal with flippers whose average variable degree exceeds two.

\begin{claim}\label{Claim_flippper1}
	For any $d>0,\delta>0$ there exists $\eps>0$ such that
	\begin{align*}
		\Erw\brk{\sum_{U\in\fF_\eps(\vA)}|U|\vecone\cbc{\sum_{x\in U}|\partial x\cap C_\slush(\vA)|\geq(2+\delta)|U|}}&=o(1).
	\end{align*}
\end{claim}
\begin{proof}
	Recalling $p=d/n \wedge 1$, we write the simple-minded bound
	\begin{align}\label{eqClaim_flippper1_1}
		\Erw\brk{u\vZ(u,c,m)}&\leq u\binom nu\binom nc\binom{uc}m p^m;
	\end{align}
	here $\binom nu$ counts the number of choices for $U$, $\binom nc$ accounts for the number of  possible sets of $c$ check nodes, $\binom{uc}m$ bounds the number of bipartite graphs on the chosen variable and check sets, and $p^m$ bounds the probability that the chosen subgraph is actually contained in $G(\vA)$.
	We aim to bound the r.h.s.\ of \eqref{eqClaim_flippper1_1} subject to the constraints
	\begin{align}\label{eqClaim_flippper1_2}
		m\geq\max\cbc{2c,(2+\delta)u},&&1\leq u\leq\eps n&&\mbox{for a small enough $\eps>0$.}
	\end{align}
	We consider three separate cases.
	\begin{description}
		\item[Case 1: $c\leq u$] we estimate
		\begin{align}\label{eqClaim_flippper1_3}
			\binom nu\binom nc\binom{uc}m p^m&\leq \bcfr{\eul n}{u}^{2u}\bcfr{\eul uc d}{mn}^m\leq
			\bcfr{\eul n}{u}^{2u}\bcfr{\eul c d}{2n}^{(2+\delta)u}\leq \bc{\eul^{4+\delta}d^{2+\delta}}^u\bcfr{u}{n}^{\delta u}.
		\end{align}
		Combining \eqref{eqClaim_flippper1_1}--\eqref{eqClaim_flippper1_3}, we obtain
		\begin{align}\label{eqClaim_flippper1_4}
			\sum_{1\leq c\leq u\leq\eps n}\Erw\brk{u\vZ(u,c,m)}&\leq
			\sum_{1\leq u\leq\eps n}u^2\bc{\eul^{4+\delta}d^{2+\delta}}^u\bcfr{u}{n}^{\delta u}=o(1).
		\end{align}
		\item[Case 2: $u\leq c\leq 100u$] due to \eqref{eqClaim_flippper1_2} we obtain
		\begin{equation}\label{eqClaim_flippper1_5}
			\begin{split}
			\binom nu\binom nc\binom{uc}m p^m&
			\leq\bcfr{\eul n}{u}^{u}\bcfr{\eul n}{c}^c\bcfr{\eul u d}{2n}^c\bcfr{\eul u d}{2n}^{m/2}
			\leq\bcfr{\eul n}{u}^{u}\bcfr{\eul^2d}{2}^c\bcfr{\eul ud}{2n}^{u(1+\delta/2)}\\
			&\leq\bcfr{\eul^2d}{2}^{400 u}\bcfr{u}{n}^{\delta u/2}.
			\end{split}
		\end{equation}
		Combining \eqref{eqClaim_flippper1_1} and \eqref{eqClaim_flippper1_5}, we get
		\begin{align}\label{eqClaim_flippper1_5}
			\sum_{\substack{1 \leq u\leq\eps n\\ u\leq c\leq100 u}}\Erw\brk{u\vZ(u,c,m)}&\leq\sum_{1\leq u\leq\eps n}100u^2\bcfr{\eul^2d}{2}^{400 u}\bcfr{u}{n}^{\delta/2}=o(1).
		\end{align}
		\item[Case 3: $100u \leq c\leq n$] the condition \eqref{eqClaim_flippper1_2} yields
		\begin{align*}
			\binom nu\binom nc\binom{uc}m p^m\leq\bcfr{100\eul n}{c}^{1.1c}\bcfr{\eul d u}{n}^{2c}\leq\bcfr{\eul d u}{n}^{c/2}.
		\end{align*}
		Hence,
		\begin{align}\label{eqClaim_flippper1_6}
			\sum_{\substack{1 \leq u\leq\eps n\\ 100 u\leq c\leq n}}\Erw\brk{u\vZ(u,c,m)}&\leq\sum_{1\leq u\leq\eps n}u\sum_{100 u\leq c\leq n}\bcfr{\eul d u}{n}^{c/2}
			\leq\sum_{1\leq u\leq\eps n}u\bcfr{\eul d u}{n}^{u}=o(1).
		\end{align}
	\end{description}
	Finally, the assertion follows from \eqref{eqClaim_flippper1_4}, \eqref{eqClaim_flippper1_5} and \eqref{eqClaim_flippper1_6}.
\end{proof}

\noindent
Complementing Claim~\ref{Claim_flippper1}, we now estimate the sizes of flippers of average check degree greater than two.

\begin{claim}\label{Claim_flippper2}
	For any $d>0,\delta>0$ there exists $\eps>0$ such that
	\begin{align*}
		\Erw\brk{\sum_{U\in\fF_\eps(\vA)}|U|\vecone\cbc{\sum_{a\in\partial U\cap C_\slush(\vA)}|\partial a\cap U|\geq(2+\delta)|C|}}&=o(1).
	\end{align*}
\end{claim}
\begin{proof}
	The proof is rather similar to the proof of the previous claim, except that we swap the roles of $u$ and $c$.
	Once more we start from the naive bound \eqref{eqClaim_flippper1_1}, but this time $m$ satisfies  $m\geq \max\cbc{2u,(2+\delta)c}$ and $1\leq u\leq\eps n$.
	\begin{description}
		\item[Case 1: $u\leq c$] we have
		\begin{align}\label{eqClaim_flippper2_3}
			\binom nu\binom nc\binom{uc}m p^m&\leq
			\bcfr{\eul n}{c}^{2c}\bcfr{\eul u d}{2n}^{(2+\delta)c}\leq \bc{\eul d}^{5c}\bcfr{u}{n}^{\delta c}.
		\end{align}
		\item[Case 2: $c\leq u\leq 100c$] we estimate
		\begin{equation}\label{eqClaim_flippper2_5}
			\begin{split}
			\binom nu\binom nc\binom{uc}m p^m&
			\leq\bcfr{\eul n}{u}^{u}\bcfr{\eul n}{c}^c\bcfr{\eul c d}{2n}^u\bcfr{\eul c d}{2n}^{m/2}
			\leq\bcfr{\eul n}{c}^{c}\bcfr{\eul^2d}{2}^u\bcfr{\eul cd}{2n}^{c(1+\delta/2)}\\
			&\leq\bcfr{100\eul^2d}{2}^{u}\bcfr{u}{n}^{\delta u/200}.
			\end{split}
		\end{equation}
		\item[Case 3: $100c \leq u$] we have
		\begin{align}\label{eqClaim_flippper2_6}
			\binom nu\binom nc\binom{uc}m p^m\leq\bcfr{\eul n}{u}^{1.1u}\bcfr{\eul d c}{n}^{2u}\leq\bcfr{\eul d u}{n}^{c/2}.
		\end{align}
	\end{description}
	Summing \eqref{eqClaim_flippper2_3}, \eqref{eqClaim_flippper2_5} and \eqref{eqClaim_flippper2_6} on $u,c,m$ such that $m\geq(2+\delta)c$, we obtain  ${\sum_{u,c,m}\Erw\brk{uZ(u,c,m)}=o(1).}$
\end{proof}

\noindent
Finally, we need to deal with flippers of average variable and constraint degree about two.  

\begin{claim}\label{Claim_Flipper3}
	For any $d>\eul$ there exists $\eps>0$ such that for any $\omega=\omega(n)\gg1$ we have
	\begin{align*}
		\pr\brk{\sum_{U\in\fF_\eps(\vA)}|U|\vecone\cbc{\sum_{x\in U}|\partial x\cap C_\slush(\vA)|\leq(2+\eps)|U|,
				\sum_{a\in\partial U\cap C_\slush(\vA)}|\partial a\cap U|\leq(2+\eps)|C|}>\omega}&=o(1).
	\end{align*}
\end{claim}
\begin{proof}
	Choose $L=L(d)>0$ sufficiently large and subsequently $\eps>0$ sufficiently small.
	Moreover, for a vertex $u$ of $G_\slush(\vA)$ let $d_\slush(u)$ signify the degree of $u$ in $G_\slush(\vA)$.
	Further, with $\nu,\lambda$ from \eqref{eqlambdanu} let $\cD$ be the event that the graph $G_\slush{}(\vA)$ enjoys the following four properties.
	\begin{description}
		\item[D1] $|V_\slush(\vA)|= (\nu+o(1)) n$ and $|C_\slush(\vA)|= (\nu+o(1)) n$.
		\item[D2] For any $2\leq\ell\leq L$ we have
		$\sum_{x\in V_\slush(\vA)}\vecone\cbc{d_\slush(x)=\ell}= \pr\brk{\Po_{\geq2}(\lambda)=\ell}\nu n + o(n).$
		\item[D3] For any $2\leq\ell\leq L$ we have
		$\sum_{a\in C_\slush(\vA)}\vecone\cbc{d_\slush(a)=\ell}= \pr\brk{\Po_{\geq2}(\lambda)=\ell}\nu n + o(n).$
		\item[D4] The bounds from \eqref{eqLemma_tails} hold for the degree sequence of $G(\vA)$.
	\end{description}
	Then \Prop~\ref{Prop_slush} and \Lem~\ref{Lemma_tails} imply that
	\begin{align}\label{eqClaim_Flipper30}
		\pr\brk{\cD}=1-o(1).
	\end{align}
	
	We aim to count $(u,c,m)$-flippers $U\subset V_\slush{}(\vA)$ with neighbourhoods $C=\partial U\cap C_\slush{}(\vA)$ of size $|C|=c$ such that
	\begin{align}\label{eqClaim_Flipper31}
		m&=\sum_{x\in U}|\partial x\cap C|=\sum_{a\in C}|\partial a\cap U|\leq(2+\eps)(u\wedge c),&\mbox{and, of course,}&&
		&&\min_{a\in C}|\partial a\cap U|\geq2.
	\end{align}
	To estimate the number $\vZ(u,c,m)$ we recall from \Prop~\ref{Prop_slush} that the graph $G_\slush(\vA)$ is uniformly random given the degrees.
	Therefore, according to \Lem~\ref{Lemma_contig} it suffices to bound  the number of $(u,c,m)$-flippers of a random graph chosen from the pairing model with the same degree sequence.
	Thus, let $\vec\Gamma_\slush$ be a random perfect matching of the complete bipartite graph on the vertex sets
	\begin{align*}
		\cV&=\bigcup_{v\in V_\slush(\vA)}\cbc v\times[d_\slush(v)],&\cC&=\bigcup_{a\in C_\slush(\vA)}\cbc a\times[d_\slush(a)].
	\end{align*}
	Further, let $\cG_\slush$ be the multigraph obtained from $\vec\Gamma_\slush$ by contracting the clones $\cbc v\times[d_\slush(v)]$ and $\cbc a\times[d_\slush(a)]$ of the variable and constraint nodes into single vertices for all $v\in V_\slush(\vA)$, $a\in C_\slush(\vA)$.
	Due to \eqref{eqClaim_Flipper30} it suffices to establish the bound
	\begin{align}\label{eqClaim_Flipper32}
		\sum_{u,c,m:1\leq u\leq\eps n}u\Erw\brk{\vZ(u,c,m)\mid\cD}&=O(1).
	\end{align}
	
	To prove \eqref{eqClaim_Flipper32} we first count viable choices of $U$.
	Since \eqref{eqClaim_Flipper31} implies that $2u\leq m\leq(2+\eps)u$, no more than $\delta u$ of the vertices in the set $U$ have degree greater than two.
	Further, {\bf D1} and {\bf D2} show that there are no more than
	\begin{align}\label{eqClaim_Flipper33}
		\binom{(\nu+o(1))n}u\binom{u}{\eps u}\bcfr{\lambda^2+o(1)}{2(\exp(\lambda)-\lambda-1)}^{(1-\eps)u}
		\leq \bcfr{\eul L}{\eps}^{\eps u}\bcfr{\eul(\nu+o(1))n}{u}^u\bcfr{\lambda^2+o(1)}{2(\exp(\lambda)-\lambda-1)}^{u}
	\end{align}
	such sets $U$.
	
	By a similar token, most check nodes in $C$ have precisely two neighbours in $U$.
	Thus, we estimate the number of choices of $C\subset C_\slush(\vA)$ of size $c$ along with a set $\fC$ of $m$ clones of these checks as follows.
	Summing on all vectors $\vk=(k_1,\ldots,k_c)$ of integers $k_i\geq2$ with $\sum_{i}k_i=m$ and on all sequences $(b_1,\ldots,b_c)\in C_\slush{}(\vA)^c$, we obtain the bound
	\begin{align}\label{eqClaim_Flipper34a}
		\frac1{c!}\sum_{b_1,\ldots,b_c\in C_\slush(\vA)}\sum_{\vk}\prod_{i=1}^c\binom{d_\slush(b_i)}{k_i}
		=\frac1{c!}\sum_{\vk}\prod_{i=1}^c\sum_{b\in C_\slush(\vA)}\binom{d_\slush(b)}{k_i}.
	\end{align}
	Now, \eqref{eqClaim_Flipper31} implies that $\sum_{i\leq c}\vecone\cbc{k_i>2}k_i\leq 3\eps c$.
	Therefore, {\bf D3} and {\bf D4} ensure that for any $\vk$,
	\begin{align}\label{eqClaim_Flipper34b}
		\prod_{i=1}^c\sum_{b\in C_\slush(\vA)}\binom{d_\slush(b)}{k_i}&\leq L^{3\eps c}
		\prod_{i=1}^c\sum_{b\in C_\slush(\vA)}\binom{d_\slush(b)}{2}\leq L^{3\eps c}\bc{(\nu+o(1))n}^c\bcfr{\lambda^2\exp(\lambda)+o(1)}{2(\exp(\lambda)-\lambda-1)}^c.
	\end{align}
	Furthermore, there are no more than $\binom{m-c-1}{c-1}=\binom{m-c-1}{m-2c}$ possible vectors $\vk$ and thus \eqref{eqClaim_Flipper31} yields
	\begin{align}\label{eqClaim_Flipper34c}
		\binom{m-c-1}{m-2c}&\leq\bcfr{2\eul}{\eps}^{\eps c}.
	\end{align}
	Combining \eqref{eqClaim_Flipper34a}--\eqref{eqClaim_Flipper34c} with {\bf D1}, we see that the number of possible $C,\fC$ is bounded by
	\begin{align}\label{eqClaim_Flipper34}
		\bcfr{2\eul L^3}{\eps}^{\eps c}\bcfr{\eul(\nu+o(1))n}{c}^c\bcfr{\lambda^2\exp(\lambda)+o(1)}{2(\exp(\lambda)-\lambda-1)}^c.
	\end{align}
	Finally, since {\bf D2} and {\bf D4} imply that 
	\begin{align*}
		\sum_{x\in V_\slush(\vA)}d_\slush(x)=(1+o_\eps(1)) \nu n\Erw[\Po_{\geq2}(\lambda)]=(1+o_\eps(1))\frac{\nu n \lambda(\exp(\lambda)-1)}{\exp(\lambda)-\lambda-1},
	\end{align*}
	the probability that $\vec\Gamma_\slush$ matches the designated variable/check clones comes to
	\begin{align}\label{eqClaim_Flipper36}
		\frac{m!(\sum_{x\in V_\slush(\vA)}d_\slush(x)-m)!}{(\sum_{x\in V_\slush(\vA)}d_\slush(x))!}=
		\binom{\sum_{x\in V_\slush(\vA)}d_\slush(x)}{m}^{-1}
		=\bcfr{\eul(\lambda(\exp(\lambda)-1) \nu+o_\eps(1))n}{m(\exp(\lambda)-\lambda-1)}^{-m}.
	\end{align}
	
	Combining \eqref{eqClaim_Flipper33}, \eqref{eqClaim_Flipper34} and~\eqref{eqClaim_Flipper36} (and dragging all $o(1)$-error terms into the $o_\eps(1)$), we obtain
	\begin{equation*}
		\begin{split}
		\Erw&\brk{\vZ(u,c,m)\mid\cD}\\
		&\leq\bcfr{\eul\nu n}{u}^u\bcfr{\eul\nu n}{c}^c
		\bcfr{\eul(\lambda(\exp(\lambda)-1) \nu+o_\eps(1))n}{m(\exp(\lambda)-\lambda-1)}^{-m}
		\bcfr{\lambda^2\exp(\lambda)}{2(\exp(\lambda)-\lambda-1)}^c
		\bcfr{\lambda^2}{2(\exp(\lambda)-\lambda-1)}^{u}.
		\end{split}
	\end{equation*}
	Hence, \eqref{eqClaim_Flipper31} 
	yields
	\begin{align}\label{eqClaim_Flipper37}
		\Erw\brk{\vZ(u,c,m)\mid\cD}&\leq \bcfr{u}{n}^{m-u-c}
		\bcfr{\lambda^2\exp(\lambda)+o_\eps(1)}{(\exp(\lambda)-1)^2}^u.
	\end{align}
	Since $\lambda>0$ we have $\lambda^2\exp(\lambda)/((\exp(\lambda)-1)^2)<1$.
	Therefore,  \eqref{eqClaim_Flipper37}  implies \eqref{eqClaim_Flipper32} for small $\eps>0$.
\end{proof}

\begin{proof}[Proof of \Lem~\ref{PropFlipping}]
	The lemma follows from Claims~\ref{Claim_flippper1}, \ref{Claim_flippper2} and~\ref{Claim_Flipper3}.
	More precisely, let given $d>\eul$, let $\eps_1$ be the $\eps$ given by Claim~\ref{Claim_Flipper3},
	and subsequently set $\delta:=\eps_1$ and let $\eps_2,\eps_3$ be the $\eps$ given by Claims~\ref{Claim_flippper1} and~\ref{Claim_flippper2} respectively.
	Then let us set $\eps_0:= \eps_1\wedge \eps_2\wedge \eps_3$.
	
	Now Claims~\ref{Claim_flippper1} and~\ref{Claim_flippper2} imply that \whp\ there is no
	$U\in\fF_{\eps_0}(\vA)$ with $\sum_{x\in U}|\partial x\cap C_\slush(\vA)|\geq(2+\delta)|U|$
	or with $\sum_{a\in\partial U\cap C_\slush(\vA)}|\partial a\cap U|\geq(2+\delta)|C|$.
	On the other hand, conditioning on this event,
	since $\eps_0 \le \eps_1=\delta$ we have $\fF_{\eps_0}(\vA) \subset \fF_\delta(\vA)$,
	and therefore Claim~\ref{Claim_Flipper3} implies that \whp\ $F_{\eps_0}(\vA) \le \omega$
	for any function $\omega=\omega(n) \gg 1$, as required.
\end{proof}

\subsection{Proof of \Lem~\ref{Lemma_expectedSolutions}}\label{Sec_Lemma_expectedSolutions}
The proof is based on a somewhat delicate moment calculation.
Suppose that $|V_\slush(\vA)\cap\cF(\vA)|<\eps n$, i.e., very few coordinates in the slush are frozen.
Then Fact~\ref{fact_ker} implies that for most $v\in V_\slush(\vA)$ the corresponding entry $\vx_{\slush,v}$ of a random vector $\vx_\slush\in\ker\vAs$ takes the value $0$ with probability precisely $1/2$.
Furthermore, since $|V_\slush(\vA)|=\Omega(n)$ \whp, \Prop~\ref{Prop_pin} implies that for most pairs $u,v\in V_\slush(\vA)$ the entries $\vx_{\slush,u},\vx_{\slush,v}$ are stochastically independent.
Therefore, \whp\ the random vector $\vx_\slush$ has Hamming weight $(1/2+o_\eps(1))|V_\slush(\vA)|$.
Hence, a tempting first idea toward the proof of \Lem~\ref{Lemma_expectedSolutions} might be to simply calculate the expected number of vectors of Hamming weight $(1/2+o_\eps(1))|V_\slush(\vA)|$ in the kernel of $\vAs$.

This strategy would work if we could replace the $o_\eps(1)$ error term above by $O(n^{-1/2})$.
Indeed, there are trivially at most $2^{|V_\slush(\vA)|}$ candidate vectors of Hamming weight $|V_\slush(\vA)|/2+O(\sqrt n)$.
Moreover, it is not very hard to verify that a given such vector satisfies all checks with probability $\Theta\bc{2^{-|C_\slush(\vA)|}}$.
As a consequence, the expected number of vectors in $\ker\vAs$ of Hamming weight $|V_\slush(\vA)|/2+O(\sqrt n)$ tends to zero if $|C_\slush(\vA)|-|V_\slush(\vA)|\gg1$.
But unfortunately this simple calculation does not extend to larger $\eps$ as required by \Lem~\ref{Lemma_expectedSolutions}.
The reason is that for larger $\eps$ a second order term pops up, i.e., the probability that all checks are satisfied reads 
$$2^{-|\Cs(\vA)|+O_\eps(\eps^2)|\Cs(\vA)|}.$$
This quadratic term is due to the presence of checks of degree two -- checks of larger degree only produce at most
	a cubic error term.

It is therefore also tempting to modify the argument by observing that a Chernoff bound shows that the number of vectors of Hamming weight
	at least $(1/2+\eps)|V_\slush|$ or at most $(1/2-\eps)|V_\slush|$ is \linebreak $2^{|V_\slush|} \exp\bc{-\eps^2 |V_\slush| + O(\eps^3 n)}$
	and wonder, therefore, whether this $\exp(-\eps^2 |V_\slush|)$ would be enough to outweigh the quadratic term in the probability bound above.
	A more careful analysis shows that the number of checks of degree~2 is approximately $\gamma|C_\slush|$, where
	$\gamma =\pr\brk{\Po_{\ge 2}(\lambda)=2}$ (see Proposition~\ref{Prop_slush}), while the error term in the probability that
	such a check is satisfied is approximately $\exp(4\eps^2)$. The question then becomes whether $\gamma$ is smaller than $1/4$,
	in which case we could again use the first moment method.
	Unfortunately it is not hard to see that $\gamma$ tends to $1$ as $d$ tends to $\eul$ from above (since then $\lambda$ tends to $0$),
	and so at best this strategy would only work for sufficiently large $d$.

Instead, we deal with the problem of the quadratic error term by observing that a check node of degree two simply imposes an equality constraint on its two adjacent variables.
Thus, any two variable nodes that appear in a check node of degree two can be contracted into a single variable node and then the check node can be eliminated.
A variant of the moment calculation, without the quadratic error term, can then be applied to the matrix that the multigraph resulting from the contraction procedure induces. The argument in the second attempt at the first moment method shows that
	the $\exp(-\eps^2|V_\slush|)$ decay in the proportion of vectors with an imbalance of $\eps$ will be enough to handle the cubic and higher-order error terms.

To carry out this programme we first investigate the subgraph $G_\slush'(\vA)$ obtained from $G_\slush{}(\vA)$ by deleting all checks of degree greater than two.
More precisely, invoking \Lem~\ref{Lemma_contig}, for the apparent technical reason we will instead analyse the random multigraph $\cG_\slush'$ that results by applying the contraction procedure to the random multigraph $\cG_\slush$ chosen from the pairing model with the same degrees as $\G_\slush(\vA)$.
The proof of the following lemma can be found in \Sec~\ref{Sec_Cor_lambda}.

\begin{lemma}\label{Cor_lambda}
	For any $d>\eul$ there exists $b>0$ such that for any $\omega=\omega(n)\gg1$ the random graph $\cG_\slush'$ enjoys the following properties \whp{}
	\begin{enumerate}[(i)]
		\item The largest component of $\cG_\slush'$ has size at most $\omega\log n$.
		\item $\cG_\slush'$ contains no more than $\omega$ cycles.
		\item For any $t>0$ no more than $|V_\slush(\vA)|\exp(-bt)$ variable nodes belong to components of size at least $t$.
	\end{enumerate}
\end{lemma}

Now obtain the multigraph $\cG_\slush''$ from $\cG_\slush$ by deleting all checks of degree two and contracting every connected component of $\cG_\slush'$ into a single variable node.
Let us write $\cV_\slush''$ and $\cC_\slush''$ for the set of variable and check nodes of $\cG_\slush''$ and let $\cA_\slush''$ denote the matrix encoded by $\cG_\slush''$.
Further, for $v\in\cV_\slush''\cup\cC_\slush''$ let $d_\slush''(v)$ be the degree of $v$ in $\cG_\slush''$.
Finally, let $\cK_\eps''$ be the set of all vectors $\xi\in\ker\cA_\slush''$ such that
\begin{align*}
	\abs{\frac12-\frac{\sum_{x\in \cV_\slush''}d_\slush''(x)\vecone\cbc{\xi_x=0}}{\sum_{x\in \cV_\slush''}d_\slush''(x)}}&<\eps.
\end{align*}
\noindent
In \Sec~\ref{Sec_Lemma_slush_first} we will prove the following statement.

\begin{lemma}\label{Lemma_slush_first}
	For any $d>\eul$ there exists $\eps>0$ such that for any $\omega=\omega(n)\gg1$ we have
	$$\pr\brk{|\cC_\slush''|\geq |\cV_\slush''|+\omega\mbox{ and }\cK_\eps''\neq\emptyset}=o(1).$$
\end{lemma}

\noindent
Roughly, this says that if we have significantly more checks than variables, then all kernel vectors
	are imbalanced.
In addition, we observe the following.

\begin{lemma}\label{Lemma_rs_slush}
	For any $d>\eul$, $\eps>0$ there exists $\delta>0$ such that 
	\begin{align*}
		\pr\brk{|V_\slush(\vA)\setminus\cF(\vA)|>(1-\delta)|V_\slush(\vA)|\mbox{ and }\cK_\eps''=\emptyset}=o(1).
	\end{align*}
\end{lemma}

\noindent
The (slightly incorrect, but intuitive) interpretation is that if almost all of the slush is unfrozen, then there must be some balanced kernel vectors.
The proof of \Lem~\ref{Lemma_rs_slush} can be found in \Sec~\ref{Sec_Lemma_rs_slush}.

\begin{proof}[Proof of \Lem~\ref{Lemma_expectedSolutions}]
	The assertion is an immediate consequence of \Lem s~\ref{Lemma_slush_first} and~\ref{Lemma_rs_slush}.
\end{proof}

\subsection{Proof of \Lem~\ref{Cor_lambda}}\label{Sec_Cor_lambda}
We apply a branching process argument to a random graph chosen from the pairing model, not unlike the one from~\cite{MolloyReed}.
Specifically, let $(d_\slush(v))_{v\in V_\slush(\vA)}$ be the degree sequence of the graph $G_\slush(\vA)$ and let $\vm_\slush'$ be the number of check of degree two in $G_\slush(\vA)$.
Let us write $b_1,\ldots,b_{\vm_s'}$ for the check nodes of $\cG_\slush'$.
Starting from an edge exiting $b_1$, we will explore the set of all nodes of $\cG_\slush'$ that can be reached via that edge.
We will describe this exploration process as a branching process, which will turn out to be subcritical.

To be precise, let $\vec\Delta=\sum_{v\in V_\slush(\vA)}d_\slush(v)$ and let $\vec\Gamma_\slush'$ be a random perfect matching of the complete bipartite graph with vertex sets
\begin{align*}
	\cV=\bigcup_{v\in V_\slush(\vA)}\cbc v\times[d_\slush(v)]&&\mbox{and}&&
	\cC=\bc{\cbc{\alpha_1,\ldots,\alpha_{\vm'_\slush{}}}\times[2]}\cup\cbc{\beta_1,\ldots,\beta_{\vec\Delta-2\vm_{\slush}'}}.
\end{align*}
As always, $\cbc v\times[d_\slush(v)]$ and $\cbc{\alpha_i}\times[2]$ represent sets of clones of the variable node $v$ and the check node $\alpha_i$, respectively.
The `ballast' clones $\beta_1,\ldots,\beta_{\vec\Delta-2\vm_{\slush}'}$ are included so that both sides of the bipartition have the same size.
Further, deleting $\beta_1,\ldots,\beta_{\vec\Delta-2\vm_{\slush}'}$ and contracting the other clones into single vertices, we obtain a random multigraph $\cG(\vec\Gamma)$ from the matching $\vec\Gamma$.
This multigraph is identical in distribution to $\cG_\slush'$.

\begin{claim}\label{Claim_Gamma1}
	\Whp{} all connected components of $\cG(\vec\Gamma)$ have size $O(\log n)$.
\end{claim}
\begin{proof}
	To trace the set of nodes reachable from $(\alpha_1,1)$, we classify each clone as either unexplored, active or inactive.
	At the start of the process only $(\alpha_1,1)$ is active and all other clones are unexplored; thus,
	\begin{align*}
		\cA_0&=\cbc{(\alpha_1,1)},&\cU_0&=\cbc{(\alpha_1,2),(\alpha_2,1),(\alpha_2,2),\ldots,(\alpha_{\vm'_\slush{}},1),(\alpha_{\vm'_\slush{}},2)}\setminus\cA_0,&\cI_0&=\emptyset.
	\end{align*}
	The classification determines the order in which the edges of the matching $\vec\Gamma$ are exposed.
	Specifically, if at some time $t\geq1$ no active check clone remains, the process stops and we let $\stoptime=t-1$. 
	Otherwise, at time step $t\geq1$ an active clone $(\alpha_{\vi_t},\vh_t)\in\cA_{t-1}$ is chosen uniformly at random and we let $\cI_{t}=\cI_{t-1}\cup\{(\alpha_{\vi_t},\vh_t)\}$.
	If the second clone $(\alpha_{\vi_t},3-\vh_t)$ of the same check is either active or inactive, we let \linebreak $\cU_t=\cU_{t-1}$,  $\cA_t=\cA_{t-1}\setminus\cbc{(\alpha_{\vi_t},\vh_t)}$.
	Otherwise we expose the edge of $\vec\Gamma$ incident with the other clone $(\alpha_{\vi_t},3-\vh_t)$ of check $\alpha_{\vi_t}$.
	Let  $\vy_t$ be the variable node on the other end of this edge.
	We then declare all as yet inactive clones of checks $\alpha_i$, $i\in[\vm_\slush']$, that are adjacent to clones of $\vy_t$ active.
	Formally, we let
	\begin{align*}
		\cI_{t}&=\cI_{t-1}\cup\{(\alpha_{\vi_t},1),(\alpha_{\vi_t},2)\},&
		\cA_t&=\bc{\cA_{t-1}\cup\bc{\partial_{\vec\Gamma}(\vy_t\times[d_\slush{}(\vy_t)])\cap \cbc{(\alpha_i,1),(\alpha_i,2):i\in[\vm'_\slush]}}}\setminus\cI_t
	\end{align*}
	and $\cU_t=\cU_{t-1}\setminus(\cA_t\cup\cI_t).$
	Let $\fA_t$ be the $\sigma$-algebra generated by the first $t$ step of the process.
	
	The aim is to investigate the stopping time $\stoptime$.
	We may condition on the event $d_\slush{}(v)\leq\log^2n$ for all $v$.
	Moreover, we claim that for $1\leq t\leq \stoptime\wedge\log^3n$,
	\begin{align}\label{eqdrift}
		\Erw\brk{\abs{\cA_{t}}-\abs{\cA_{t-1}}\mid\fA_{t-1}}&<0.
	\end{align}
	Indeed, $\abs{\cA_{t}}-\abs{\cA_{t-1}}$ is trivially negative if $(b_{\vi_t},3-\vh_t)\not\in\cU_{t-1}$.
	Further, if $(\alpha_{\vi_t},3-\vh_t)\in\cU_{t-1}$, then $\vec\Gamma$ matches this clone to a random vacant variable clone.
	Because $t\leq \log^3n$ and $\max_vd_\slush(v)\leq\log^2n$ while the slush has size $|V_\slush(\vA)|=\Omega(n)$, the distribution of $d_{\slush}(\vy_t)$ is within $O(n^{-0.99})$ in total variation of the distribution $(d_\slush(v)/\vec\Delta)_{v\in V_\slush(\vA)}$ of the degree of the variable node of a random variable clone.
	We subsequently expose all edges of $\vec\Gamma$ incident with a clone of $\vy_t$ that was unexplored at time $t-1$.
	Once more because $t\leq \log^3n$ and $\max_vd_\slush(v)\leq\log^2n$, the conditional probability that a specific unexplored clone of  $\vy_t$ links to an unexplored clone from the set $\cbc{(\alpha_i,1),(\alpha_i,2):i\in[\vm'_\slush]}$ is bounded by $2\vm'_\slush/\vec\Delta+O(n^{-0.99})$.
	Therefore, we obtain the bound 
	\begin{align}\label{eqdrift2}
		\Erw\brk{\abs{\cA_{t}}-\abs{\cA_{t-1}}\mid\fA_{t-1}}&\leq o(1)-1+\Erw\brk{ \frac{2\vm'_\slush}{\vec\Delta^2}\sum_{v\in V_\slush(\vA)}d_\slush{}(v)(d_\slush{}(v)-1)}\leq \frac{\lambda^2\exp(\lambda)}{(\exp(\lambda)-1)^2} -1+o(1).
	\end{align}
	Moreover, it is easy to check that $\lambda>0$ for all $d>\eul$ and that
	\begin{align}\label{eq_lambda}
		\frac{z^2\exp(z)}{(\exp(z)-1)^2}&<1&&\mbox{for any }z>0.
	\end{align} 
	Thus, \eqref{eqdrift} follows from \eqref{eqdrift2} and \eqref{eq_lambda}.
	Finally, \eqref{eqdrift} implies that $(|\cA_t|)_t$ is dominated by a random walk with a negative drift.
	Consequently, $\pr\brk{\stoptime\geq c\log n}=o(n^{-1})$ for a suitable $c>0$. The assertion follows from the union bound.
\end{proof}

\begin{claim}\label{Claim_Gammat}
	There exists $b=b(d)>0$ such that \whp{} for all $t>0$ the number of variable nodes of $\cG_\slush'$ that belong to components of size at least $t$ is bounded by $|V_\slush(\vA)|\exp(-bt)$.
\end{claim}
\begin{proof}
	Let $\vZ_t$ be the number of variable nodes of $\cG_\slush'$ that belong to components of size at least $t$.
	Tracing the same exploration process as in the previous proof and using \eqref{eqdrift2}, we find $\zeta=\zeta(d)>0$ such that
	\begin{align}\label{eqClaim_Gammat1}
		\Erw[\vZ_t]&\leq |V_\slush{}(\vA)|\exp(-2\zeta t).
	\end{align}
	If $t>\log\log n$, say, then the assertion simply follows from \eqref{eqClaim_Gammat1} and Markov's inequality.
	Thus, suppose that $t\leq\log\log n$ and $|V_\slush(\vA)|=\Omega(n)$ and that the largest component of $\cG_\slush'$ contains no more than $\log n\log\log n$ variable nodes.
	Then adding to or removing from $\cG_\slush'$ a single edge can alter $\vZ_t$ by at most $2t$.
	Therefore, the assertion follows from \eqref{eqClaim_Gammat1} and Azuma's inequality (see, e.g., \cite[Corollary~7.2.2]{AlonSpencer}).
\end{proof}

\noindent
As a next step we need to estimate the number of short cycles.

\begin{claim}\label{Claim_Gamma2}
	The expected number of nodes on cycles of $\cG_\slush'$ of size at most $\log^2n$ is bounded.
\end{claim}
\begin{proof}
	Let $\ell\leq\log^2n$, let $\vy=(y_1,\ldots,y_\ell)\in V_\slush(\vA)^\ell$ be a sequence of variables, let $\vi=(i_1,i_1',\ldots,i_\ell,i_\ell')$ be a sequence that contains two clones of each variable $y_1,\ldots,y_\ell$ and let $\valpha=(\alpha_1,\ldots,\alpha_\ell)$ be a sequence of $\ell$ distinct checks of degree two.
	Let $\cE(\vy,\vi,\valpha)$ be the event that $\vec\Gamma$ connects the two clones of $\alpha_h$ with $(y_h,i_h')$ and $(y_{h+1},i_{h+1})$.
	Since \Prop~\ref{Prop_slush} shows that $\vec\Delta=\Omega(n)$ and $\ell\leq\log^2n$, we obtain
	\begin{align*}
		\pr\brk{\cE(\vy,\vi,\valpha)\mid (d_x)_x,\vm'_\slush{}}&\sim \bc{2/\DELTA^2}^\ell.
	\end{align*}
	Furthermore, we have 
	\begin{align*}
		\Erw\brk{\sum_{x\in V_\slush(\vA)}\frac{d_{y_i}(d_{y_i}-1)}{|V_\slush(\vA)|}  }&\sim \frac{\lambda^2\exp(\lambda)}{\exp(\lambda)-\lambda-1},&
		\Erw\brk{\frac{\DELTA}{|V_\slush(\vA)|}}&\sim\frac{\lambda(\exp(\lambda)-1)}{\exp(\lambda)-\lambda-1},&
		\Erw\brk{\frac{\vm_s'}{|V_\slush(\vA)|}}&\sim\frac{\lambda^2}{2(\exp(\lambda)-\lambda-1)}.
	\end{align*}
	Consequently, the expected number of nodes on cycles of length $\ell$ works out to be
	\begin{align*}
		\frac1{2\ell}\sum_{\vy,\vi,\valpha}2\ell\pr\brk{\cE(\vy,\vi,\valpha)\mid (d_x)_x,\vm'_\slush{}}&\sim\bcfr{\lambda^2\exp(\lambda)}{(\exp(\lambda)-1)^2}^\ell=\exp(-\Omega(\ell)).
	\end{align*}
	Summing on $\ell$ completes the proof.
\end{proof}

\noindent
\begin{proof}[Proof of \Lem~\ref{Cor_lambda}]
	The statement follows from Claims~\ref{Claim_Gamma1}--\ref{Claim_Gamma2}.
\end{proof}

\subsection{Proof of \Lem~\ref{Lemma_slush_first}}\label{Sec_Lemma_slush_first}
To simplify the notation we introduce $N=|\cV_\slush''|$, $M=|\cC_\slush''|$.
Moreover, we write $d_1,\ldots,d_N$ for the degrees of the variable nodes of $\cG_\slush''$ and $k_1,\ldots,k_M\geq3$ for the degrees of the constraints.
We need the following facts about $M,N$ and the degrees.

\begin{claim}\label{Claim_slushdegs}
	\Whp{} we have
	\begin{align}\label{eqDegs}
		M,N&=\Omega(n),&\max_{1\leq i\leq N} d_i&\leq\log^{3}N,&\max_{1\leq i\leq M} k_i&\leq\log^{2}N,&
		\sum_{i=1}^M k_i^2&=O(M),&
		\sum_{i=1}^N d_i^2&=O(N).
	\end{align}
\end{claim}
\begin{proof}
	The first estimate follows immediately from \Prop~\ref{Prop_slush} and \Lem~\ref{Cor_lambda}.
	The second statement follows from \Lem~\ref{Cor_lambda} (i) and the fact that the maximum degree of $G(\vA)$ is of order $\log n$ \whp{}, which also implies the third bound.
	Similarly, the sum of the squares of the check degrees of $G(\vA)$ is bounded \whp{} due to routine bounds on the tails of the binomial distribution.
	This implies that $\sum_{i=1}^M k_i^2=O(M)$ because $M=\Omega(n)$ \whp{} by \Prop~\ref{Prop_slush}.
	To obtain the final bound we apply the Chernoff bound to conclude that for any $d>0$ there exists $b>0$ such that \whp{}
	\begin{align}\label{eqClaim_slushdegs1}
		\frac1n\sum_{i=1}^n\vecone\cbc{|\partial_{G(\vA)} v_i|\geq t}&\leq\exp(-bt)/b.
	\end{align}
	In other words, the degree sequence of $G(\vA)$ has an exponentially decaying tail \whp{}
	Assuming $N=\Omega(n)$, we see that \eqref{eqClaim_slushdegs1} implies the bound
	\begin{align}\label{eqClaim_slushdegs2}
		\frac1N\sum_{i=1}^N\vecone\cbc{d_i\geq t}&\leq\exp(-b't)/b'
	\end{align}
	for some $b'>0$.
	Furthermore, \Lem~\ref{Cor_lambda} (iii) implies an exponentially decaying tail for the component sizes of $\cG_\slush'$.
	Since $\cG_\slush''$ is obtained by contracting the components of $\cG_\slush'$, the desired bounds follow from \eqref{eqClaim_slushdegs2} and \Lem~\ref{Lemma_weighted_tails}.
\end{proof}

In the following we will condition on the event $\cD$ that the conditions \eqref{eqDegs} are satisfied.
Let $\SIGMA\in\FF_2^N$ be a uniformly random vector.
We will prove \Lem~\ref{Lemma_slush_first} by estimating the probability that $\SIGMA\in\cK_\eps''$.
To this end, let
\begin{align*}
	\vW=\frac{\sum_{i=1}^Nd_i\vecone\cbc{\SIGMA_i=1}}{\sum_{i=1}^Nd_i}
\end{align*}
count the degree-weighted one-entries of $\SIGMA$.
The following claim bounds the probability that $\vW$ deviates significantly from $1/2$.

\begin{claim}\label{Claim_Z1}
	For any $d>\eul$ there is $s=s(d)>0$ such that $\pr\brk{|\vW-1/2|\geq t\mid\cD}\leq2\exp(-st^2N)$.
\end{claim}
\begin{proof}
	This is an immediate consequence of \eqref{eqDegs} and Azuma's inequality.
\end{proof}

\noindent
As a next step we calculate the probability that $\SIGMA\in\ker\cA_\slush''$ given $\vW$.

\begin{claim}\label{Claim_Z2}
	For any $d>\eul$ there exist $\eps>0,\gamma>0$ such that uniformly for every $w\in(1/2-\eps,1/2+\eps)$ for which  $w\sum_{i=1}^Mk_i$ is an even integer we have
	$$\log\pr\brk{\cA_\slush''\SIGMA=0\mid\vW=w,\cD}\leq -M\log 2-\gamma M(w-1/2)^3+O(1).$$
\end{claim}
\begin{proof}
	Consider a random vector $\vxi=(\vxi_{ij})_{i\in[M],j\in[k_i]}$ where we choose every entry $\vxi_{ij}\in\field$ to be a one with probability $w$ independently.
	Let $\cS$ be the event that $\sum_{j\in[k_i]}\vxi_{ij}=0$ for all $i\in[M]$.
	Moreover, let
	\begin{align*}
		\cR&=\cbc{\sum_{i=1}^M\sum_{j=1}^{k_i}\bc{\vecone\{\vxi_{i,j}=1\}-w}=0}.
	\end{align*}
	Because $\cG_\slush''$ is drawn from the pairing model, we have
	\begin{align}\label{eqClaim_Z2_1}
		\pr\brk{\cA_\slush''\SIGMA=0\mid\vW=w,\cD}&=\pr\brk{\cS\mid\cR}.
	\end{align}
	
	We will calculate the probability on the r.h.s.\ of \eqref{eqClaim_Z2_1} via Bayes' rule.
	The unconditional probabilities are computed easily.
	Indeed, for every $i\in[M]$ we have
	\begin{align*}
		\pr\brk{\sum_{j\in[k_i]}\vxi_{ij}=0}&=
		\sum_{j=0}^k\vecone\cbc{j\mbox{ even}}\binom kjw^j(1-w)^{k-j}\\
		&=\frac12\brk{\sum_{j=0}^k\binom kjw^j(1-w)^{k-j}+\sum_{j=0}^k\binom kj(-w)^j(1-w)^{k-j}}
		=\frac{1+(1-2w)^k}{2}.
	\end{align*}
	Hence,
	\begin{align}\label{eqClaim_Z2_2}
		\pr\brk{\cS}&=\prod_{i=1}^M\frac{1+(1-2w)^{k_i}}{2}.
	\end{align}
	Furthermore, the local limit theorem for the binomial distribution shows that 
	\begin{align}\label{eqClaim_Z2_3}
		\pr\brk\cR=\Theta(M^{-1/2}).
	\end{align}
	In addition, \eqref{eqDegs} and the local limit theorem for sums of independent random variables yield
	\begin{align}\label{eqClaim_Z2_4}
		\pr\brk{\cR\mid\cS}&=\Theta(M^{-1/2}).
	\end{align}
	Combining \eqref{eqClaim_Z2_2}--\eqref{eqClaim_Z2_4} and recalling that the $\vxi_{ij}$ are independent, we obtain
	\begin{align}\label{eqClaim_Z2_5}
		\log\pr\brk{\cS\mid\cR}&=\sum_{i=1}^M\log\frac{1+(1-2w)^{k_i}}{2}+O(1)=-M\log2+\sum_{i=1}^M\log(1+(1-2w)^{k_i})+O(1).
	\end{align}
	To complete the proof we compute the derivatives of the last expression, keeping in mind that $k_i\geq3$ for all $i$:
	\begin{align*}
		\frac{\partial\log\pr\brk{\cS\mid\cR}}{\partial w}&=\sum_{i=1}^M\frac{-2k_i(1-2w)^{k_i-1}}{1+(1-2w)^{k_i}},\\
		\frac{\partial^2\log\pr\brk{\cS\mid\cR}}{\partial w^2}&=\sum_{i=1}^M\frac{4k_i(k_i-1)(1-2w)^{k_i-2}}{1+(1-2w)^{k_i}}
		-\frac{4k_i^2(1-2w)^{2k_i-2}}{\bc{1+(1-2w)^{k_i}}^2},\\
		\frac{\partial^3\log\pr\brk{\cS\mid\cR}}{\partial w^3}&=	
		\sum_{i=1}^M\frac{-8k_i(k_i-1)(k_i-2)(1-2w)^{k_i-3}}{1+(1-2w)^{k_i}}
		+\frac{8k_i^2(k_i-1)(1-2w)^{k_i-2}(1-2w)^{k_i-1}}{\bc{1+(1-2w)^{k_i}}^2}\\
		&\qquad+\frac{16k_i^2(k_i-1)(1-2w)^{2k_i-3}}{\bc{1+(1-2w)^{k_i}}^2}
		-\frac{16k_i^3(1-2w)^{3k_i-2}}{\bc{1+(1-2w)^{k_i}}^3}.
	\end{align*}
	Evaluating these derivatives at $w=1/2$, we obtain
	\begin{align}\label{eqClaim_Z2_6}
		\frac{\partial\log\pr\brk{\cS\mid\cR}}{\partial w}\bigg|_{w=1/2}&=
		\frac{\partial^2\log\pr\brk{\cS\mid\cR}}{\partial w^2}\bigg|_{w=1/2}=0,&
		\frac{\partial^3\log\pr\brk{\cS\mid\cR}}{\partial w^3}&=-48\sum_{i=1}^{M}\vecone\cbc{k_i=3}.
	\end{align}
	Finally, combining \eqref{eqClaim_Z2_1}, \eqref{eqClaim_Z2_5} and \eqref{eqClaim_Z2_6} with Taylor's formula completes the proof.
\end{proof}

\begin{proof}[Proof of \Lem~\ref{Lemma_slush_first}]
	Choose $\eps=\eps(d)>0$ small enough.
	Summing over $w\in(1/2-\eps,1/2+\eps)$ such that $w\sum_{i=1}^Nd_i$ is an even integer, we obtain
	\begin{align*}
		\pr\big[\cK_\eps\neq\emptyset\mid\cD,\,M\geq N+\omega\big]
		&\leq2^N\pr\brk{\cA_\slush''\SIGMA=0, \; |\vW-1/2|<\eps\mid \cD,\,M\geq N+\omega}\\
		&\leq2^N\sum_{w}\pr\big[\vW=w\mid\cD,\,M\geq N+\omega\big]\pr\brk{\cA_\slush''\SIGMA=0\mid\vW=w,\cD,\,M\geq N+\omega}.
	\end{align*}
	Combining this bound with Claims~\ref{Claim_Z1} and~\ref{Claim_Z2}, we obtain
	\begin{align*}
		\pr\big[\cK_\eps\neq\emptyset&\mid\cD,\,M\geq N+\omega\big]\\
		&\leq2^N\sum_{h=1}^{\lceil\eps\sqrt N\rceil}\sum_{w:h-1\leq w\sqrt N\leq h}
		\pr\big[\vW=w\mid\cD,\,M\geq N+\omega\big]\pr\brk{\cA_\slush''\SIGMA=0\mid\vW=w,\cD,\,M\geq N+\omega}\\
		&\leq2^{N-M}
		\sum_{1\leq h\leq \eps\sqrt n}\exp\Big(-\Omega\bc{h^2}+O\bc{h^3MN^{-3/2}}\Big)=O\bc{2^{N-M}}=o(1),
	\end{align*}
	provided that $M\geq N+\omega$ and $\eps>0$ is small enough.
\end{proof}

\subsection{Proof of \Lem~\ref{Lemma_rs_slush}}\label{Sec_Lemma_rs_slush}
The following observation is an easy consequence of the construction of $\vA_\slush$.

\begin{claim}\label{Claim_minor}
	If $v,y\in V(\vA_\slush)$ are variables such that $\xi_v=\xi_y$ for all $\xi\in\ker\vA_\slush$, then $\xi_v=\xi_y$ for all $\xi\in\ker\vA$.
\end{claim}
\begin{proof}
	By construction the matrix $\vA_\slush$ is the minor of $\vA$ induced on $V_\slush(\vA)\times C_\slush(\vA)$.
	Although some of the checks $a\in C_\slush(\vA)$ may contain variables $v\not\in V_\slush(\vA)$, all such $v$ are frozen in $\vA$.
	Therefore, any $\xi\in\ker\vA$ induces a vector $\xi_{\slush}\in\ker\vA_\slush$.
\end{proof}

We now combine Claim~\ref{Claim_minor} with \Prop~\ref{Prop_pin} to prove the lemma.
Hence, let $\cU$ be the event that $|V_\slush{}(\vA)\setminus\cF(\vA)|>(1-\delta)|V_\slush(\vA)|$.
Provided that $\delta=\delta(d,\eps)>0$ is chosen small enough, routine tail bounds for the binomial distribution imply that the event
\begin{align}\label{eqLemma_rs_slush_1}
	\cE=\cbc{\sum_{v\in V_\slush(\vA)\cap\cF(\vA)}d_\slush(v)<\frac\eps4\sum_{v\in V_\slush(\vA)}d_\slush(v)} \quad \mbox{satisfies}\quad\pr\brk{\cU\setminus\cE}&=o(1).
\end{align}
Further, with $\vx_\slush=(\vx_{\slush,y})_{y\in V_\slush(\vA)}\in\ker\vA_\slush{}$ chosen randomly, \Prop~\ref{Prop_pin} and Claim~\ref{Claim_minor} ensure that the event
\begin{align*}
	\cbc{\sum_{y,y'\in V_\slush(\vA)\setminus\cF(\vA)}
		\abs{\pr\brk{\vx_{\slush,y}=\vx_{\slush,y'}=0\mid\vA}-\frac14}<|V_\slush(\vA)|\log^{-9}n}
\end{align*}
has probability $1-o(1)$.
As a consequence, since all degrees of $G_\slush(\vA)$ are bounded by $\log n$ \whp, the event
\begin{align*}
	\cR=\cbc{\sum_{y,y'\in V_\slush(\vA)\setminus\cF(\vA)}d_\slush(y)d_\slush(y')
		\abs{\pr\brk{\vx_{\slush,y}=\vx_{\slush,y'}=0\mid\vA}-\frac14}<
		\bc{\sum_{y\in V_\slush(\vA)}d_\slush(y)}^2\log^{-4}n}
\end{align*}
satisfies $\pr\brk{\cR}=1-o(1)$.
Hence, \eqref{eqLemma_rs_slush_1} yields $\pr\brk{\cU\setminus(\cE\cap\cR)}=o(1).$
In effect, it suffices to prove that on the event $\cU\cap\cE\cap\cR$ we have $\cK_\eps\neq\emptyset$.

To verify this we recall that any variables $y,y'$ that get contracted in the course of the construction of $G_\slush''(\vA)$ deterministically satisfy $\vx_{\slush,y}=\vx_{\slush,y'}$.
As a consequence, for a random $\vx_\slush''\in\ker\vA_\slush''$ we have
\begin{align*}
	\sum_{y,y'\in V_\slush''(\vA)\setminus\cF(\vA_\slush'')}d_\slush''(y)d_\slush''(y')
	&\abs{\pr\brk{\vx_{\slush,y}''=\vx_{\slush,y'}''=0\mid\vA}-\frac14}\\
	&=\sum_{y,y'\in V_\slush(\vA)\setminus\cF(\vA)}d_\slush(y)d_\slush(y')
	\abs{\pr\brk{\vx_{\slush,y}=\vx_{\slush,y'}=0\mid\vA}-\frac14}.
\end{align*}
Therefore, if $\cU\cap\cE\cap\cR$ occurs, then so does the event
\begin{align*}
	\cS&=\cbc{
		\sum_{y,y'\in V_\slush''(\vA)\setminus\cF(\vA_\slush'')}d_\slush''(y)d_\slush''(y')
		\abs{\pr\brk{\vx_{\slush,y}''=\vx_{\slush,y'}''=0\mid\vA}-\frac14}
		<\bc{\sum_{y\in V_\slush''(\vA)\setminus\cF(\vA_\slush'')}d_\slush''(y)}^2\log^{-3}n}.
\end{align*}
To complete the proof, consider the random variable
\begin{align*}
	\vX=\frac{\sum_{y\in V_\slush''(\vA)\setminus\cF(\vA_\slush'')}d_\slush''(y)\vecone\cbc{\vx_{\slush,y}''=0}}{\sum_{y\in V_\slush''(\vA)\setminus\cF(\vA_\slush'')}d_\slush''(y)}.
\end{align*}
Then on $\cU\cap\cE\cap\cR$ we have $\Erw[\vX\mid\vA]\sim1/2$ because $\vx_{\slush,y}''=0$ with probability $1/2$ for every $y\in V_\slush''(\vA)\setminus\cF(\vA_\slush'')$.
Moreover, because $\cU\cap\cE\cap\cR\subset\cS$ the conditional second moment works out to be $\Erw[\vX^2\mid\vA]\sim1/4$.
Hence, Chebyshev's inequality shows that $\pr\brk{|\vX-1/2|<\eps/4\mid\vA}=1-o(1)$.
In particular, on $\cU\cap\cE\cap\cR$ there exists a vector $\xi\in\ker\vA_\slush''$ such that
\begin{align*}
	\abs{\frac{\sum_{y\in V_\slush''(\vA)\setminus\cF(\vA_\slush'')}d_\slush''(y)\vecone\cbc{\xi_{y}''=0}}{\sum_{y\in V_\slush''(\vA)\setminus\cF(\vA_\slush'')}d_\slush''(y)}-\frac12}<\frac\eps4.
\end{align*}
Recalling the definition of the event \eqref{eqLemma_rs_slush_1}, we conclude that $\xi\in\cK_\eps$ and thus $\cK_\eps\neq\emptyset$.

\appendix 
\section*{Appendix}

\section{The pinning operation and the overlap}\label{Sec_pino}

\subsection{Proof of \Prop~\ref{Prop_pin}}\label{Sec_Prop_pin}
Let $A$ be an $m\times n$-matrix over $\field$ and let $\vs_1,\vs_2,\ldots\in[n]$ be a sequence of uniformly distributed random variables, mutually independent and independent of all other sources of randomness.
Further, for an integer $t\geq0$ let $A[t]$ be the matrix obtained by adding $t$ more rows to $A$ such that the $j$-th new row contains precisely one non-zero entry in position $\vs_j$.
The proof of \Prop~\ref{Prop_pin} is based on the following fact.

\begin{lemma}[{\cite[Proposition~2.4  \& Equation~(3.12)]{Gao}}]\label{Lemma_Ayre}
	For $\eps>0,\ell>0$ let $T=T(\eps,\ell)=\lceil 4\ell^3/\eps^4\rceil+1$. 
	Then for all $m,n>0$ and all $m\times n$-matrices $A$ over $\field$ the following is true.
	Draw $\vec t\in[T]$ uniformly and choose $\vx\in\ker A[\vt]$ randomly.
	Then
	\begin{align*}
		\sum_{\substack{i_1,\ldots,i_\ell\in[n]\\\sigma\in\field^\ell}}
		\Erw\abs{\pr\brk{\vx_{i_1}=\sigma_1,\ldots,\vx_{i_\ell}=\sigma_\ell\mid A[\vt]}
			-\prod_{h=1}^\ell\pr\brk{\vx_{i_h}=\sigma_h\mid A[\vt]}}
		&<\eps n^\ell.
	\end{align*}
\end{lemma}

\noindent
To prove \Prop~\ref{Prop_pin} we will combine \Lem~\ref{Lemma_Ayre} with the observation that the random matrix $\vAnp$ is essentially invariant under the random perturbation required by \Lem~\ref{Lemma_Ayre}.
To be precise, let $\cZ$ be the set of all indices $i\in[n]$ such that $\vA_{ij}=0$ for all $j\in[n]$.
Further, for an integer $t\geq0$ let $\vA\bck t$ be the matrix obtained from $\vA$ as follows.
If $|\cZ|\leq t$, then $\vA\bck t=\vA$.
Otherwise draw a family $\vz_1,\ldots,\vz_t\in\cZ$ of $t$ distinct row indices uniformly at random and obtain $\vA\bck t$ from $\vA$ by replacing the $\vi_h$-th entry in row $\vz_h$ by one for $h=1,\ldots,t$, where $\vi_h$ is chosen uniformly at random from $[n]$ independently for each $h \in [t]$.
Thus, instead of attaching $t$ new rows as in \Lem~\ref{Lemma_Ayre} we simply insert a single non-zero entry into $t$ random all-zero rows of $\vA$.

\begin{lemma}\label{Lemma_cheat}
	Let $d>0$, let $T=o(\sqrt n)$ be an integer and choose $\vt\in[T]$ uniformly.
	Then $\dTV(\vA,\vA\bck\vt)=o(1)$.
\end{lemma}
\begin{proof}
	Because each entry of $\vA$ is non-zero with probability $d/n$ independently, the number $\vX$ of rows of $\vA$ with at most one non-zero entry has distribution $\Bin(n,(1-d/n)^n+d(1-d/n)^{n-1})$.
	Further, given $\vX$ the number $\vX_0$ of all-zero rows has a binomial distribution
	\begin{align*}
		\vX_0\disteq\Bin\bc{\vX,\frac{(1-d/n)^n}{(1-d/n)^n+d(1-d/n)^{n-1}}}.
	\end{align*}
	Let $\vA\mid(\vX,\vX_0)$ denote the distribution of $\vA$ given $\vX,\vX_0$.
	We have $\vX\geq\exp(-d)n$ \whp{} Given ${\vX\geq\exp(-d)n}$ the conditional variance satisfies $\Var[\vX_0\mid\vX]=\Omega(n)$.
	Therefore, the local limit theorem for the binomial distribution implies that
	$\vA\mid(\vX,\vX_0)$ and $\vA\mid(\vX,\vX_0-\vt)$ have total variation distance $o(1)$.
	From a more elementary point of view, one can check by hand that 
		for any integer $k = (1+o(1))\Erw \vX_0$ and any \linebreak $t = o(\sqrt{\Var \vX_0})$, which is certainly the case
		if $t\in [T]$, we have
		$\pr\brk{\vX=k} = (1+o(1))\pr\brk{\vX=k-t}$. Therefore $\vX_0,\vX_0-\vt$ have total variation
		distance $o(1)$, and thus the same also holds for $\vA\mid(\vX,\vX_0)$ and $\vA\mid(\vX,\vX_0-\vt)$.
	
	Furthermore, $\vA\mid(\vX,\vX_0-\vt)$ is distributed precisely as $\vA\bck\vt$.
\end{proof}

\begin{proof}[Proof of \Prop~\ref{Prop_pin}]
	The proposition is an immediate consequence of \Lem s~\ref{Lemma_Ayre} and~\ref{Lemma_cheat}.
	More precisely, observe that for fixed $\sigma$, the summand in Lemma~\ref{Lemma_Ayre} is identical for any choice of $i_1,\ldots,i_\ell$ by symmetry,
		while clearly the summand for each choice of $\sigma$ can be bounded by the sum over $\sigma$,
		so we obtain
		$$
		\Erw\abs{\pr\brk{\vx_{i_1}=\sigma_1,\ldots,\vx_{i_\ell}=\sigma_\ell\mid A[\vt]}
			-\prod_{h=1}^\ell\pr\brk{\vx_{i_h}=\sigma_h\mid A[\vt]}} < \eps.
		$$
		We now choose $\eps = n^{-2\gamma}$ for some $\gamma < 1/16$
		(note that Lemma~\ref{Lemma_Ayre} is a statement for all $m,n>0$ rather than an asymptotic statement,
		and therefore the choice of $\eps$ as a function of $n$ is permissible).
		This choice ensures that $T= \Theta(\eps^{-4}) = o(\sqrt{n})$, so that we can still apply Lemma~\ref{Lemma_cheat}.
\end{proof}

\subsection{Proof of \Cor~\ref{Cor_pin}}\label{Sec_Cor_pin}
Due to \Prop~\ref{Prop_pin} we may assume that $\vA$ satisfies
\begin{align}\label{eqCor_pin1}
	\frac{1}{n^2}\sum_{h,i=1}^n
	\abs{\pr\brk{\vx_h=\sigma_1,\vx_i=\sigma_2\mid\vAnp}
		-\pr\brk{\vx_h=\sigma_1\mid\vAnp}\pr\brk{\vx_i=\sigma_2\mid\vAnp}}&=o(1)
	&&\mbox{for all }\sigma_1,\sigma_2\in\FF_2.
\end{align}
Hence, fix $x\in\ker\vA$.
For $\sigma\in\field$ let $\cI(x,\sigma)=\cbc{i\in[n]\setminus\cF(\vA):x_i=\sigma}$.
Further, define
\begin{align*}
	R_\sigma(x,x')&=\frac{1}{n}\sum_{i\in\cI(x,\sigma)}\vecone\cbc{x_i'=\sigma}.
\end{align*}
Then Fact~\ref{fact_ker} implies that
\begin{align}\label{eqCor_pin2}
	\Erw\brk{R_\sigma(x,\vx')\mid\vA}&=\frac{|\cI(x,\sigma)|}{2n}.
\end{align}
Moreover, \eqref{eqCor_pin1} implies that $\Var\brk{R_\sigma(x,\vx')\mid\vA}=o(1)$.
Combining this bound with \eqref{eqCor_pin2} and applying Chebyshev's inequality, we conclude that 
\begin{align}\label{eqCor_pin4}
	\Erw\brk{\abs{R_\sigma(x,\vx')-\frac{|\cI(x,\sigma)|}{2n}}\mid\vA}&=o(1).
\end{align}
Further, since $R(x,\vx')=f(\vA)+\sum_{\sigma\in\field}R_\sigma(x,\vx')$, \eqref{eqCor_pin4} shows that
\begin{align}\label{eqCor_pin5}
	\Erw\brk{\abs{R(x,\vx')-\bc{f(\vA)+(1-f(\vA))/2}}\mid\vA}&=o(1)&&\mbox{for every }x\in\ker\vA.
\end{align}
Averaging \eqref{eqCor_pin5} on $x\in\ker\vA$ completes the proof.

\section{Proof of \Lem~\ref{Lemma_contig}}\label{Sec_contig}

\noindent
We first note that since in the pairing model we must connect variable nodes with check nodes,
certainly $\cG_\slush$ cannot contain any loops. We therefore need to show that there is at least
a constant probability of creating no double-edges.

Suppose that $d_1,\ldots,d_n$ are the degrees of variable nodes in $G_\slush(\vA)$ (where
we set $d_i=0$ if the corresponding node is not in $G_\slush(\vA)$), and similarly let
$\hat d_1,\ldots,\hat d_n$ be the degrees of check nodes. \linebreak Let  $m:= \sum_{i=1}^n d_i = \sum_{i=1}^n \hat d_i$.
It follows from Proposition~\ref{Prop_slush} that \whp\ 
$m=\Theta(n)$. It also follows from the fact that the degree of a node in $G_\slush(\vA)$ are
necessarily at most its degree in $G(\vA)$ that \whp\ 
\linebreak $\sum_{i=1}^n d_i^2,\sum_{i=1}^n \hat d_i^2 = O(n)$.
In what follows, we will implicitly condition on these high probability events.

Let $X=X(d_1,\ldots,d_n,\hat d_1,\ldots,\hat d_n)$ be the random variable
counting the number of double-edges in $\cG_\slush$.
Then we have
$$
\Erw[X] = \sum_{i=1}^n\sum_{j=1}^n 2 \binom{d_i}{2} \binom{\hat d_j}{2} \frac{1}{m(m-1)} =O(1).
$$
Similarly, it is an easy exercise to show that for any integer $\ell \in \NN$ the $\ell$-th moment of $X$ satisfies \linebreak
$\Erw[(X)_\ell]=(1+o(1))\Erw[X]^\ell$. Therefore $X$ is asymptotically distributed
as a $\Po\bc{\Erw[X]}$ random variable, and we have $\Pr[X=0]\to \exp\bc{-\Erw[X]}>0$, as required.

To show that $\cG_\slush$ conditioned on being simple has the same distribution as $G_\slush(\vA)$,
we simply need to observe that every simple bipartite graph with the appropriate distribution is equally
likely to be $G_\slush(\vA)$. To see this, consider two Tanner graphs $S,S'$ with the same degree distribution,
and a Tanner graph $H$ such that $H_\slush = S$. Let $H'$ be the Tanner graph obtained from $H$ by replacing
$S$ with $S'$, but otherwise leaving edges unchanged. 
Then the peeling process used to obtain the slush is completely identical on $H\setminus S$  and~$H'\setminus S'$,
and therefore $H'_\slush=S'$.
Since $H,H'$ have the same number of edges, both are equally likely to be~$G(\vA)$.
Summing over all possibilities for $H$ such that $H_\slush = S$, we deduce that $S,S'$ are
equally likely to be~$G_\slush(\vA)$.

\section{Proof of \Lem~\ref{Lemma_tails}}\label{Sec_Lemma_tails}

\noindent
For the first part of the lemma, notice that $|\partial v|$ is distributed as a binomial random variable with \linebreak parameters~$n$ and $p$ for any $v\in V(\vA)\cup C(\vA)$. Suppose $v \in V(\vA)$ and let $c=\ceil{\log(n)/2}$. Then we have
\begin{align}\label{eq:maxdegvar}
	\Pr\brk{ \exists v: \abs{\partial v} \geq c} &\leq n \binom{n}{c} p^c \leq n \binom{n}{c} \bc{\frac{d}{n}}^c \nonumber \\
	&\leq n \bc{\frac{ed}{c}}^c = \exp\brk{ \bc{1-\frac{\log 2 }{2}} \log n - \frac{\log (n)}{2} \cdot \bc{\log \log(n) } + O(\log\log n)  } =o(1).
\end{align}
Similarly, for a constraint $a \in C(\vA)$ we have
\begin{align} \label{eq:maxdegcons}
	\Pr\brk{ \exists a: \abs{\partial a} \geq c} =o(1).
\end{align}
Combining \eqref{eq:maxdegvar} and \eqref{eq:maxdegcons} completes the proof of the first part. For the second part, let $x_0$ be an arbitrary variable node. Then,
\begin{align*}
	\Erw\brk{ \sum_{x\in V(\vA)}\frac1{\ell!}{\prod_{j=1}^\ell(|\partial x|-j+1)}}
	= \frac{n}{\ell!} \Erw\brk{\prod_{j=1}^\ell(|\partial x_0|-j+1)}= \frac{n}{\ell!} \frac{n!}{ (n-\ell)! } {p}^\ell \leq \frac{d^\ell n}{\ell!}.
\end{align*}
Hence, the assertion follows from Markov's inequality.

\section{Proof of \Lem~\ref{Lemma_weighted_tails}\label{Sec_Lem_weighted_tails}}

\noindent
Assume, without loss of generality, that $0<c_1<10^{-5}$. Moreover, let $c_0>0$ , define $a=\exp(c_1)>1$ and $\log_a^{(m)}n := \log_a \ldots \log_a n$, where the logarithm with basis $a$ is taken $m$ times. For any $m \in \NN$ (or more precisely for any $m$ such that we have $s_m >0$), define
$$
s_m := 6\log_a^{(m)} n.
$$	
Let us set $q_j:= \max \cbc{w_i : i \in P_j}$, and 
define the event
$$
\cE_{j,m} := \cbc{s_{m+1} < \max\cbc{q_j,|P_j|} \le s_m}
$$
and the set
$$
E_{m} := \cbc{j : \cE_{j,m} \mbox{ holds}}.
$$
Note in particular that $\bigcup_{m' \ge m}\cE_{j,m'}$ is the event that $|P_j| \le s_m$ and $w_i \le s_m$ for all $i \in P_j$,
i.e., both the partition class and all associated weights are at most $s_m$.
We also observe that $\bigcup_{m=1}^\infty E_{m} = [\ell]$.
We further define
$$
x_m := \frac{1}{n} \sum_{j\in E_{m}}\bc{\sum_{i\in P_j}w_i }^2,
$$
so in particular we have
\begin{equation} \label{eq:partitionedsum}
	x= \sum_{m=1}^{\infty} x_m.
\end{equation}
We therefore aim to bound each $x_m$.
Let $m_0 = m_0(n)$ be the largest integer such that  $s_{m_0}\ge \frac{100\log(1/c_1)}{c_1} $.

We first consider the case when $m \le m_0$.
Observe that if $j \in E_{m}$, then we have $|P_j| \le s_m$ and for all $i \in P_j$ we have
$w_i \le s_m$,
and therefore
\begin{equation}\label{eq:weightsquared}
	\bc{\sum_{i\in P_j}w_i }^2 \le s_m^4.
\end{equation}
On the other hand, we can bound $|E_{m}|$ from above by making a case distinction.
Let us define
\begin{align*}
	E_{m}^{(1)} & := \cbc{j : \cE_{j,m} \mbox{ holds and }  q_j \ge |P_j|}, \\
	E_{m}^{(2)} & := \cbc{j : \cE_{j,m} \mbox{ holds and }  q_j \le |P_j|}. 
\end{align*}

\textbf{Case 1:} $q_j \ge |P_j|$.\\
Then we have $w_i \ge s_{m+1}$ for some $i \in P_j$,
but since this can hold for at most $ c_0 a^{-s_{m+1}}n \le c_0 s_m^{-5}n$
values of~$i$, we have
$$
|E_{m}^{(1)}| \le c_0 s_m^{-5}n.
$$

\textbf{Case 2:} $q_j \le |P_j|$.\\
Then we have $|P_j| \ge s_{m+1}$, which can also only hold for at most $ c_0 a^{-s_{m+1}}n \le c_0 s_m^{-5}n$  values of $j$,
so 
$$
|E_{m}^{(2)}| \le c_0 s_m^{-5}n.
$$

Thus we have $|E_{m}| \le 2 c_0 s_m^{-5} n$ and together with~\eqref{eq:weightsquared} we deduce that $x_m \le 2 c_0 s_m^{-1}$.
Thus~\eqref{eq:partitionedsum} gives
\begin{equation}\label{eq:splitsum}
	x \le 2 c_0 \sum_{m=1}^{m_0} \frac{1}{s_m} + \sum_{m=m_0+1}^\infty x_m.
\end{equation}
We further observe that 
for any $m \le m_0$ we have
$$
\frac{s_{m}}{s_{m-1}}= \frac{6 \log_a \left(\frac{s_{m-1}}{6}\right)}{s_{m-1}}
\le \frac{6\log_a s_{m-1}}{s_{m-1}} \le \frac{6\log_a s_{m_0}}{s_{m_0}}.
$$
We have
$$ \frac{6\log_a s_{m_0}}{s_{m_0}} =  \frac{6}{100 \log(1/c_1)}\Big(\log100 + \log(1/c_1) + \log \log (1/c_1)\Big).
$$
In order to  bound the ratio $\frac{6\log_a s_{m_0}}{s_{m_0}}$, we define the function
$$g(c_1) = \frac{6}{10} \bc{ \log(100) + \log\bc{ \frac{1}{c_1} } + \log\log\bc{ \frac{1}{c_1}}}- \log\bc{ \frac{1}{c_1}}.$$
We have $\lim_{c_1 \rightarrow 0} g(c_1)= -\infty$ and $g(10^{-5}) < -0.375985860$. Also, 
$$ g'(c_1)=\frac{2}{5c_1} - \frac{3}{5 c_1 \log\bc{1/c_1}} >0, $$
so $g$ is increasing in that interval and $g(c_1) < 0$. Thus,  we have $ \frac{6\log_a s_{m_0}}{s_{m_0}} < 1/10$ because $ \frac{6\log_a s_{m_0}}{s_{m_0}} < 1/10$ is equivalent to $g(c_1) < 0$. Therefore,
\begin{equation}\label{eq:splitsumfirst}
	\sum_{m=1}^{m_0} \frac{1}{s_m} \le \frac{1}{ s_{m_0}}\left(1+\frac{1}{10} + \frac{1}{100}+ \ldots\right) \le 10^{-9}.
\end{equation}
It remains to estimate $\sum_{m=m_0+1}^\infty x_m$,
for which we now restrict attention to $i$ and $j$ such that \linebreak $w_i,|P_j| \le s_{m_0+1} \le 100 \frac{\log(1/c_1)}{c_1} $.
Then we have $\bc{\sum_{i \in P_j} w_i}^2 \le 10^8 \bc{ \frac{\log (1/c_1)}{c_1}}^4 $, and we trivially have \linebreak $|\bigcup_{m \ge m_0+1} E_{m}| \le \ell \le n$, therefore
\begin{equation}\label{eq:splitsumsecond}
	\sum_{m=m_0+1}^\infty x_m \le 10^8 \bc{\frac{\log (1/c_1)}{c_1}  }^4
\end{equation}
and substituting~\eqref{eq:splitsumfirst} and~\eqref{eq:splitsumsecond} into~\eqref{eq:splitsum} gives
$$
x \le 2\cdot c_0 \cdot  10^{-9} + 10^8 \bc{\frac{\log (1/c_1)}{c_1}  }^4.
$$
For the case $c_1 \geq 10^{-5}$, choose $c'_1$ such that $c'_1 \leq 10^{-5}$, then
$c_0 \exp(-c_1 t) \leq c_0 \exp(-c'_1 t)$. Thus, by considering  the pair $(c_0,c'_1)$ and the above reasoning we get 
$$ c_2 =  2\cdot c_0 \cdot  10^{-9} + 10^8 \bc{\frac{\log (1/c'_1)}{c'_1}  }^4.$$

\section*{Acknowledgments} 
We are grateful to Jane Gao for a helpful conversation at the beginning of this project that brought the two-peaked nature of the function $\Phi_d$ to our attention.

\bibliographystyle{amsplain}

%
%

\begin{aicauthors}
\begin{authorinfo}[ACO]
  Amin Coja-Oghlan\\
  TU Dortmund, Faculty of Computer Science, Chair~2,\\
  Otto Hahn St~12, 44227 Dortmund, Germany.\\
  amin.coja-oghlan\imageat{}tu-dortmund\imagedot{}de 
\end{authorinfo}
\begin{authorinfo}[OC]
 Oliver Cooley\\
  Ludwig-Maximilians-Universit\"at M\"unchen,\\
 Theresienstra{\ss}e 39, 80333 M\"unchen, Germany.\\
  cooley\imageat{}math\imagedot{}lmu\imagedot{}de 
\end{authorinfo}
\begin{authorinfo}[MK]
  Mihyun Kang\\
  Graz University of Technology, Institute of Discrete Mathematics,\\
  Steyrergasse 30, 8010 Graz, Austria.\\
  kang\imageat{}math\imagedot{}tugraz\imagedot{}at
\end{authorinfo}
\begin{authorinfo}[JL]
  Joon Lee\\
  EPFL IC IINFCOM LTHC,\\
   INR 138 (Batiment INR) Station 14, CH-1015 Lausanne, Switzerland. \\
  joonhyung.lee\imageat{}epfl\imagedot{}ch
\end{authorinfo}

\begin{authorinfo}[JRB]
	Jean Bernoulli Ravelomanana\\
	EPFL IC IINFCOM LTHC,\\
	INR 139 (Batiment INR) Station 14, CH-1015 Lausanne, Switzerland. \\
	jean.ravelomanana\imageat{}epfl\imagedot{}ch
\end{authorinfo}
\end{aicauthors}

\end{document}